\documentclass[11pt,twoside]{amsart}

\usepackage{amsmath,amsthm,amssymb,amsfonts,amscd}
\usepackage{url}
\usepackage{color}
\usepackage{setspace}
\usepackage{enumitem}

\usepackage{epsfig}
\usepackage{graphicx}
\usepackage{mathrsfs}
\usepackage{mathabx}
\usepackage{pinlabel}
\usepackage[T1]{fontenc}
\usepackage[outercaption]{sidecap}

\usepackage{tikz-cd}

\definecolor{light-gray}{gray}{0.60}

\newcounter{notes}%

\theoremstyle{plain}
\newtheorem{theorem}{Theorem}
\newtheorem{proposition}[theorem]{Proposition}
\newtheorem{corollary}[theorem]{Corollary}

\newtheorem{lemma}[theorem]{Lemma}
\newtheorem{fact}[theorem]{Fact}

\theoremstyle{definition}
\newtheorem{definition}[theorem]{Definition}

\newtheorem{remark}[theorem]{Remark}
\newtheorem{remarks}[theorem]{Remarks}

\newtheorem{notation}[theorem]{Notation}

\numberwithin{theorem}{section}
\numberwithin{equation}{section}

\newcommand{\D}{\mathrm{d}}

\newcommand{\NN}{\mathbb{N}}
\newcommand{\ZZ}{\mathbb{Z}}

\newcommand{\RR}{\mathbb{R}}
\newcommand{\CC}{\mathbb{C}}
\newcommand{\HH}{\mathbb{H}}
\newcommand{\PP}{\mathbb{P}}

\newcommand{\SL}{\mathrm{SL}}
\newcommand{\GL}{\mathrm{GL}}
\newcommand{\SO}{\mathrm{SO}}
\newcommand{\OO}{\mathrm{O}}

\newcommand{\PSL}{\mathrm{PSL}}
\newcommand{\PGL}{\mathrm{PGL}}

\newcommand{\g}{\mathfrak{g}}

\newcommand{\oo}{\mathfrak{o}}

\newcommand{\Ad}{\operatorname{Ad}}

\newcommand{\Hom}{\mathrm{Hom}}

\newcommand{\ie}{i.e.\ }
\newcommand{\eg}{e.g.\ }
\newcommand{\resp}{resp.\ }

\newcommand{\Kap}{C}
\newcommand{\kap}{c}
\newcommand{\col}{m}
\newcommand{\Hpqv}{\widecheck{\mathbb{H}}^{p,q} }
\newcommand{\Hpqvv}{\overset{\rotatebox{90}{\footnotesize\!\!<<}}{\mathbb{H}}{}^{p,q} }

\newcommand{\PF}{{\mathrm{\scriptscriptstyle PF}}}

\newcommand{\llangle}{\langle \! \langle}
\newcommand{\rrangle}{\rangle \! \rangle}

\title{Proper affine actions for right-angled Coxeter groups}

\author{Jeffrey Danciger}
\address{Department of Mathematics, The University of Texas at Austin, 1 University Station C1200, Austin, TX 78712, USA}
\email{jdanciger@math.utexas.edu}

\author{Fran\c{c}ois Gu\'eritaud}
\address{CNRS and Universit\'e Lille 1, Laboratoire Paul Painlev\'e, 59655 Villeneuve d'Ascq Cedex, France}
\email{francois.gueritaud@math.univ-lille1.fr}

\author{Fanny Kassel}
\address{CNRS and Institut des Hautes \'Etudes Scientifiques, 35 route de Chartres, 91440 Bures-sur-Yvette, France}
\email{kassel@ihes.fr}

\thanks{J.D. was partially supported by an Alfred P. Sloan Foundation fellowship, and by the National Science Foundation under grants DMS~1510254, and DMS~1812216.
F.G. and F.K. were partially supported by the Agence Nationale de la Recherche under grants DiscGroup (ANR-11-BS01-013) and DynGeo (ANR-16-CE40-0025-01), and through the Labex CEMPI (ANR-11-LABX-0007-01).
This project received funding from the European Research Council (ERC) under the European Union's Horizon 2020 research and innovation programme (ERC starting grant DiGGeS, grant agreement No 715982).
The authors also acknowledge support from the GEAR Network, funded by the National Science Foundation under grant numbers DMS 1107452, 1107263, and 1107367 (``RNMS: GEometric structures And Representation varieties").
Part of this work was completed while J.D. and F.K. were in residence at the MSRI in Berkeley, California, for the program \emph{Dynamics on Moduli Spaces of Geometric Structures} (Spring 2015) supported by NSF grant DMS~1440140}

\begin{document}

\maketitle

\begin{abstract}
For any right-angled Coxeter group $\Gamma$ on $k$ generators, we construct proper actions of~$\Gamma$ on $\OO(p,q+1)$ by right-and-left multiplication, and on the Lie algebra $\oo(p,q+1)$ by affine transformations, for some $p,q\in\NN$ with $p+q+1=k$.
As a consequence, any virtually special group admits proper affine actions on some~$\RR^n$: this includes \eg surface groups, hyperbolic 3-manifold groups, examples of word hyperbolic groups of arbitrarily large virtual cohomological dimension, etc. 
We also study some examples in cohomological dimension two and four, for which the dimension of the affine space may be substantially reduced.
\end{abstract}

\section{Introduction} \label{sec:intro}

Tiling space with regular shapes is an old endeavor, both practical and ornamental.
It is also at the heart of crystallography, and Hilbert, prompted by recent progress in that discipline, asked in his 18th problem for a better understanding of regular tilings of Euclidean space $\RR^n$. 
In 1910, Bieberbach \cite{bie11-12} gave a partial answer by showing that a discrete group~$\Gamma$ acting properly by affine Euclidean isometries on~$\RR^n$ has a finite-index subgroup acting as a lattice of translations on some affine subspace $\mathbb{R}^m$.
Moreover, $m=n$ if and only if the quotient $\Gamma\backslash \RR^n$ is compact, and the number $\mathscr{N}_n$ of such cocompact examples $\Gamma$ up to affine conjugation is finite for fixed $n$.
Crystallographers had known since 1891 that $\mathscr{N}_2=17$ and $\mathscr{N}_3=219$ (or $230$ if chiral meshes are counted twice), a result due independently to Schoenflies and Fedorov.

The picture for affine actions becomes much less familiar in the absence of an invariant Euclidean metric.
The Auslander conjecture \cite{aus64} states that if $\Gamma$ acts properly discontinuously and cocompactly on~$\RR^n$ by affine transformations, then $\Gamma$ should be virtually (\ie up to finite index) solvable, or equivalently \cite{mil77}, virtually polycyclic.
This conjecture has been proved up to dimension six \cite{fg83,tom16,ams12} and in certain special cases \cite{gk83,tom90,ams11}, but remains wide open in general.

In 1983, Margulis \cite{mar83, mar84} constructed the first examples of proper actions of \emph{nonabelian free groups} $\Gamma$ on~$\RR^3$, answering a question of Milnor~\cite{mil77}.
These actions do not violate the Auslander conjecture as they are not cocompact.
They preserve a flat Lorentzian structure on~$\RR^3$ and the corresponding affine $3$-manifolds are now known as \emph{Margulis spacetimes}.
Drumm~\cite{dru92} constructed more examples of Margulis spacetimes by building explicit fundamental domains in~$\RR^3$ bounded by polyhedral surfaces called \emph{crooked planes}; it is now known \cite{cdg16, dgk-strips, dgk-parab} that all Margulis spacetimes are obtained in this way.
Full properness criteria for affine actions by free groups on $\RR^3$ were given by Goldman--Labourie--Margulis~\cite{glm09} and subsequently by the authors~\cite{dgk16, dgk-strips, dgk-parab}. 
In higher dimensions, Abels--Margulis--Soifer \cite{ams97,ams02} have studied proper affine actions by free groups whose linear part is Zariski-dense in an indefinite orthogonal group, showing that such actions exist if and only if the signature is, up to sign, of the form $(2m,2m-1)$ with $m\geq 1$.
In \cite{smi16}, Smilga generalized Margulis's construction and showed that for any noncompact real semisimple Lie group~$G$ there exist proper actions, on the Lie algebra $\g\simeq\RR^{\dim(G)}$, of nonabelian free discrete subgroups of $G\ltimes\g$ acting affinely via the adjoint action, with Zariski-dense linear part; Margulis spacetimes correspond to $G=\PSL(2,\RR)\simeq\SO(2,1)_0$.
More recently, Smilga gave a sufficient condition \cite{smi16bis} (also conjectured to be necessary), given a real semisimple Lie group $G$ and a linear representation $V$ of~$G$, for the semidirect product $G\ltimes V$ to admit a nonabelian free discrete subgroup acting properly on~$V$ and whose linear part is Zariski-dense in~$G$.

\subsection{New examples of proper affine actions}

The existence of proper affine actions by nonabelian free groups suggests the possibility that other finitely generated groups which are not virtually solvable might also admit proper affine actions.
However, in the more than thirty years since Margulis's discovery, very few examples have appeared.
In particular, until now, all known examples of word hyperbolic groups acting properly by affine transformations on~$\RR^n$ were virtually free groups. In this paper, we give many new examples, both word hyperbolic and not, by establishing the following.

\begin{theorem} \label{thm:main}
Any right-angled Coxeter group on $k$ generators admits proper affine actions on $\RR^{k(k-1)/2}$.
\end{theorem}

Right-angled Coxeter groups, while simple to describe in terms of generators and relations, have a rich structure and contain many interesting subgroups.
For example, the fundamental group of any closed orientable surface of negative Euler characteristic embeds as a finite-index subgroup in the right-angled pentagon group.
Further, since any right-angled Artin group embeds into a right-angled Coxeter group \cite{dj00}, we obtain the following answer to a question of Wise \cite[Problem\,13.47]{wis14}.

\begin{corollary} \label{cor:RAAG}
Any right-angled Artin group admits proper affine actions on~$\RR^n$ for some $n\geq 1$.
\end{corollary}

See \eg \cite{bb97} for interesting subgroups of right-angled Artin groups, for which Corollary~\ref{cor:RAAG} provides proper affine actions.
Note that, in general, if a group $\Gamma$ admits a subgroup $\Gamma'$ of index~$m$ with a proper affine action on~$\RR^n$, then the induced action of $\Gamma$ on $(\RR^n)^{\Gamma/\Gamma'}\simeq \RR^{mn}$ is itself affine and proper.
Haglund--Wise \cite{hw08} proved that the fundamental group of any \emph{special} nonpositively curved cube complex embeds into a right-angled Artin group, and so we obtain the following.

\begin{corollary} \label{cor:virt-special}
Any virtually special group admits a proper affine action on~$\RR^n$ for some $n\geq 1$.
\end{corollary}

Virtually special groups include:
\begin{itemize}
  \item all Coxeter groups (not necessarily right-angled) \cite{hw10};
  \item all cubulated word hyperbolic groups, using Agol's virtual specialness theorem~\cite{ago13};
  \item therefore, all fundamental groups of closed hyperbolic $3$-manifolds, using \cite{sag95,km12}: see \cite{bw12};
  \item the fundamental groups of many other $3$-manifolds, see \cite{wis11,liu13,pw18}.
\end{itemize}

Januszkiewicz--\'Swi\c{a}tkowski \cite{js03} found word hyperbolic right-angled Coxeter groups of arbitrarily large virtual cohomological dimension; see also \cite{osa13} for another construction.
Hence another consequence of Theorem~\ref{thm:main} is:

\begin{corollary} \label{cor:vcd}
There exist proper affine actions by word hyperbolic groups of arbitrarily large virtual cohomological dimension.
\end{corollary}

The Auslander conjecture is equivalent to the statement that a group acting properly discontinuously by affine transformations on~$\RR^n$ is either virtually solvable, or has virtual cohomological dimension $<n$.
In the examples from Theorem~\ref{thm:main}, the dimension $n = k(k-1)/2$ of the affine space grows quadratically in the number of generators $k$, while the virtual cohomological dimension of the Coxeter group acting is naively bounded above by~$k$ (and is even much smaller in the examples above \cite{js03, osa13}).
Hence, Theorem~\ref{thm:main} is far from giving counterexamples to the Auslander Conjecture.

\subsection{An outline: properness from contraction properties} \label{subsec:outline}

In order to describe our approach to proving Theorem~\ref{thm:main}, start with a Lie group $G$ acting by isometries on a complete metric space $\mathbb{X}$.
Consider a discrete group $\Gamma$ and a representation $(\rho,\rho'):\Gamma\rightarrow G\times G$ such that $\Gamma$ acts properly discontinuously on $\mathbb{X}$ via~$\rho$.
The action of $\Gamma$ on~$G$ by right-and-left multiplication via $(\rho,\rho')$, given by 
\begin{equation} \label{eqn:aff-group-action}
\gamma\bullet g := \rho'(\gamma)g\rho(\gamma)^{-1},
\end{equation}
is not necessarily properly discontinuous: for instance, if $\rho'=h\rho(\cdot)h^{-1}$ for some~$h$, then $h\in G$ is a global fixed point.
On the other hand, in some cases the action~\eqref{eqn:aff-group-action} may be shown to be properly discontinuous by exhibiting a map  $f : \mathbb{X}\to \mathbb{X}$ which is \emph{uniformly contracting} (\ie with Lipschitz cons\-tant $<1$) and $(\rho,\rho')$-equivariant: $f\circ \rho(\gamma)=\rho'(\gamma)\circ f$ for all $\gamma \in \Gamma$.
The basic idea is that the $(\rho, \rho')$-action of $\Gamma$ on $G$ projects equivariantly down to the $\rho$-action on $\mathbb X$ via the \emph{fixed point map} $g \mapsto \mathrm{Fix}(g^{-1}\circ f)$, which is well defined by the contraction property.
Proper discontinuity of the action on the base~$\mathbb X$ then implies proper discontinuity upstairs on~$G$: see Section~\ref{subsec:metric-contract-G}.

This principle suggests an infinitesimal analogue, as in \cite{dgk16}, for $(\rho, \rho')$ very close to the diagonal of $G\times G$ and $f$ close to $\mathrm{Id}_\mathbb{X}$.
Let $\g = T_e G$ be the Lie algebra of~$G$, with $G$ acting on~$\g$ via the adjoint representation.
Given a representation $(\rho,u) : \Gamma\to G\ltimes\g$, the affine action of $\Gamma$ on~$\g$ given by
\begin{equation}\label{eqn:lie-algebra-action}
\gamma \bullet v := \Ad(\rho(\gamma))v + u(\gamma)
\end{equation}
is in many cases obtained from~\eqref{eqn:aff-group-action} by a limiting and rescaling process, thinking of $G \ltimes \g \cong TG$ as the normal bundle to the diagonal in $G \times G$.
Such an affine action on~$\g$ will be properly discontinuous (because $\g$ will project onto~$\mathbb{X}$ in an equivariant way similar to the above) if we can build a \emph{uniformly contracting vector field} on $\mathbb{X}$ satisfying an appropriate $(\rho,u)$-equivariance property (Sections \ref{subsec:def-inf-metric-contract}--\ref{subsec:metric-contract-g}).

The affine actions we construct for Theorem~\ref{thm:main} will all be of the form~\eqref{eqn:lie-algebra-action} for $G = \OO(p,q+1)$ an indefinite orthogonal group.
Indeed, a right-angled Coxeter group $\Gamma$ on $k$ generators (say, infinite and irreducible) admits explicit families of discrete embeddings $\rho : \Gamma\to\OO(p,q+1)$ as a \emph{reflection group}, with $p+q+1=k$, which have long been studied by Tits, Vinberg, and others (see Section~\ref{subsec:remind-Coxeter}).

The above strategy of ensuring properness from contraction works well when $q = 0$: in this case we take $\mathbb X$ to be the Riemannian symmetric space of $G = \OO(p,1)$, namely the real hyperbolic space $\HH^p$.
For representations $\rho,\rho' : \Gamma\to G$ as reflection groups as above, the action of $\Gamma$ on $\HH^p$ via~$\rho$ is by reflections in the walls of a polytope $P_\rho$ of $\HH^p$, and similarly for~$\rho'$. 
Natural $(\rho,\rho')$-equivariant maps $f$ are constructed by taking $P_\rho$ projectively to $P_{\rho'}$, wall to wall, and extending equivariantly by the reflections.
The map $f$ will turn out to be uniformly contracting as soon as, roughly speaking, $P_{\rho'}$ is obtained by pushing the walls of $P_\rho$ closer together. See Section~\ref{sec:metric-contraction}, and the examples of Section~\ref{sec:riem-ex}.

For the general case of Theorem~\ref{thm:main}, we cannot consider only $q=0$, as most right-angled Coxeter groups do not embed in $\OO(p,1)$.
In the case $q > 0$, the orthogonal group $G = \OO(p,q+1)$ has higher real rank, and we are faced with a dilemma.
On the one hand, the contraction strategy, as described above, starts with a metric $G$-space $\mathbb X$, such as the Riemannian symmetric space $\mathbb X_G$ of~$G$.
On the other hand, a perhaps more natural space in which to see the geometry of the Tits--Vinberg representations $\rho: \Gamma \to G=\OO(p,q+1)$ is in a \emph{pseudo-Riemannian symmetric space}, namely the pseudo-Riemannian analogue $\HH^{p,q}\subset\PP(\RR^k)$ of $\HH^p$ in signature $(p,q)$.
Indeed:
\begin{enumerate}[label=(\roman*),leftmargin=1cm]
  \item The generators of~$\Gamma$ naturally act via~$\rho$ by reflections in the walls of a polytope $P_\rho$ of $\HH^{p,q}$, depending continuously on~$\rho$.
  \item We may build natural $(\rho,\rho')$-equivariant maps $f$ by taking $P_\rho$ projectively to~$P_{\rho'}$, walls to walls, just as in the $\HH^p$ case.
  \item Since the ``distances'' in $\HH^{p,q}$ are computed by a simple cross-ratio formula, similar to $\HH^p$ in the projective model, the ``contraction'' properties of our map $f$ (for some suitable notion of contraction) are easy to check locally in the fundamental domain~$P_\rho$.
\end{enumerate}
By contrast, for the action on the Riemannian symmetric space $\mathbb X_G$, there are no obvious choices of fundamental domains~(i) or of explicit equivariant maps~(ii) with which to work. Further, any $G$-invariant metric on $\mathbb{X}_G$ involves the singular values of $k\times k$ matrices, making local contraction~(iii) potentially difficult to check.
Hence, we abandon the Riemannian symmetric space~$\mathbb X_G$, and prove Theorem~\ref{thm:main} by employing a version of the contraction strategy outlined above, adjusted and reinterpreted appropriately to work in the pseudo-Riemannian space $\HH^{p,q}$ (see Section~\ref{sec:pseudo-Riem-contraction}).
Despite the obvious hurdle that $\HH^{p,q}$ is not a metric space, enough structure survives to apply our approach: a key step will be to check that $\rho(\Gamma)$-orbits in $\HH^{p,q}$ escape mostly in \emph{spacelike} directions, in which their growth resembles that of actions on~$\HH^p$ (Lemma~\ref{lem:spacelike-Gamma-orbit}).

In the remainder of this introduction, we give a more precise account of our parallel results concerning actions on $G$ and on~$\g$, starting with the case of~$\g$ yielding Theorem~\ref{thm:main}.

\subsection{Proper actions on Lie algebras} \label{subsec:propli}

Let $G$ be a Lie group, acting on its Lie algebra~$\g$ via the adjoint action, and let $\Gamma$ be a discrete group.
We consider affine actions of $\Gamma$ on $\g$ determined, as in~\eqref{eqn:lie-algebra-action}, by a representation $(\rho,u): \Gamma \to G \ltimes \g \simeq TG$ where $\rho : \Gamma\to G$ is a group homomorphism and $u : \Gamma\to\g$ a \emph{$\rho$-cocycle}, \ie a map satisfying
\begin{equation} \label{eqn:cocycle}
u(\gamma_1\gamma_2) = u(\gamma_1) + \Ad(\rho(\gamma_1))\,u(\gamma_2)
\end{equation}
for all $\gamma_1,\gamma_2\in\Gamma$.
For instance, for any smooth path $(\rho_t)_{t \in I}$ in $\Hom(\Gamma,G)$ (where $I$ is an open interval) and any $t\in I$, the map $u_t : \Gamma\to\g$ given by $u_t(\gamma) = \frac{\D}{\D\tau}\big |_{\tau=t}\,\rho_{\tau}(\gamma) \rho_t(\gamma)^{-1}$ is a $\rho_t$-cocycle; it is the unique $\rho_t$-cocycle such that for all $\gamma\in\Gamma$,
\begin{equation} \label{eqn:integrate}
\rho_{s}(\gamma) = \mathrm{e}^{(\tau-t)u_t(\gamma)+o(\tau-t)}\rho_t(\gamma) \:\: \text{ as } \tau\to t.
\end{equation}
The cocycles in this paper will all be constructed in this way.
(In general there may exist cocycles which are not integrable, \ie not tangent to any deformation path: see \cite[\S 2]{lm85}.)
We prove the following.

\begin{theorem} \label{thm:proper-action-g}
For any irreducible right-angled Coxeter group~$\Gamma$ on $k$ generators, there exist $p,q\in\NN$ with $p+q+1=k$ and a smooth path $(\rho_t)_{t\in I}$ in $\Hom(\Gamma,G)$ of faithful and discrete representations into $G:=\OO(p,q+1)$ (where $I\neq\varnothing$ is an open interval) such that for any $t\in I$, the affine action of $\Gamma$ on $\g \simeq \RR^{k(k-1)/2}$ via $(\rho_t,\frac{\D}{\D s}\big |_{s=t}\,\rho_{s}\rho_t^{-1})$ is properly discontinuous.
\end{theorem}

Since any right-angled Coxeter group is a direct product of irreducible ones, we obtain Theorem~\ref{thm:main} by applying Theorem~\ref{thm:proper-action-g} to each irreducible factor and then taking the direct sum of the resulting affine actions.

We also use similar techniques as in Theorem~\ref{thm:proper-action-g} to construct, in some specific cases, examples of proper affine actions on $\g=\oo(p,q+1)$ where $p+q+1$ is smaller than the number $k$ of generators of~$\Gamma$.

\begin{proposition} \label{prop:ex-affine}
(a) For any even $k\geq 6$, the group $\Gamma$ generated by reflections in the sides of a convex right-angled $k$-gon of $\HH^2$ admits proper affine actions on $\g=\oo(3,1)\simeq\RR^6$.

(b) The group $\Gamma$ generated by reflections in the faces of a $4$-dimensional regular right-angled $120$-cell admits proper affine actions on $\g=\oo(8,1)\simeq\nolinebreak\RR^{36}$. 
\end{proposition}

The group $\Gamma$ is virtually the fundamental group of a closed surface of genus $\geq 2$ in (a), and of a closed hyperbolic 4-manifold in (b).
Both examples follow from a general face-coloring method explained in Proposition~\ref{prop:color}.
The baby case of this method (involving a single color) also gives a direct way to construct Margulis spacetimes, see Remark~\ref{rem:colors}.(4).

Whereas the examples of proper affine actions by free groups of Margulis, Drumm, Abels--Margulis--Soifer, and Smilga all relied to some degree on the idea of free groups playing ping pong on~$\RR^n$, 
for Theorem~\ref{thm:proper-action-g} and Proposition~\ref{prop:ex-affine} we rather use a sufficient condition for properness based on the idea of contraction explained in Section~\ref{subsec:outline} (see Propositions~\ref{prop:coarse-proj-g} and~\ref{prop:coarse-proj-g-pq}).
This condition generalizes a properness criterion from \cite{dgk16, dgk-strips, dgk-parab} for actions on $\oo(2,1)\simeq\mathfrak{psl}(2,\RR)$ in terms of uniformly contracting vector fields on~$\HH^2$.
From this properness criterion, the topology of the quotient manifolds (\emph{Margulis spacetimes}) may be read off directly. Similarly here, the properness of the affine actions of Proposition~\ref{prop:ex-affine} will be derived from uniformly contracting vector fields on~$\HH^p$ (for $p = 3, 8$), and again the topology of the quotient manifolds will be clear from our methods (see Remark~\ref{rem:topology}). 
However, the contraction arguments in $\HH^{p,q}$ for the general case of Theorem~\ref{thm:proper-action-g} are too coarse to control the topology of the quotients.

We note that the affine actions in Theorem~\ref{thm:proper-action-g} preserve a nondegenerate symmetric bilinear form on~$\g$, namely the Killing form, of signature given by~\eqref{eqn:sign-Killing} below.
This induces a flat pseudo-Riemannian metric on the quotient manifolds.

\subsection{Proper actions on Lie groups} \label{subsec:propLi}

Following \cite{dgk16,dgk-strips}, and in the spirit of Section~\ref{subsec:outline}, we view the proper affine actions on the Lie algebra~$\g$ in Theorem~\ref{thm:proper-action-g} as ``infinitesimal analogues'' of proper actions on the corresponding Lie group~$G$ for the action of $G\times G$ by right-and-left multiplication~\eqref{eqn:aff-group-action}.
We prove the following ``macroscopic version'' of Theorem~\ref{thm:proper-action-g}.

\begin{theorem} \label{thm:proper-action-G}
For any irreducible right-angled Coxeter group~$\Gamma$ on $k$ generators, there exist $p,q\in\NN$ with $p+q+1=k$ and a smooth path $(\rho_t)_{t\in I}$ in $\Hom(\Gamma,G)$ of faithful and discrete representations into $G:=\OO(p,q+1)$ (where $I\neq\varnothing$ is an open interval) such that for any $t\neq s$ in~$I$, the action of $\Gamma$ on~$G$ by right-and-left multiplication via $(\rho_t,\rho_{s})$ is properly discontinuous.
\end{theorem}

In general, it is easy to obtain proper actions on~$G$ by right-and-left multiplication by considering a discrete group~$\Gamma$, a representation $\rho\in\Hom(\Gamma,G)$ with finite kernel and discrete image, and a representation $\rho'\in\Hom(\Gamma,G)$ with bounded image (for instance the constant representation, with image $\{e\}\subset G$): such proper actions are often called \emph{standard}.
The point of Theorem~\ref{thm:proper-action-G} is to build nonstandard proper actions on~$G$, where both factors are faithful and discrete --- and in fact, can be arbitrarily close to each other.

We also construct examples of proper actions on $G=\OO(p,q+1)$ where $p+q+1$ is smaller than the number $k$ of generators of~$\Gamma$, in the same cases as for Proposition~\ref{prop:ex-affine}.

\begin{proposition} \label{prop:ex-G}
(a) For any even $k\geq 6$, the group $\Gamma$ generated by reflections in the sides of a convex right-angled $k$-gon in $\HH^2$ admits proper actions on $G=\OO(3,1)$ by right-and-left multiplication via pairs $(\rho,\rho')\in\Hom(\Gamma,G)^2$ with $\rho,\rho'$ both faithful and discrete.

(b) The group $\Gamma$ generated by reflections in the faces of a $4$-dimensional regular right-angled $120$-cell admits proper actions on $G=\OO(8,1)$ by right-and-left multiplication via pairs $(\rho,\rho')\in\Hom(\Gamma,G)^2$ with $\rho,\rho'$ both faithful and discrete.
\end{proposition}

Full properness criteria for proper actions on $\OO(n,1)$ via $(\rho,\rho')$ with $\rho$ geometrically finite were provided in \cite{kasPhD,gk17} in terms of uniform contraction in~$\HH^n$ (see 
Remark~\ref{rem:converse-Riem}).

\begin{remark}
For $p\geq 1$, the group $G=\OO(p,q+1)$ has four connected components.
The proper actions on~$G$ constructed in Theorem~\ref{thm:proper-action-G} and Proposition~\ref{prop:ex-G} all yield proper actions on the identity component~$G_0$.
\end{remark}

For $p=2$ and $q=0$, the identity component $G_0 = \OO(2,1)_0$ is the so-called \emph{anti-de Sitter} $3$-space $\mathrm{AdS}^3$, a Lorentzian analogue of~$\HH^3$.
The group of orientation-preserving isometries of $\mathrm{AdS}^3$ identifies with the quotient of the four diagonal components of $G\times G$ by $\{(\mathrm{Id}, \mathrm{Id}),(-\mathrm{Id},-\mathrm{Id})\}$, acting on~$G_0$ by right-and-left multiplication.
Many examples of proper actions on $\mathrm{AdS}^3$ were constructed since the 1980s, see in particular \cite{kr85,sal00,kasPhD,gk17,gkw15,dt16}.

Examples of nonstandard cocompact proper actions on $\OO(n,1)$ by right-and-left multiplication for $n>2$ were constructed in \cite{ghy95,kob98}, using deformation techniques.
After we announced the results of this paper, Lakeland--Leininger \cite{ll17} found examples of nonstandard cocompact proper actions on $\OO(3,1)$ and $\OO(4,1)$ by right-angled Coxeter groups which cannot be obtained from standard proper actions by deformation.
Note that for cocompact proper actions on $\OO(n,1)$ by right-and-left multiplication via $(\rho,\rho')$, exactly one of $\rho$ or~$\rho'$ has finite kernel and discrete image \cite{kas08,gk17}.
On the other hand, in the noncompact proper actions that we construct in Theorem~\ref{thm:proper-action-G} and Proposition~\ref{prop:ex-G}, both $\rho,\rho'$ have finite kernel and discrete image.

\subsection{Plan of the paper}

In Section~\ref{sec:prelim} we recall some background on pseudo-Riemannian hyperbolic spaces $\HH^{p,q}$.
In Section~\ref{sec:metric-contraction} we state and prove some sufficient criteria for properness, expressed in terms of uniform contraction in metric spaces (Propositions~\ref{prop:coarse-proj-G} and~\ref{prop:coarse-proj-g}).
In Section~\ref{sec:riem-ex} we give examples in $\HH^3$ and~$\HH^8$, establishing Propositions \ref{prop:ex-affine} and~\ref{prop:ex-G}.
In Section~\ref{sec:pseudo-Riem-contraction} we 
state and prove analogous criteria for general $\HH^{p,q}$ (Theorem~\ref{thm:contract-proper}).
In Section~\ref{sec:space-contr-Coxeter} we prove Theorems~\ref{thm:proper-action-g} (hence also~\ref{thm:main}) and~\ref{thm:proper-action-G}  by constructing appropriate families of representations $(\rho_t)_{t\in I}$ to which Theorem~\ref{thm:contract-proper} applies.

\subsection*{Acknowledgements}

We would like to thank Ian Agol, Yves Benoist, Suhyoung Choi, Bill Goldman, Gye-Seon Lee, Ludovic Marquis, Vivien Ripoll, and Anna Wienhard for interesting discussions.
We also thank Piotr\linebreak Przytycki for providing several references, Dani Wise for pointing out \cite[Problem\,13.47]{wis14} and encouraging us, and both referees for making many constructive remarks.
We gratefully acknowledge the hospitality and excellent working conditions provided by the MSRI in Berkeley, where most of this work was done in the Spring 2015, the CNRS-Pauli Institute in Vienna, and the IHES in Bures-sur-Yvette.

\section{Notation and reminders}\label{sec:prelim}

In this section we set up some notation and recall a few definitions and basic facts on properly convex domains in projective space, on the pseudo-Riemannian hyperbolic spaces $\HH^{p,q}$, and on eigenvalues and principal values of elements of $\GL(\RR^{p+q+1})$.

\subsection{Properly convex domains in projective space} \label{subsec:prop-conv}

Let $V$ be a real vector space of dimension $\geq 2$.
Recall that an open subset $\Omega$ of $\PP(V)$ is said to be \emph{properly convex} if it is convex and bounded in some affine chart of $\PP(V)$.
There is a natural metric $d_{\Omega}$ on~$\Omega$, the \emph{Hilbert metric}:
$$d_{\Omega}(x,y) := \frac{1}{2} \log \, [a,x,y,b]$$
for all distinct $x,y\in\Omega$, where $[\cdot,\cdot,\cdot,\cdot]$ is 
the $\PP^1(\RR)$-valued cross-ratio on a projective line, normalized so that $[0,1,t,\infty]=t$, and where $a,b$ are the intersection points of $\partial\Omega$ with the projective line through $x$ and~$y$, with $a,x,y,b$ in this order.
The metric space $(\Omega,d_{\Omega})$ is proper (\ie closed balls are compact) and complete.

The group
$\mathrm{Aut}(\Omega) := \{g\in\PGL(V) ~|~ g\cdot\Omega=\Omega\}$
acts on~$\Omega$ by isometries for~$d_{\Omega}$.
As a consequence, any discrete subgroup of $\mathrm{Aut}(\Omega)$ acts properly discontinuously on~$\Omega$.

Let $\widetilde{\Omega}\subset \RR^k$ be a convex open cone lifting~$\Omega$.
There is a unique lift of $\mathrm{Aut}(\Omega)$ to $\SL^\pm(V)$ that preserves $\widetilde{\Omega}$.
The \emph{dual convex cone} of~$\widetilde{\Omega}$ is by definition
$$\widetilde{\Omega}^* := \big\{ \varphi \in V^* ~|~ \varphi(v)<0\quad \forall v\in\overline{\widetilde{\Omega}}\big\},$$
where $\overline{\widetilde{\Omega}}$ is the closure of $\widetilde{\Omega}$ in $V\smallsetminus \{0\}$.
The image $\Omega^*$ of $\widetilde{\Omega}^*$ in $\PP(V^*)$ does not depend on the chosen lift $\widetilde{\Omega}$; it is a nonempty properly convex open subset of $\PP(V^*)$, called the \emph{dual convex set} of~$\Omega$, and preserved by the dual action of $\mathrm{Aut}(\Omega)$ on $\PP(V^*)$.
We can use any nondegenerate symmetric bilinear form $\langle\cdot,\cdot\rangle$ on~$V$ to view $\widetilde{\Omega}^*$ and $\Omega^*$ as subsets of $V$ and $\PP(V)$ respectively:
\begin{equation} \label{eqn:dual-in-P(V)}
\widetilde{\Omega}^* \simeq \big \{ w\in V ~|~ \langle w,v\rangle<0 \quad\forall v\in\overline{\widetilde{\Omega}}\big \} \quad\mathrm{and}\quad \Omega^*:=\PP(\widetilde{\Omega}^*).
\end{equation}

\begin{remark} \label{rem:Hilb-metric-include}
It follows from the definition that if $\Omega'\subset \Omega$ are nonempty properly convex open subsets of $\PP(\RR^{n+1})$, then the corresponding Hilbert metrics satisfy $d_{\Omega'}(x,y) \geq d_{\Omega}(x,y)$ for all $x,y\in\Omega'$.
When $\Omega$ is an ellipsoid, $(\Omega, d_\Omega)$ is isometric to the real hyperbolic space.
\end{remark}

\subsection{The pseudo-Riemannian space $\HH^{p,q}$} \label{subsec:prelim-Hpq}

For $p,q\in\NN$ with $p\geq 1$, let $\RR^{p,q+1}$ be $\RR^{p+q+1}$ endowed with a symmetric bilinear form $\langle \cdot,\cdot\rangle_{p,q+1}$ of signature $(p,q+\nolinebreak 1)$.
We set
$$\HH^{p,q} := \left \{ [v]\in\PP(\RR^{p,q+1}) ~\middle |~ \langle v,v\rangle_{p,q+1} < 0 \right \} .$$
The form $\langle \cdot ,\cdot \rangle_{p,q+1}$ induces a pseudo-Riemannian metric $\mathsf{g}^{p,q}$ of signature $(p,q)$ on $\HH^{p,q}$.
Explicitly, the metric $\mathsf{g}^{p,q}$ at a point $[v]$ is obtained from the restriction of $\langle \cdot,\cdot\rangle_{p,q+1}$ to the tangent space at $v/\sqrt{-\langle v,v\rangle_{p,q+1}}$ of the hypersurface
$$\widehat{\HH}^{p,q} := \big\{ v\in\RR^{p,q+1} ~|~ \langle v,v\rangle_{p,q+1} = -1\big\} ,$$
a double cover of $\HH^{p,q}$ with covering group $\{\mathrm{Id}, -\mathrm{Id}\}$.
The sectional curvature of $\mathsf{g}^{p,q}$ is constant negative, hence $\HH^{p,q}$ can be thought of as a pseudo-Rieman\-nian analogue of the real hyperbolic space $\HH^p = \HH^{p,0}$ in signature $(p,q)$.

The isometry group of the pseudo-Riemannian space $\HH^{p,q}$ is $\mathrm{PO}(p,q+\nolinebreak 1)=\OO(p,q+1)/\{\mathrm{Id}, -\mathrm{Id}\}$.
The point stabilizers are conjugate to $\OO(p,q)$, hence $\HH^{p,q} \simeq \mathrm{PO}(p,q+1)/\OO(p,q)$.

The set $\HH^p = \HH^{p,0}$ is a properly convex open subset of $\PP(\RR^{p,1})$, and the Hilbert metric $d_{\HH^p}$ on~$\HH^p$ coincides with the standard hyperbolic metric.
On the other hand, for $q\geq 1$ the space $\HH^{p,q}$ is \emph{not} convex in $\PP(\RR^{p,q+1})$.
The boundary of $\HH^{p,q}$ in $\PP(\RR^{p,q+1})$, given by
$$\partial\HH^{p,q} = \big\{ [v]\in\PP(\RR^{p,q+1}) ~|~ \langle v,v\rangle_{p,q+1} = 0 \big\},$$
is a quadric which, in any Euclidean chart of $\PP(\RR^{p,q+1})$, has $p-1$ positive and $q$ negative principal curvature directions at each point.

Consider $x\in\HH^{p,q}$ lifting to $\widehat{x}\in\widehat{\HH}^{p,q}$.
A nonzero vector $V_x\in T_x\HH^{p,q}$ and the geodesic line $\mathcal{L}$ it generates are called \emph{spacelike} (\resp \emph{lightlike}, \resp \emph{timelike}) if $\mathsf{g}^{p,q}_x(V_x,V_x)$ is positive (\resp zero, \resp negative).
The line $\mathcal{L}$ is then the intersection of $\HH^{p,q}$ with a projective line meeting $\partial\HH^{p,q}$ in two (\resp one, \resp zero) points: see Figure~\ref{fig:H123}.
For instance, if $V_x\in T_x\HH^{p,q} \simeq T_{\widehat{x}} \widehat{\HH}^{p,q} \simeq \widehat{x}^\perp\subset\RR^{p,q+1}$ satisfies $\langle V_x,V_x\rangle_{p,q+1}=1$, then $\mathcal{L}$ is spacelike, with unit-speed parametrization
\begin{equation} \label{eqn:param-geod}
t\longmapsto \exp_x(t V_x) = [\cosh (t) \, \widehat{x}+\sinh (t) \, V_x] \in \HH^{p,q}.
\end{equation}
In general, the totally geodesic subspaces of $\HH^{p,q}$ are exactly the intersections of $\HH^{p,q}$ with projective subspaces of $\PP(\RR^{p,q+1})$.
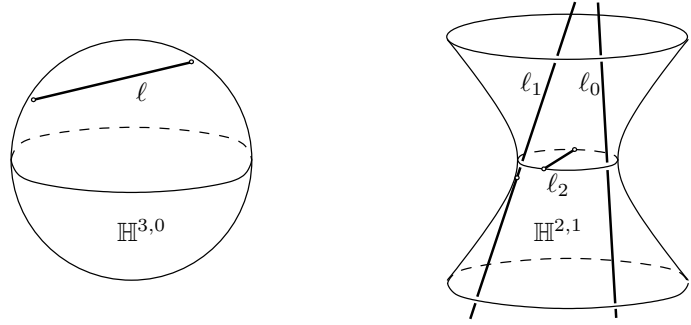
\begin{figure}[h!]
\centering
\labellist
\small\hair 2pt
\pinlabel $\HH^{3,0}$ [u] at 50 35
\pinlabel $\HH^{2,1}$ [u] at 207 35
\pinlabel $\ell$ [u] at 50 87
\pinlabel $\ell_2$ [u] at 207 52
\pinlabel $\ell_0$ [u] at 219 90 
\pinlabel $\ell_1$ [u] at 197 90
\endlabellist
\includegraphics[scale=1]{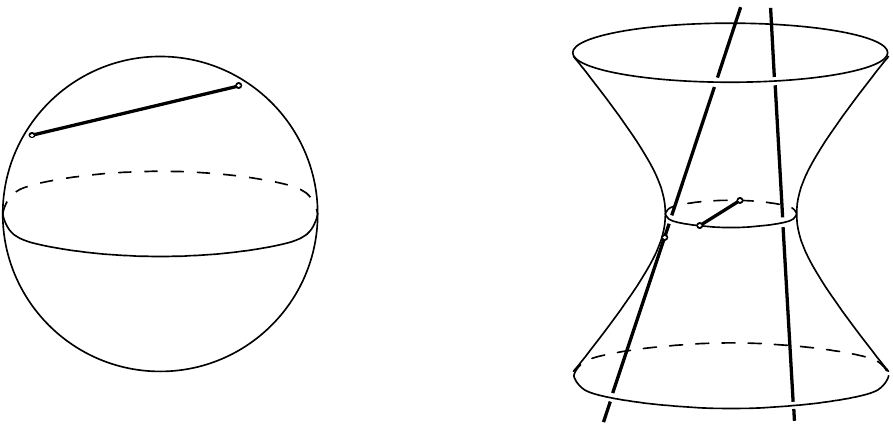}
\caption{Left: $\HH^3=\HH^{3,0}$ with a geodesic line $\ell$ (necessarily spacelike). Right: $\HH^{2,1}$ with three geodesic lines $\ell_2$ (spacelike), $\ell_1$ (lightlike), and $\ell_0$ (timelike).}
\label{fig:H123}
\end{figure}
As in \cite{gm16}, we shall use the following convention.

\begin{notation} \label{not:d-Hpq}
If $x,y\in\HH^{p,q}$ are distinct points belonging to a \emph{spacelike} line, we denote by $d_{\HH^{p,q}}(x,y) > 0$ the pseudo-Riemannian distance between $x$ and~$y$, obtained by integrating $\sqrt{\mathsf{g}^{p,q}}$ over the geodesic path from $x$ to~$y$.
If $x,y\in\HH^{p,q}$ are equal or belong to a lightlike or timelike line, we set $d_{\HH^{p,q}}(x,y) := 0$.
\end{notation}

Using~\eqref{eqn:param-geod}, we see that for any distinct points $x, y \in \HH^{p,q}$ lying on a spacelike line $\mathcal{L}$,
\begin{equation} \label{eqn:d-Hpq-inner-prod}
d_{\HH^{p,q}}(x,y) = \mathrm{arccosh}\,|\langle\widehat{x},\widehat{y}\rangle_{p,q+1}| > 0
\end{equation}
where $\widehat{x},\widehat{y}\in\widehat{\HH}^{p,q}$ are respective lifts of $x,y$.
The following Hilbert geometry interpretation, well-known in the $\HH^p$ setting, also holds in $\HH^{p,q}$ because $\mathcal{L}$ is a copy of $\HH^1$: normalizing the cross-ratio $[\cdot,\cdot,\cdot,\cdot]$ as in Section~\ref{subsec:prop-conv},
\begin{equation} \label{eqn:d-Hpq-cr}
d_{\HH^{p,q}}(x,y) = \frac{1}{2}\log\,[a,x,y,b] > 0
\end{equation}
where $a,b$ are the two intersection points of $\partial\HH^{p,q}$ with the projective line through $x$ and~$y$, with $a,x,y,b$ in this order.
The following (see Figure~\ref{fig:tang-vect}) is a pseudo-Riemannian analogue of the \emph{first variation formula} in Riemannian geometry.

\begin{proposition} \label{prop:deriv-dHpq}
For any $x,y\in\HH^{p,q}$ on a spacelike line and for any tangent vectors $Z_x\in T_x\HH^{p,q}$ and $Z_y\in T_y\HH^{p,q}$,
$$ \frac{\D}{\D t}\Big|_{t=0} \, d_{\HH^{p,q}} \big(\exp_x(tZ_x),\, \exp_y(tZ_y)\big )  = - \mathsf{g}^{p,q}_x \big (Z_x,V_x^y \big) - \mathsf{g}^{p,q}_y \big( Z_y,V_y^x \big),$$ 
where $V_z^{z'}\in T_z^{+1} \HH^{p,q}$ is the unit vector at $z$ pointing to~$z'$.
\end{proposition}

\begin{proof}
Let $\widehat{x}, \widehat{y} \in \widehat{\HH}^{p,q}$ be respective lifts of $x,y$ with $\langle\widehat{x},\widehat{x}\rangle_{p,q+1} = \langle\widehat{y},\widehat{y}\rangle_{p,q+1}\linebreak = -1$ and $\langle\widehat{x},\widehat{y}\rangle_{p,q+1} < 0$.
We view $Z_x$ and~$V_x^y$ as vectors of $\RR^{p,q+1}$ via the canonical identifications $T_x\HH^{p,q} \simeq T_{\widehat{x}}\widehat{\HH}^{p,q} \simeq \widehat{x}^{\perp} \subset \RR^{p,q+1}$, and similarly for $Z_y$ and~$V_y^x$.
The vectors $V_x^y$ and~$V_y^x$ are unit spacelike.
Let $\delta:=d_{\HH^{p,q}}(x,y)>\nolinebreak 0$.
By~\eqref{eqn:param-geod} we may write $\widehat{y} = \cosh(\delta)\,\widehat{x} + \sinh(\delta)\,V_x^y$ and $\widehat{x} = \cosh(\delta)\,\widehat{y} + \sinh(\delta)\,V_y^x$, and $\sinh(\delta)=\sqrt{\langle\widehat{x},\widehat{y}\rangle_{p,q+1}^2-1}$ by \eqref{eqn:d-Hpq-inner-prod}.

For $t\in\RR$, the point $\exp_x(t Z_x)\in\HH^{p,q}$ of \eqref{eqn:param-geod} lifts to the vector $\cosh t \widehat{x} + \sinh t Z_x$ in $\widehat{\HH}^{p,q} \subset \RR^{p,q+1}$ and similarly for $\exp_y(t Z_y)$.
Then \eqref{eqn:d-Hpq-inner-prod} yields
\begin{align*}
& \sinh(\delta)\frac{\D}{\D t}\Big|_{t=0} d_{\HH^{p,q}} \big (\exp_x(tZ_x),\, \exp_y(tZ_y)\big )\\ &= \frac{\D}{\D t}\Big|_{t=0} \cosh\big( d_{\HH^{p,q}} \big (\exp_x(tZ_x),\, \exp_y(tZ_y)\big ) \big)\\ 
& = \frac{\D}{\D t}\Big|_{t=0} -\langle \cosh(t) \ \widehat{x} + \sinh t Z_x, \cosh(t) \ \widehat{y} + \sinh(t) Z_y \rangle_{p,q+1}\\
& = -\langle Z_x, \widehat{y} \rangle - \langle \widehat{x}, Z_y \rangle_{p,q+1} \\
& =  -\langle Z_x, \cosh(\delta) \,\widehat{x} + \sinh (\delta) V_x^y \rangle_{p,q+1} - \langle \cosh(\delta) \, \widehat{y} + \sinh(\delta) \, V_y^x, Z_y \rangle_{p,q+1}\\
& = -\sinh (\delta) \langle Z_x,  V_x^y \rangle_{p,q+1} - \sinh (\delta) \langle V_y^x, Z_y \rangle_{p,q+1},
\end{align*}
and the result follows from the definition of the metric $\mathsf{g}^{p,q}$.
\end{proof}
\begin{figure}[h!]
\centering
\labellist
\small\hair 2pt
\pinlabel $x$ [u] at 0 -2
\pinlabel $Z_x$ [u] at 49 30
\pinlabel $V_x^y$ [u] at 26 -4
\pinlabel $y$ [u] at 199 -2
\pinlabel $Z_y$ [u] at 157 26
\pinlabel $V_y^x$ [u] at 172 -4
\pinlabel $\delta$ [u] at 100 -1
\endlabellist
\includegraphics[width=10cm]{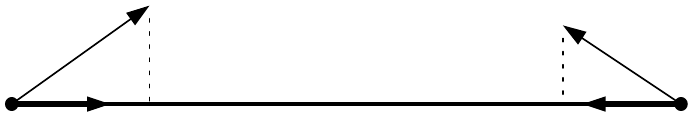}
\caption{Illustration of Proposition~\ref{prop:deriv-dHpq}}
\label{fig:tang-vect}
\end{figure}
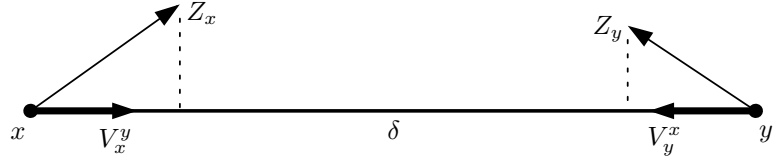

Note that when $q\geq 1$, the function $d_{\HH^{p,q}}$ is \emph{not} a distance function on $\HH^{p,q}$ in the usual sense: for many triples it does not satisfy the triangle inequality.
See \cite[\S\,3]{gm16} for further discussion of this issue.

\begin{remark} \label{rem:hadamard}
If $(\mathbb{X}, \mathsf{g})$ is a Hadamard manifold (\ie a simply connected finite-dimensional Riemannian manifold of nonpositive curvature), then for any $x,y\in \mathbb{X}$ and $Z_x\in T_x\mathbb{X}$ and $Z_y\in T_y\mathbb{X}$, one has as in Proposition~\ref{prop:deriv-dHpq}: 
$$ \frac{\D}{\D t}\Big|_{t=0} \, d_{\mathbb{X}} \big(\exp_x(tZ_x), \exp_y(tZ_y)\big )  = - \mathsf{g}_x (Z_x,V_x^y ) - \mathsf{g}_y ( Z_y,V_y^x )$$
where $V_z^{z'}\in T_z \mathbb{X}$ is the unit vector at $z$ pointing to~$z'$.
This follows from the first variation formula, and the absence of conjugate points \cite[Cor.~2.4 \& Th.~5.1]{kono}.
\end{remark}

\subsection{The Lie algebra $\oo(p,q+1)$} \label{subsec:o(p,q+1)}

Since $\mathrm{PO}(p,q+1)=\mathrm{Isom}(\HH^{p,q})$, the Lie algebra $\oo(p,q+1)$ identifies with the set of \emph{Killing fields} on~$\HH^{p,q}$, \ie of vector fields whose flow is isometric: an element $Y\in\oo(p,q+1)$ corresponds to the Killing field
\begin{equation}
\label{eqn:defkill}
x \longmapsto \frac{\D}{\D t}\Big|_{t=0} \exp(tY)\cdot x \: \in T_x\HH^{p,q},
\end{equation}
see \eg \cite[Ex.\,17.9]{der00}.
This identification is canonical, and we will use it throughout Sections~\ref{sec:pseudo-Riem-contraction} and~\ref{sec:space-contr-Coxeter}, writing \eg
$$u(\gamma)(x)\in T_x\HH^{p,q}$$
for the value of the Killing field associated to $u(\gamma)$ at a point $x\in \HH^{p,q}$, when $(\rho,u):\Gamma\rightarrow\OO(p,q+1)\ltimes\mathfrak{o}(p,q+1)$ is a representation.

The Lie algebra $\oo(p,q+1)$, endowed with its Killing form $\kappa_{p,q+1}$, identifies with $\RR^{p',q'}$ where
\begin{equation} \label{eqn:sign-Killing}
(p',q') = \big(p(q+1) ~,~ (p^2+q^2-p+q)/2\big).
\end{equation}
The adjoint action of $\OO(p,q+1)$ on $\oo(p,q+1)$ preserves $\kappa_{p,q+1}$.
Using geometric properties of actions on $\HH^{p,q}$, we shall construct proper affine actions on $\oo(p,q+1)\simeq\RR^{p',q'}$.

\subsection{Relating the pseudo-metric $d_{\HH^{p,q}}$ with the highest eigenvalue} \label{subsec:remind-lambda1}

For any $g\in \mathrm{GL}(\RR^{p+q+1})$, we denote by $\lambda_1(g)\geq\dots\geq\lambda_{p+q+1}(g)$ (\resp $\mu_1(g)\geq\dots\geq\mu_{p+q+1}(g)$) the logarithms of the moduli of the eigenvalues (\resp singular values) of~$g$.
If $\Vert\cdot\Vert$ denotes the operator norm associated to the standard Euclidean norm on $\RR^{p+q+1}$, then
\begin{equation}\label{eqn:muis}
\mu_1(g) = \log \Vert g\Vert.
\end{equation}

An element $g\in\mathrm{GL}(\RR^{p+q+1})$ is called \emph{proximal in $\PP(\RR^{p+q+1})$}, or \emph{proximal} for short, if $\lambda_1(g)>\lambda_2(g)$; equivalently, $g$ admits a unique attracting fixed point in $\PP(\RR^{p,q+1})$.
If $g\in \OO(p,q+1)\subset \mathrm{GL}(\RR^{p+q+1})$, then $\lambda_i(g) = -\lambda_{p+q+2-i}(g)$ and $\mu_i(g) = -\mu_{p+q+2-i}(g)$ for all~$i$.
In particular, any proximal element $g\in\OO(p,q+1)$ has, not only an attracting fixed point, but also a repelling fixed point in $\PP(\RR^{p,q+1})$; these points belong to $\partial\HH^{p,q}$.

In Section~\ref{subsec:proof-pseudo-Riem-contract-proper-G} we shall use the following classical observations.

\begin{lemma} \label{lem:orbit-growth}
Let $g\in\OO(p,q+1)$ and let $y\in \HH^{p,q}$.
\begin{enumerate}
  \item\label{item:growth-general} We have
  $\displaystyle \underset{n\to +\infty}{\limsup} \; \frac{1}{n} \, d_{\HH^{p,q}}(y,g^n y) \leq \lambda_1(g)$.
  \item\label{item:growth-prox} If $g$ is proximal, with attracting and repelling fixed points $\xi_g^{\pm}\in\partial\HH^{p,q}$, and if $y\notin (\xi_g^+)^{\perp} \cup (\xi_g^-)^{\perp}$, then
  $\displaystyle \frac{1}{n} \, d_{\HH^{p,q}}(y,g^n y) \underset{n\to +\infty}{\longrightarrow} \lambda_1(g)$.
\end{enumerate}
\end{lemma}

\begin{proof}
\eqref{item:growth-general} By writing the Jordan decomposition of~$g$ as the commuting product of a hyperbolic, a unipotent, and an elliptic element, we see that $\Vert g^n\Vert$ is bounded above by a polynomial times $\mathrm{e}^{n\lambda_1(g)}$, hence so is $|\langle v, g^n v \rangle_{p,q+1}|$ where $[v]=y$.
We conclude using \eqref{eqn:d-Hpq-inner-prod}.

\eqref{item:growth-prox} Again, by \eqref{eqn:d-Hpq-inner-prod}, it suffices to study the growth of $\langle v, g^n v \rangle_{p,q+1}$ where $[v]=y$.
The projective hyperplane $(\xi_g^\pm)^\perp$ is the projectivization of the sum of the generalized eigenspaces of~$g$ for eigenvalues other than $\mathrm{e}^{\mp \lambda_1(g)}$.
Therefore the assumption on $y$ means that~$v$, when decomposed over the generalized eigenspaces of~$g$, has nonzero components $v^+, v^-$ along $\xi_g^+$ and~$\xi_g^-$.
These components are \emph{not} orthogonal.
In the pairing $\langle v, g^n v \rangle_{p,q+1}$, the term $\langle v^-, g^n v^+ \rangle_{p,q+1} = \mathrm{e}^{n\lambda_1(g)} \langle v^-,  v^+ \rangle_{p,q+1}$ therefore dominates all the others, and grows like $\mathrm{e}^{n\lambda_1(g)}$ as $n\to +\infty$.
We conclude using \eqref{eqn:d-Hpq-inner-prod}.
\end{proof}

We shall also use the following fact, which combines a result of Abels--Margulis--Soifer \cite[Th.\,5.17]{ams95} with a small compactness argument of Benoist: see the two lemmas of \cite[\S\,4.5]{ben97}.
Recall that a representation into $\GL(\RR^{p,q+1})$ is called \emph{strongly irreducible} if its image does not preserve any nonempty finite union of nonzero proper linear subspaces of~$\RR^{p,q+1}$.

\begin{fact}[\cite{ams95,ben97}] \label{fact:AMS}
Let $\Gamma$ be a discrete group and $\rho' : \Gamma\to\GL(\RR^{p,q+1})$ a strongly irreducible representation such that $\rho'(\Gamma)$ contains a proximal element.
Then there exist a finite set $F\subset\Gamma$ and a constant $C_{\rho'}\geq 0$ such that for any $\gamma\in\Gamma$, we may find $f\in F$ such that $\rho'(\gamma f)$ is proximal and satisfies $\lambda_1(\rho'(\gamma f))\geq \mu_1(\rho'(\gamma))-C_{\rho'}$.
\end{fact}

\subsection{A Finsler metric on the Riemannian symmetric space}

Let $\mathbb{X}=G/(\OO(p)\times \OO(q+1))$ be the Riemannian symmetric space of $G=\OO(p,q+\nolinebreak 1)$, with basepoint $x_0=[e]\in\mathbb{X}$.
In Section~\ref{subsec:proof-pseudo-Riem-contract-proper-G} we will use the following $G$-invariant Finsler metric $d_{\mathbb{X}}$ on~$\mathbb{X}$:
\begin{equation} \label{eqn:finsler}
d_{\mathbb{X}}(g\cdot x_0,g'\cdot x_0) := \mu_1(g^{-1}g') = \log \Vert g^{-1}g'\Vert
\end{equation}
for all $g,g'\in G$, where $\Vert\cdot\Vert$ is the Euclidean operator norm on $\RR^{p+q+1}$ as above.
This is indeed a metric: $d_{\mathbb{X}}$ vanishes only on the diagonal of $\mathbb{X}\times\mathbb{X}$ because $\mu_1|_G$ vanishes only on $G \cap \OO(p+q+1) = \OO(p) \times \OO(q+1)$; symmetry follows from the equality $\mu_1(g) = \mu_1(g^{-1})$ for $g\in G=\OO(p,q+1)$; and $d_{\mathbb{X}}$ satisfies the triangle inequality because the operator norm $\Vert \cdot \Vert$ is submultiplicative.

\section{Metric contraction and properness} \label{sec:metric-contraction}

In this section we give some sufficient conditions for the properness of actions of discrete groups on $\OO(p,1)$ and $\oo(p,1)$; we shall use these conditions to prove Propositions \ref{prop:ex-affine} and~\ref{prop:ex-G} in Section~\ref{sec:riem-ex}.
Extensions to $\OO(p,q+1)$ and $\oo(p,q+1)$ will be given in Section~\ref{sec:pseudo-Riem-contraction}, and used to prove Theorems \ref{thm:proper-action-g} and~\ref{thm:proper-action-G} in Section~\ref{sec:space-contr-Coxeter}.

In the whole section, we consider, for a topological group~$G$:
\begin{itemize}
   \item the action of $G\times G$ on~$G$ by right-and-left multiplication: $(g_1,g_2)\cdot g=g_2gg_1^{-1}$;
   \item if $G$ is a Lie group with Lie algebra~$\g$, the affine action of $G\ltimes\g$ on~$\g$ through the adjoint action: $(g,Z)\cdot Y=\Ad(g)Y+Z$.
\end{itemize}

\subsection{Actions on groups} \label{subsec:metric-contract-G}

Let $G$ be a topological group acting continuously by isometries on a complete metric space $(\mathbb{X},d)$ which is proper (\ie closed balls are compact).
 Given a discrete group $\Gamma$ and a representation $(\rho, \rho'):\nolinebreak\Gamma \to G\times G$, recall that a map $f:\mathbb{X} \to \mathbb{X}$ is called \emph{$(\rho,\rho')$-equivariant} if for all $\gamma\in \Gamma$ and $x\in \mathbb{X}$,
\begin{equation} \label{eqn:rho-rho'-equiv}
f(\rho(\gamma)\cdot x)=\rho'(\gamma)\cdot f(x).
\end{equation}
We shall use the following terminology.

\begin{definition} \label{def:contract-equivar-deform-G}
Let $\Gamma$ be a discrete group and $\rho : \Gamma\to G$ a representation defining a properly discontinuous action of $\Gamma$ on~$\mathbb{X}$.
A representation $\rho' :\nolinebreak\Gamma\to\nolinebreak G$ is \emph{coarsely uniformly contracting} with respect to~$(\rho,\mathbb{X})$ if there exist a nonempty $\rho(\Gamma)$-invariant closed set $\mathcal{O}\subset \mathbb{X}$ and a $(\rho,\rho')$-equivariant continuous map $f : \mathcal{O} \to \mathbb{X}$ which is \emph{coarsely $\Kap$-Lipschitz} for some $\Kap<1$, \ie there exists $\Kap'\in\RR$ such that for all $x,y\in\mathcal{O}$,
  $$d(f(x),f(y)) \leq \Kap\,d(x,y)+\Kap' .$$
  If we can take $\Kap'=0$, then we say that $f$ is \emph{$\Kap$-Lipschitz} and $\rho'$ is \emph{uniformly contracting} with respect to~$(\rho, \mathbb{X})$.
\end{definition}

The following general statement, applied to $(G,\mathbb{X})=(\OO(p,1), \HH^p)$, will let us derive properness of certain actions on~$G$ by right-and-left multiplication from coarse uniform contraction on~$\mathbb{X}$.
By $\mathcal{F}(\mathcal{O})$ we will always refer to the set of compact subsets of a complete proper metric space~$\mathcal{O}$, endowed with the Hausdorff topology.
Note that a properly discontinuous action of $\Gamma$ on $\mathcal O$ induces a properly discontinuous action of $\Gamma$ on $\mathcal F(\mathcal O)$.

\begin{proposition} \label{prop:coarse-proj-G}
Let $G$ be a topological group acting continuously by isometries on a proper complete metric space $(\mathbb{X},d)$.
Let $\Gamma$ be a discrete group and $(\rho,\rho'):\Gamma \rightarrow G\times G$ a representation such that the action of $\Gamma$ on $\mathbb{X}$ via~$\rho$ is properly discontinuous, and such that $\rho'$ is coarsely uniformly contracting with respect to~$(\rho, \mathbb{X})$, with $\mathcal{O}$, $f$, $C<1$, and $C'$ as in Definition~\ref{def:contract-equivar-deform-G}. 
Then the map
\begin{eqnarray*}
\Pi :\quad G & \longrightarrow & \hspace{.45cm} \mathcal{F}(\mathcal{O})\\
g & \longmapsto & \big\{ x\in\mathcal{O} ~|~ 
d(g\cdot x, f(x)) 
\text{ \emph{is minimal}}\big\}
\end{eqnarray*}
is well defined and takes any compact set to a compact set.
Moreover, $\Pi$ is equivariant with respect to the actions of $\Gamma$ on~$G$ by right-and-left multiplication via $(\rho,\rho')$, and on~$\mathcal{F}(\mathcal{O})$ via~$\rho$.
In particular, the action of $\Gamma$ on~$G$ by right-and-left multiplication via $(\rho,\rho')$ is properly discontinuous.
\end{proposition}

\begin{proof}
Choose a basepoint $x_0\in\mathcal{O}$.
For any $g\in G$ and $x\in\mathcal{O}$, we have
\begin{align} \label{eqn:dist-gx-fx}
d(g\cdot x,f(x)) & \geq d(g\cdot x, g\cdot x_0) - d(g\cdot x_0, f(x_0))-d(f(x_0), f(x))
\\
& \geq (1-C) \, d(x,x_0) - (C'+d(g\cdot x_0, f(x_0))).\nonumber
\end{align}
Thus $x\mapsto d(g\cdot x, f(x))$ is a proper function on the proper metric space $(\mathcal{O},d|_{\mathcal{O}\times\mathcal{O}})$ for any $g\in G$, and so $\Pi$ is well defined.

The map $\Pi$ takes compact sets to compact sets.
Indeed, let $\mathscr{C}$ be a compact subset of~$G$.
By continuity of the action, there exists $R>0$ such that for all $g\in\mathscr{C}$ we have $d(g\cdot x_0,f(x_0))\leq R$, hence $\Pi(g)$ is contained in $\{ x\in\mathcal{O} \,|\, d(x,x_0)\leq (C'+2R)/(1-C)\}$ by \eqref{eqn:dist-gx-fx}.

The equivariance of $\Pi$ follows from that of~$f$: for any $\gamma\in\Gamma$ and $x\in\mathcal{O}$,
$$d \big(\rho'(\gamma)g\rho(\gamma)^{-1}\cdot (\rho(\gamma)\cdot x), f(\rho(\gamma)\cdot x) \big) = d(\rho'(\gamma)g\cdot x, \rho'(\gamma)\cdot f(x)) = d(g\cdot x, f(x)).$$

By equivariance of~$\Pi$, since the action of $\Gamma$ on $\mathcal{F}(\mathcal{O})$ via $\rho$ is properly discontinuous, so is the action of $\Gamma$ on~$G$ by right-and-left multiplication via $(\rho,\rho')$.
Indeed, if $\mathscr C$ is a compact subset of~$G$, then $\Pi(\mathscr C)$ is a compact subset of $\mathcal{F}(\mathcal{O})$.
By properness of the action of $\Gamma$ on $\mathcal F(\mathcal O)$, there is a finite subset $S \subset \Gamma$ such that $\rho(\gamma)\cdot \Pi(\mathscr C) \cap \Pi(\mathscr C) = \varnothing$ for all $\gamma \in \Gamma \smallsetminus S$.
By equivariance of~$\Pi$, we have $\rho'(\gamma)\mathscr C\rho(\gamma)^{-1} \cap \mathscr C = \varnothing$ for all $\gamma \in \Gamma \smallsetminus S$~as~well.
\end{proof}

\subsection{Equivariance and contraction for vector fields} \label{subsec:def-inf-metric-contract}

Suppose now that $G$ is a finite-dimensional Lie group, $\mathbb{X}$ is a Hadamard manifold, and $G$ acts smoothly by isometries on~$\mathbb{X}$.

There is a natural linear map $\Psi$ from the Lie algebra $\g$ of~$G$ to the space of \emph{Killing fields} on~$\mathbb{X}$, \ie vector fields on~$\mathbb{X}$ whose flow is isometric: it takes $Y\in\g$ to the vector field $\Psi(Y) := (x \mapsto \frac{\D}{\D t}\big|_{t=0} \exp(tY)\cdot x)$ as in \eqref{eqn:defkill}.
For any $(g,Y,x)\in G\times\g\times\mathbb{X}$ we have
\begin{equation} \label{eqn:pushforward-Psi}
\Psi(\mathrm{Ad}(g)Y)(g\cdot x) = g_* (\Psi(Y)(x)).
\end{equation}

Similarly to the notions of equivariance \eqref{eqn:rho-rho'-equiv} and contraction (Definition~\ref{def:contract-equivar-deform-G}) above, we shall use the following terminology.

\begin{definition} \label{def:equiv-field}
Let $(\rho, u):\Gamma \to G\ltimes \g$ be a representation.
A vector field $Z$ on $\mathbb{X}$ is \emph{$(\rho,u)$-equivariant} if whenever $Z(x)$ belongs to some $\Psi(Y)\in\Psi(\g)$, the vector $Z(\rho(\gamma)\cdot x)$ belongs to $\Psi((\rho,u)(\gamma)\cdot Y)$: namely, for all $\gamma\in \Gamma$ and $x\in \mathbb{X}$,
\begin{equation} \label{eqn:equivar}
Z(\rho(\gamma)\cdot x) = \rho(\gamma)_*Z(x) + \Psi(u(\gamma))(\rho(\gamma)\cdot x).
\end{equation}
\end{definition}

\begin{definition} \label{def:contract-equivar-deform-g}
Let $\Gamma$ be a discrete group and $\rho : \Gamma\to G$ a representation defining a properly discontinuous action of $\Gamma$ on~$\mathbb{X}$.
  A $\rho$-cocycle $u : \Gamma\to\g$ is \emph{coarsely uniformly contracting} with respect to~$\mathbb{X}$ if there exist a nonempty $\rho(\Gamma)$-invariant closed set $\mathcal{O}\subset \mathbb{X}$ and a $(\rho,u)$-equivariant continuous vector field $Z$ on~$\mathcal{O}$ which is \emph{coarsely $\kap$-lipschitz} for some $\kap<0$, \ie there exists $\kap'\in\RR$ such that for all $x,y\in\mathcal{O}$,
  \begin{equation}\label{eqn:contract}
  \frac{\D}{\D t}\Big|_{t=0} \: d \left ( \exp_x(tZ(x)), \exp_y(tZ(y)) \right ) \leq \kap \, d (x,y) + \kap' .
  \end{equation}
  If we can take $\kap'=0$, then we say that $Z$ is \emph{$\kap$-lipschitz} and $u$ is \emph{uniformly contracting} with respect to~$\mathbb{X}$.
\end{definition}

Here we use a lowercase \emph{l} in \emph{lipschitz} to emphasize the infinitesimal aspect.
The left-hand side of~\eqref{eqn:contract} is linear in~$Z$ by Remark~\ref{rem:hadamard} and vanishes by definition for $Z$ a Killing field.
The $(\rho, u)$-equivariance of $Z$ entails that both sides of \eqref{eqn:contract} are invariant under replacing $(x,y)$ with $(\rho(\gamma)\cdot x, \rho(\gamma)\cdot y)$. 

Definitions \ref{def:equiv-field} and~\ref{def:contract-equivar-deform-g} are motivated by the classical notion of equivariance~\eqref{eqn:rho-rho'-equiv} and by Definition~\ref{def:contract-equivar-deform-G}, via the following construction.
This construction will actually produce all uniformly contracting cocycles appearing in this paper, and also lies at the heart of \cite{dgk16} with $(G,\mathbb{X})=(\SO(2,1),\HH^2)$.

\begin{lemma} \label{lem:derivate}
Consider an open interval $I\ni 0$, a smooth path $(\rho_\tau)_{\tau \in I}$ in $\mathrm{Hom}(\Gamma, G)$, and the $\rho_0$-cocycle $u := \frac{\D}{\D \tau}\big |_{\tau=0}\,\rho_{\tau} \rho_0^{-1}$.
For any smooth family $(f_\tau : \mathbb{X} \to\nolinebreak\mathbb{X})_{\tau \in I}$ of maps such that $f_0=\mathrm{Id}_{\mathbb{X}}$ and $f_\tau$ is $(\rho_0, \rho_{\tau})$-equivariant for all $\tau \in I$, the derivative $Z(x):=\frac{\D}{\D \tau} {\big |}_{\tau=0}\, f_\tau(x)$ is $(\rho_0,u)$-equivariant.
If moreover there exists $\kap\in\RR$ such that $f_\tau$ is $(1+\kap \tau)$-Lipschitz for all $\tau\geq 0$, then $Z$ is $\kap$-lipschitz.
In particular, if $\kap<0$, then $u$ is uniformly contracting with respect to $\mathbb{X}$.
\end{lemma}

\begin{proof}
For any $\tau\geq 0$ and $x\in\mathbb{X}$, by equivariance of $f_\tau$ we have
$$f_\tau(\rho_0(\gamma)\cdot x) = \big(\rho_{\tau}(\gamma)\rho_0(\gamma)^{-1}\big)\,\rho_0(\gamma)\cdot f_\tau(x).$$
Differentiating both sides with respect to $\tau$ at $\tau=0$ yields
$$Z(\rho_0(\gamma)\cdot x)=(\mathrm{Id}\circ \rho_0(\gamma))_*(Z(x)) + \Psi(u(\gamma))\big(\rho_0(\gamma)\cdot f_0(x)\big),$$
hence the equivariance property \eqref{eqn:equivar}.

Suppose that there exists $\kap\in\RR$ such that $f_\tau$ is $(1+\kap \tau)$-Lipschitz for all $\tau\geq 0$.
For any $x,y\in\mathbb{X}$, since $\frac{\D}{\D \tau} {\big |}_{\tau=0}\, \exp_x(\tau Z(x)) = Z(x) = \frac{\D}{\D \tau} {\big |}_{\tau=0}\, f_\tau(x)$ and similarly for~$y$, we see that
$$\frac{\D}{\D \tau}\Big|_{\tau=0} \: d \left ( \exp_x(\tau Z(x)), \exp_y(\tau Z(y)) \right ) = \frac{\D}{\D \tau} {\Big |}_{\tau=0}\, d (f_\tau(x),f_\tau(y))$$
by applying Proposition~\ref{prop:deriv-dHpq} twice.
Since the Lipschitz inequality\linebreak $d(f_\tau(x), f_\tau(y)) \leq (1+\kap \tau)\,d (x,y)$ is an equality at $\tau=0$, we can differentiate it at $\tau=0$, yielding the uniform contraction property \eqref{eqn:contract}.
\end{proof}

\subsection{Actions on Lie algebras} \label{subsec:metric-contract-g}

Here is the infinitesimal counterpart of Proposition~\ref{prop:coarse-proj-G}.
Again, we denote by $\Psi$ the natural linear map from the Lie algebra $\g$ of~$G$ to the set of Killing fields on~$\mathbb{X}$.

\begin{proposition} \label{prop:coarse-proj-g}
Let $G$ be a finite-dimensional Lie group acting smoothly by isometries on a Hadamard manifold $\mathbb{X}$.
Let $\Gamma$ be a discrete group and $(\rho,u) : \Gamma \to G\ltimes \g$ a representation such that $\Gamma$ acts properly discontinuously on $\mathbb{X}$ via~$\rho$, and such that $u$ is coarsely uniformly contracting with respect to~$\mathbb{X}$, with $\mathcal{O}$, $Z$, $\kap<0$, and~$\kap'$ as in Definition~\ref{def:contract-equivar-deform-g}.
Then the map
\begin{eqnarray*}
\pi :\quad \g & \longrightarrow & \hspace{.45cm} \mathcal{F}(\mathcal{O})\\
Y & \longmapsto & \big\{ x\in\mathcal{O} ~|~ \Vert Z(x) - \Psi(Y)(x) \Vert  
\text{ \emph{is minimal}}\big\}
\end{eqnarray*}
is well defined and takes any compact set to a compact set.
Moreover, $\pi$ is equivariant with respect to the affine action of $\Gamma$ on~$\g$ via $(\rho,u)$ and 
the action of $\Gamma$ on~$\mathcal{F}(\mathcal{O})$ via~$\rho$.
In particular, the affine action of $\Gamma$ on~$\g$ via $(\rho,u)$ is properly discontinuous.
\end{proposition}

\begin{proof}
Choose a basepoint $x_0\in\mathcal{O}$. 
For any vector field $V$ on~$\mathcal{O}$ and any $x\in\mathcal{O}$, Remark~\ref{rem:hadamard} implies
$$- \Vert V(x_0)\Vert - \Vert V(x)\Vert \leq \frac{\D}{\D t}\Big|_{t=0} \: d \big( \exp_x(tV(x)), \exp_{x_0}(tV(x_0))\big) \, ;$$
in particular, if $V$ is coarsely $\kap$-lipschitz (Definition~\ref{def:contract-equivar-deform-g}), then
\begin{equation} \label{eqn:ineq-c-lip}
- \Vert V(x_0)\Vert - \Vert V(x)\Vert \leq \kap\, d(x, x_0) + \kap' .
\end{equation}
By Remark~\ref{rem:hadamard}, the sum of a coarsely $\kap$-lipschitz vector field and a Killing field is still coarsely $\kap$-lipschitz.
Therefore, for any $Y\in \g$, by applying \eqref{eqn:ineq-c-lip} to $V=Z-\Psi(Y)$, we find
$$ \Vert Z(x) - \Psi(Y)(x) \Vert \geq |\kap|\,d(x,x_0) - (\Vert Z(x_0) - \Psi(Y)(x_0) \Vert + \kap').$$
Thus $x\mapsto\Vert Z(x) - \Psi(Y)(x) \Vert$ is a proper function on~$\mathbb{X}$ for any $Y\in\g$, \ie $\pi$ is well defined. 
Moreover, since $Y\mapsto \Vert Z(x_0) - \Psi(Y)(x_0) \Vert$ is bounded on compact sets, $\pi$ takes compact sets to compact sets.
The equivariance of~$\pi$ follows from that of~$Z$, from the linearity of~$\Psi$, and from \eqref{eqn:pushforward-Psi}: for any $\gamma\in\Gamma$ and $x\in\mathcal{O}$,
\begin{align*}
 & \Vert \big(Z - \Psi\big(\mathrm{Ad}(\rho(\gamma))Y+u(\gamma)\big)\big)(\rho(\gamma)\cdot x) \Vert \\
 = & \left \Vert \rho(\gamma)_* (Z(x)) + \Psi(u(\gamma))(\rho(\gamma)\cdot x) - \Psi\big(\mathrm{Ad}(\rho(\gamma)) Y + u(\gamma)\big)(\rho(\gamma)\cdot x) \right \Vert \\
 = & \Vert \rho(\gamma)_*(Z(x)) - \rho(\gamma)_*(\Psi(Y)(x)) \Vert =  \Vert (Z - \Psi(Y))(x) \Vert.
\end{align*}
By equivariance of~$\pi$, since the action of $\Gamma$ on $\mathcal{F}(\mathcal{O})$ via $\rho$ is properly discontinuous, so is the affine action of $\Gamma$ on~$\g$ via $(\rho,u)$.
\end{proof}

\subsection{Fibrations}

While the coarse projection arguments of Sections \ref{subsec:metric-contract-G} and \ref{subsec:metric-contract-g} (and, later, Section~\ref{subsec:proof-pseudo-Riem-contract-proper-g}) are useful for determining proper discontinuity of an action, such arguments seem to give little information about the topology of the quotient manifolds.
However, when the coarsely Lipschitz maps $f$ and lipschitz vector fields $Z$ of Propositions \ref{prop:coarse-proj-G} and~\ref{prop:coarse-proj-g} are well behaved, we can deduce explicit fibrations for the quotient manifolds modeled on the group $G$ and its Lie algebra~$\g$.
This idea already appeared in \cite[Prop.\,7.2]{gk17} and in \cite[Prop.\,6.3]{dgk16}.

\begin{proposition} \label{prop:fibrations}
\begin{enumerate}[leftmargin=0.6cm]
  \item In the context of Proposition~\ref{prop:coarse-proj-G}, suppose that $f$ is $C$-Lipschitz (\ie $C'=0$) and that $\mathcal{O}=\mathbb{\mathbb{X}}$.
  Then $\Pi : G \to \mathcal{F}(\mathcal{O})$ takes any $g\in G$ to a singleton of $\mathcal{O}=\mathbb{X}$, \ie we have a $\Gamma$-equivariant map $\Pi : G\to\mathbb{X}$, and this map is continuous.
  If furthermore $G$ acts transitively on $\mathbb{X}$, with point stabilizer $K$, and if $\Gamma$ is torsion-free, then the quotient of $G$ under the action of $\Gamma$ by right-and-left multiplication via $(\rho, \rho')$ is a $K$-bundle over the manifold $\rho(\Gamma)\backslash \mathbb{X}$.
  \item In the context of Proposition~\ref{prop:coarse-proj-g}, suppose that $Z$ is $c$-Lipschitz (\ie $c'=\nolinebreak 0$) and that $\mathcal{O}=\mathbb{\mathbb{X}}$.
  Then $\pi : \g \to \mathcal{F}(\mathcal{O})$ takes any $Y\in\g$ to a singleton of $\mathcal{O}=\mathbb{X}$, \ie we have a $\Gamma$-equivariant map $\pi : \g\to\mathbb{X}$, and this map is continuous.
  If furthermore $G$ acts transitively on $\mathbb{X}$, with infinitesimal point stabilizer $\mathfrak{k}$, and $\Gamma$ is torsion-free, then the quotient of $\g$ under the affine action of $\Gamma$ via $(\rho, u)$ is a $\mathfrak{k}$-bundle over the manifold $\rho(\Gamma)\backslash\mathbb{X}$.
\end{enumerate}
\end{proposition}

In~(2), by the \emph{infinitesimal stabilizer} of a point $x\in\mathbb{X}$ we mean the set of elements $Y\in\g$ corresponding (via $\Psi$) to Killing fields on $\mathbb{X}$ that vanish at~$x$, or equivalently the Lie algebra of the stabilizer of $x$ in~$G$.

\begin{proof}
\begin{enumerate}[leftmargin=0.7cm]
  \item For any $g\in G$ the map $g^{-1}\circ f:\mathbb{X}\to \mathbb{X}$ is $C$-Lipschitz, hence admits a unique fixed point $\Pi(g)$ in~$\mathbb{X}$ since $C<1$.
  The map $\Pi : G\to\mathbb{X}$ is continuous: indeed, if $g'\in G$ is close enough to~$g$ that $d(x, g'^{-1}\circ f(x))\leq (1-C)\,\varepsilon$ where $x=\Pi(g)$, then $g'^{-1}\circ f$ takes the $\varepsilon$-ball centered at $x$ to itself, hence $\Pi(g')$ lies within $\varepsilon$ of $x=\Pi(g)$.
  
  If $G$ acts transitively on $\mathbb{X}$, then $\Pi$ is surjective, and each fiber $\Pi^{-1}(x)=\{g\in G~|~ g\cdot x = f(x) \}$ is a left $G$-translate of the stabilizer of $x$ in~$G$.
  This gives $G$ the structure of a $\Gamma$-equivariant $K$-bundle over~$\mathbb{X}$, which descends to the quotient manifolds if $\Gamma$ has no torsion; this structure is smooth if $f$ is.

  \item For any $Y\in\g$ the vector field $Z - \Psi(Y)$ is $c$-lipschitz on~$\mathbb{X}$, hence inward-pointing on any large enough sphere since $c<0$.
  By Brouwer's fixed point theorem, $Z-\Psi(Y)$ therefore admits a zero $\pi(Y)$ in~$\mathbb{X}$, unique since $c<0$.
  The map $\pi : \g\to\mathbb{X}$ is continuous: indeed, if $Y'\in \g$ is close enough to $Y$ that $\Vert \Psi(Y-Y')(x) \Vert \leq |c|\,\varepsilon$ where $x=\pi(Y)$, then $Z - \Psi(Y') = (Z-\Psi(Y))+ \Psi(Y-Y')$ is inward-pointing on the sphere of radius $\varepsilon$ centered at~$x$ (for the Killing field $\Psi(Y-Y')$ has constant component along any given geodesic through~$x$), hence $\pi(Y')$ lies within $\varepsilon$ of $x=\pi(Y)$.
  
  If $G$ acts transitively on $\mathbb{X}$, then $\pi$ is surjective, and each fiber $\pi^{-1}(x)=\{Y\in\g~|~ \Psi(Y)(x) = Z(x) \}$ is a $\g$-translate of the infinitesimal stabilizer of~$x$.
  This gives $\g$ the structure of a $\Gamma$-equivariant $\mathfrak{k}$-bundle over~$\mathbb{X}$, which descends to the quotient manifolds if $\Gamma$ has no torsion; this structure is smooth if $Z$ is.\qedhere
\end{enumerate}
\end{proof}

\begin{remark} \label{rem:topology}
In~(2) above with $G$ acting transitively on $\mathbb{X}$, the quotient $(\rho,u)\Gamma\backslash\g$ is isomorphic, as a $\mathfrak{k}$-bundle, to the quotient by $\rho(\Gamma)$ of the tautological $\mathfrak{k}$-subbundle $B$ of $\g\times \mathbb{X} \rightarrow \mathbb{X}$ whose fiber above $x\in\mathbb{X}$ is the infinitesimal stabilizer of~$x$.
Indeed, by a partition-of-unity argument, $\pi:\mathfrak{g}\rightarrow \mathbb{X}$ admits a $\Gamma$-equivariant section $\sigma$ such that $\sigma(x)(x)=Z(x)$ for all $x\in \mathbb{X}$, and $(Y,x)\mapsto \sigma(x)+Y$ then defines an equivariant bundle isomorphism $B\overset{\sim}{\rightarrow}\g$.

In~(1), no section exists in general, but one can still describe the bundle structure on~$G$ topologically as a pullback of the tautological $K$-bundle over~$\mathbb{X}$ by any $(\rho, \rho')$-equivariant map (not necessarily contracting). 
\end{remark}

\begin{remark} \label{rem:converse-Riem}
Suppose $(G,\mathbb{X})=(\OO(p,1), \HH^p)$.
\begin{enumerate}[leftmargin=0.6cm] \setcounter{enumi}{-1}
  \item As in Section~\ref{subsec:o(p,q+1)}, the Killing form $\kappa_{p,1}$ on $\g=\mathfrak{o}(p,1)$ has signature $(p,p(p-1)/2)$.
  The stabilizer $K=\OO(p)\times\OO(1)$ (\resp the infinitesimal stabilizer $\mathfrak{k}=\mathfrak{o}(p)$) appearing in Proposition~\ref{prop:fibrations} is a maximal negative definite totally geodesic subspace of $G$ (\resp linear subspace of~$\g$), for the $G$-invariant pseudo-Riemannian structure induced by~$\kappa_{p,1}$.
  \item When $\rho$ is geometrically finite, a converse to Proposition~\ref{prop:coarse-proj-G} holds: up to switching $\rho$ and~$\rho'$, the action of $\Gamma$ on $G$ by right-and-left multiplication via $(\rho, \rho')$ is properly discontinuous if and only if the action of $\Gamma$ on $\mathbb{X}$ via~$\rho$ is properly discontinuous and $\rho'$ is coarsely uniformly contracting with respect to $(\rho, \mathbb{X})$.
  In fact in this case $\rho'$ is actually uniformly contracting with respect to $(\rho, \mathbb{X})$, and one can find a $(\rho,\rho')$-equivariant $\Kap$-Lipschitz map (for some $\Kap<1$) defined on $\mathcal{O}=\mathbb{X}=\HH^p$, making Proposition~\ref{prop:fibrations}.(1) applicable.
  This was proved in \cite{kasPhD} for $p=2$ and convex cocompact~$\rho$, and in \cite{gk17} in general.
  \item For $p=2$ and convex cocompact~$\rho$, a similar converse to Proposition~\ref{prop:coarse-proj-g}  holds up to replacing $u$ by $-u$, by \cite[Th.\,1.1]{dgk16}; again, $u$ is actually uniformly contracting with respect to~$\mathbb{X}$, and one can find a $(\rho,u)$-equivariant $\kap$-lipschitz vector field (for some $\kap<0$) defined on $\mathcal{O}=\mathbb{X}=\HH^p$, making Proposition~\ref{prop:fibrations}.(2) applicable.
  The same statement for geometrically finite~$\rho$ will be proved in \cite{dgk-parab}.
  On the other hand, this converse fails for $p=3$, as $\oo(3,1) \simeq \mathfrak{psl}(2,\CC)$ has a complex structure and properness, unlike uniform contraction, is unaffected when we multiply a cocycle by a nonzero complex number.
\end{enumerate}
\end{remark}

\section{Examples of proper actions on $\OO(p,1)$ and $\oo(p,1)$ for small~$p$} \label{sec:riem-ex}

In this section we prove Propositions~\ref{prop:ex-affine} and~\ref{prop:ex-G} by applying Propositions~\ref{prop:coarse-proj-G} and~\ref{prop:coarse-proj-g}.

\subsection{Uniformly contracting maps obtained by colorings} \label{subsec:surface}

Recall that a discrete subgroup of $\OO(p,1)$ is called \emph{convex cocompact} if it acts with compact quotient on a nonempty convex subset of the hyperbolic space $\HH^p$.
The property for a representation of a discrete group to be injective and discrete with convex cocompact image is stable under small deformations and under embedding into a larger $\OO(p',1)$.
We shall use the word \emph{coloring} for any map to a finite set: the image of an element is then called its \emph{color}.

Our proof of Propositions~\ref{prop:ex-affine} and~\ref{prop:ex-G} uses the following construction.

\begin{proposition} \label{prop:color}
Let $\Gamma$ be a convex cocompact subgroup of $\OO(p,1)$ generated by the orthogonal reflections $\{\gamma_i\}_{1\leq i \leq k}$ in the faces $\{F_i\}_{1\leq i \leq k}$ of a right-angled convex polyhedron of~$\HH^p$.
For $1\leq i\leq k$, let $v_i=(w_i,1)\in \RR^{p,1}$ be a normal vector to~$F_i$. 
Suppose there exist an integer $m\geq 0$ and a coloring $\sigma:\{1, \dots, k\} \rightarrow \{0, \dots, \col\}$ such that $\sigma(i)\neq \sigma(j)$ whenever $F_i$ intersects~$F_j$.
Let $u_0, \dots, u_{\col}$ be the vertices of a regular simplex inscribed in the unit sphere of~$\RR^{\col}$ (if $m=0$, take $u_0=0\in \RR^0$).
For any $1\leq i\leq k$ and $t\in\RR$, we set
$$v^t_i := \left ( \cosh (t) \, w_i , \sqrt{\col} \, \sinh (t) \, u_{\sigma(i)}, 1\right ) \in \RR^{p+\col,1}.$$
Then for any $t\in\RR$, the representation $\rho_t:\Gamma \rightarrow \OO(p+\col,1)$ taking $\gamma_i$ to the orthogonal reflection in $(v^t_i)^\perp \subset \RR^{p+\col,1}$ is well defined, and for small enough $|t|$ it is still injective and discrete, with convex cocompact image.
Moreover, for any $0<t\leq s$ with $t$ small enough, there exists a $(\rho_t, \rho_{s})$-equivariant, $\frac{\cosh (t)}{\cosh(s)}$-Lipschitz map $f_{t,s} : \HH^{p+\col}\to\HH^{p+\col}$; we may take $f_{t,t}=\mathrm{Id}_{\HH^{p+m}}$ and $f_{t,s}$ depending smoothly on $(t,s)$.
\end{proposition}

\begin{proof}
Let $t\in\RR$.
In order to prove that $\rho_t$ is well defined, we only need to check that $\langle v^t_i, v^t_j \rangle_{p+\col,1}=0$ whenever $F_i$ intersects~$F_j$. 
Since $\langle v_i, v_j \rangle_{p,1}=0$ we have $\langle w_i, w_j \rangle_{p,0}=1$, and $\langle u_{\sigma(i)}, u_{\sigma(j)} \rangle_{\col,0}=-1/\col$.
Therefore 
\begin{align*}
\langle v^t_i, v^t_j \rangle_{p+\col,1} & = \cosh^2 (t) \langle w_i, w_j \rangle_{p,0}+\col \, \sinh^2 (t) \langle u_{\sigma(i)}, u_{\sigma(j)} \rangle_{\col,0} -1 \\ & = \cosh^2 (t) - \sinh^2 (t)  -1 = 0.
\end{align*}
For small enough $|t|$ the representation $\rho_t : \Gamma\to\OO(p+\col,1)$ is injective and discrete with convex cocompact image, since this property is stable under embedding $\OO(p,1)$ into $\OO(p+\col,1)$ and under small deformation.

We now assume that $t>0$ is such that $\rho_t$ is faithful and discrete, and fix $s\geq t$.
Let $P_t\subset \PP(\RR^{p+\col,1})$ be the polytope bounded by the $\PP(v^t_i)^\perp$ for $1\leq i\leq k$, so that $P_t\cap \HH^{p+\col}$ is a fundamental domain for the action of $\rho_t(\Gamma)$ on $\HH^{p+\col}$, with polyhedral boundary.
Define similarly $P_{s}$.
We endow the affine chart $\{x_{p+\col+1}=\nolinebreak 1\} \simeq \RR^{p+\col}$ of $\PP(\RR^{p+\col,1})$ with the standard Euclidean metric, so that $\HH^{p+m}$ is the unit open ball centered at~$0$.
In this chart, the linear map $\frac{\cosh (s)}{\cosh (t)} \, \mathrm{Id}_{\RR^p} \oplus \frac{\sinh (s)}{\sinh (t)} \, \mathrm{Id}_{\RR^{\col}}$ takes the $v^t_i$ to the~$v^{s}_i$. 
Dually,
$$f:=\frac{\cosh (t)}{\cosh (s)} \mathrm{Id}_{\RR^p} \oplus \frac{\sinh (t)}{\sinh (s)} \mathrm{Id}_{\RR^{\col}}$$
must take $P_t$ to $P_{s}$. 
The restriction of $f$ to $P_t\cap \HH^{p+\col}$ can be $(\rho_t, \rho_{s})$-equivariantly extended by orthogonal reflections in the faces of $P_t$ and $P_{s}$, yielding a $(\rho_t,\rho_{s})$-equivariant map $f_{t,s} : \HH^{p+\col}\to\HH^{p+\col}$.
This map is projective in restriction to any $\rho_t(\Gamma)$-translate of~$P_t$, and it takes each reflection face of $P_t$ to the corresponding reflection face of $P_{s}$, hence it is globally continuous.
If $s>t$, then $f_{t,s}(\HH^{p+\col})$ is strictly contained in $\HH^{p+\col}$.

In order to check that the continuous map $f_{t,s} : \HH^{p+\col}\to\HH^{p+\col}$ is $\frac{\cosh (t)}{\cosh (s)}$-Lipschitz, by the triangle inequality, it is enough to focus on the restriction $f$ to one fundamental domain $\HH^{p+\col}\cap P_t$. 
Since $f(\HH^{p+\col})$ (an ellipsoid) is contained in the ball of radius $\frac{\cosh (t)}{\cosh (s)}<1$ centered at~$0$ in our Euclidean chart, while $\HH^{p+\col}$ is the unit ball, the result is an immediate consequence of the following Lemma~\ref{lem:compare}, which quantifies Remark~\ref{rem:Hilb-metric-include}.
\end{proof}

\begin{lemma} \label{lem:compare}
Fix a Euclidean chart $\RR^n$ of $\PP^n(\RR)$.
If $B_r$ denotes the ball of radius~$r$ in~$\RR^n$ centered at~$0$, then the Hilbert metrics on $B_r$ and~$B_1$ satisfy $d_{B_r}(x,y) \geq d_{B_1}(x,y)/r$ for all $r\in (0,1)$ and $x,y\in B_r$. 
\end{lemma}

\begin{proof}
Consider a line $\ell$ of the Euclidean chart $\RR^n$ through points $x,y \in B_1$, with $\ell\cap \partial B_1=\{a,b\}$ and $a,x,y,b$ lying in this order on~$\ell$.
We can parametrize $\ell$ at unit Euclidean velocity by $(x_t)_{t\in\RR}$ so that $(a,x,y,b)=(x_{-\alpha},x_0,x_\delta,x_{\beta})$ for some $\delta, \alpha, \beta>0$. 
We have
$$d_{B_1}(x,y)=\frac{1}{2} \log \left ( \frac{\delta+\alpha}{\delta-\beta} \middle / \frac{0+\alpha}{0-\beta} \right)\ \underset{\delta\to 0}{\sim}\ \delta \: \frac{\alpha^{-1}+\beta^{-1}}{2}.$$
The factor $\nu^{B_1}_{\ell, x} := (\alpha^{-1}+\beta^{-1})/2$ expresses the Finsler norm associated to the Hilbert metric $d_{B_1}$ near~$x$, in the direction of~$\ell$, in terms of the ambient Euclidean norm. 
If we replace $B_1$ with a scaled ball $B_{1-\tau}$ for some $\tau>0$, then the new endpoints of $\ell\cap B_{1-\tau}$ lie at linear coordinates $-\alpha_\tau$ and $\beta_\tau$ such that $\frac{\D}{\D\varsigma}\big|_{\varsigma=\tau}\,\alpha_{\varsigma} \leq -1$ and $\frac{\D}{\D\varsigma}\big|_{\varsigma=\tau}\,\beta_{\varsigma} \leq -1$.
Therefore
$$\frac{{\displaystyle\frac{\D}{\D\varsigma}\Big|_{\varsigma=\tau}\ \nu^{B_{1-\varsigma}}_{\ell, x}}}{\nu^{B_{1-\tau}}_{\ell, x}}\geq \frac{\alpha_\tau^{-2}+\beta_\tau^{-2}}{\alpha_\tau^{-1}+\beta_\tau^{-1}}= \frac{\alpha_\tau/\beta_\tau+\beta_\tau/\alpha_\tau}{\alpha_\tau+\beta_\tau} \geq \frac{1}{1-\tau},$$
where we use $\alpha_\tau+\beta_\tau \leq 2-2\tau$ for the last inequality. 
Integrating this logarithmic derivative over $\tau\in[0,1-r]$, we find $\nu^{B_r}_{\ell, x} \geq \nu^{B_1}_{\ell, x}/r$. 
This is valid for all $\ell$ and~$x$, hence $d_{B_r} \geq d_{B_1}/r$.
\end{proof}

\begin{remarks} \label{rem:colors}
\begin{enumerate}[leftmargin=0.6cm]
  \item The $(\rho_t,\rho_{s})$-equivariant maps $f_{t,s}$ of Proposition~\ref{prop:color} are not smooth, but continuous and piecewise projective. 
  Similarly, setting $u_t:= \frac{\D}{\D s}\big |_{s=t}\,\rho_{s} \rho_t^{-1} : \Gamma \rightarrow \oo(p+m,1)$, the $(\rho_t,u_t)$-equivariant vector fields $Z_t := \frac{\D}{\D s} {\big |}_{s=t}\, f_{s}$, which are uniformly contracting by Lemma~\ref{lem:derivate}, are not smooth.
  However, they can be made smooth while remaining uniformly contracting, \eg using the equivariant convolution procedure described in \cite[\S\,5.5]{dgk16}.
  \item Groups $\Gamma$ as in Proposition~\ref{prop:color} are finitely generated, hence admit a torsion-free subgroup $\Gamma_1$ of finite index by the Selberg lemma \cite[Lem.\,8]{sel60}.
  Propositions \ref{prop:coarse-proj-G}, \ref{prop:coarse-proj-g}, and \ref{prop:fibrations} apply in this setting, yielding:
  \begin{itemize}
    \item quotient manifolds $(\rho_t,\rho_{s})(\Gamma_1)\backslash\OO(p+m,1)$ with the structure of an $(\OO(p+m)\times\OO(1))$-bundle over the hyperbolic manifold $\rho_t(\Gamma_1)\backslash\HH^p$,
    \item quotient affine manifolds $(\rho_t,u_t)(\Gamma_1)\backslash\oo(p+m,1)$ with the structure of an $\mathfrak{o}(p+m,1)$-bundle over the hyperbolic manifold $\rho_t(\Gamma_1)\backslash\HH^p$.%
  \end{itemize}%
  \item \label{rem:colored-mst} The case $m=0$ (a single color) of Proposition~\ref{prop:color} is valid: it applies to groups $\Gamma$ freely generated by $k$ involutions, and acting by reflections on~$\HH^p$.
  For $p=2$, applying Proposition~\ref{prop:coarse-proj-g} to an index-two torsion-free subgroup $\Gamma_1$ of~$\Gamma$, we obtain examples of proper affine actions of free groups on $\oo(2,1)\simeq\RR^{2,1}$; the corresponding affine $3$-manifolds are \emph{Margulis spacetimes}.
\end{enumerate}
\end{remarks}

\subsection{Proof of Propositions~\ref{prop:ex-affine} and~\ref{prop:ex-G}}

Let $\Gamma$ be the discrete subgroup of $\OO(2,1)$ generated by the reflections in the faces of a convex right-angled $k$-gon in $\HH^p=\HH^2$, for $k\geq 6$ even.
Color the sides of the $k$-gon, alternatingly, with labels $0$ and~$1$.
Applying Proposition~\ref{prop:color} with $\col=1$ yields, for small enough $0<t<s$, faithful and discrete representations $\rho_t,\rho_{s} : \Gamma\to\nolinebreak\OO(3,1)$ and $(\rho_t, \rho_{s})$-equivariant, $\frac{\cosh (t)}{\cosh(s)}$-Lipschitz maps $f_{t,s} : \HH^3\to\HH^3$ (Figure~\ref{rhombus} shows a fundamental polyhedron).
In particular, $\rho_{s}$ is uniformly contracting with respect to $(\rho_t,\HH^3)$ (Definition~\ref{def:contract-equivar-deform-G}), and by Lemma~\ref{lem:derivate} the $\rho_t$-cocycle $u_t := \frac{\D}{\D s}\big |_{s=t}\,\rho_{s} \rho_t^{-1}$ is uniformly contracting with respect to~$\HH^3$ (Definition~\ref{def:contract-equivar-deform-g}) since $\frac{\D }{\D s}\big|_{s=t}\,\frac{\cosh (t)}{\cosh (s)} = -\tanh (t) <0$.
Applying Propositions \ref{prop:coarse-proj-G} and~\ref{prop:coarse-proj-g}, we obtain Propositions \ref{prop:ex-G}.(a) and~\ref{prop:ex-affine}.(a).

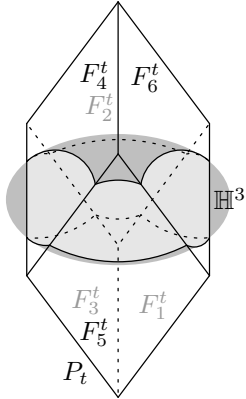
\begin{figure}[h!]
\refstepcounter{figure} \addtocounter{figure}{-1}
\label{rhombus}
\labellist
\small\hair 2pt
\pinlabel {$\HH^3$} [c] at 83.5 75
\pinlabel {$F_4^t$} [c] at 33 121
\pinlabel {$F_5^t$} [c] at 33 25
\pinlabel {$F_6^t$} [c] at 52 120
\pinlabel {$\color{light-gray}F_1^t$} [c] at 55 35
\pinlabel {$\color{light-gray}F_2^t$} [c] at 35 109
\pinlabel {$\color{light-gray}F_3^t$} [c] at 30 37
\pinlabel {$P_t$} [c] at 26 10
\endlabellist
\centering{
\begin{minipage}{3cm}\includegraphics[width=3cm]{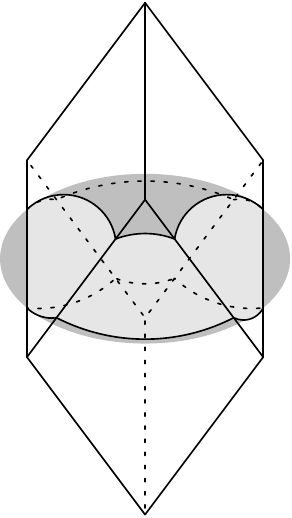}\end{minipage}\begin{minipage}{9.5cm}\caption{A fundamental domain $P_t \cap \HH^3$ for the action of $\rho_t(\Gamma)$ on~$\HH^3$, bounded by planes $F_i^t=(v_i^t)^\perp$ for $1\leq i\leq k$ (here $k=6$, \ie $\Gamma$ is a right-angled hexagon group). 
The hexa\-hedron $P_t$ becomes vertically more elongated as $t\rightarrow 0$. 
The faces $F_1^t, F_2^t, F_3^t$ are at the back; the ellipsoid $\HH^3$ is shaded.}\end{minipage}}
\end{figure} 

Similarly, in order to prove Propositions~\ref{prop:ex-G}.(b) and~\ref{prop:ex-affine}.(b), it is enough to color the faces of the regular $120$-cell of~$\RR^4$ with $\col+1=5$ colors so that adjacent faces receive different colors. 
This is a well-known construction which we briefly recall below; see Figure~\ref{fig:120}.

The 120-cell can be described as follows.
Let $\varphi=\frac{\sqrt{5}+1}{2}=1.618\dots$ be the golden ratio.
Let $w_1\dots, w_{120} \in \RR^4$ be the unit vectors obtained from the rows of the matrix
$$\frac{1}{2} \begin{pmatrix} 0 &  0 & 0 & 2 \\ 1&1&1&1 \\ 
0 & \varphi^{-1} & 1  & \varphi \end{pmatrix}$$
by sign changes and even permutations of the four coordinates.
We endow $\RR^4$ with its standard scalar product $\langle\cdot,\cdot\rangle_{4,0}$.
The affine hyperplanes $w_i+w_i^{\perp}$, for $1\leq i \leq 120$, cut out a regular 120-cell in~$\RR^4$.
Cells of the 120-cell are regular dodecahedra, four of which meet at each vertex, and two cells share a (pentagonal) face if and only if the dual vectors $w_i, w_j$ are \emph{neighbors}, which means by definition that $\langle w_i,w_j \rangle_{4,0} = \varphi/2$.
Each $w_i$ has 12 neighbors. 

\begin{figure}[h]
\centering
\hspace{-2cm}
\includegraphics[width = 4.4cm, trim=0 0 0 0]{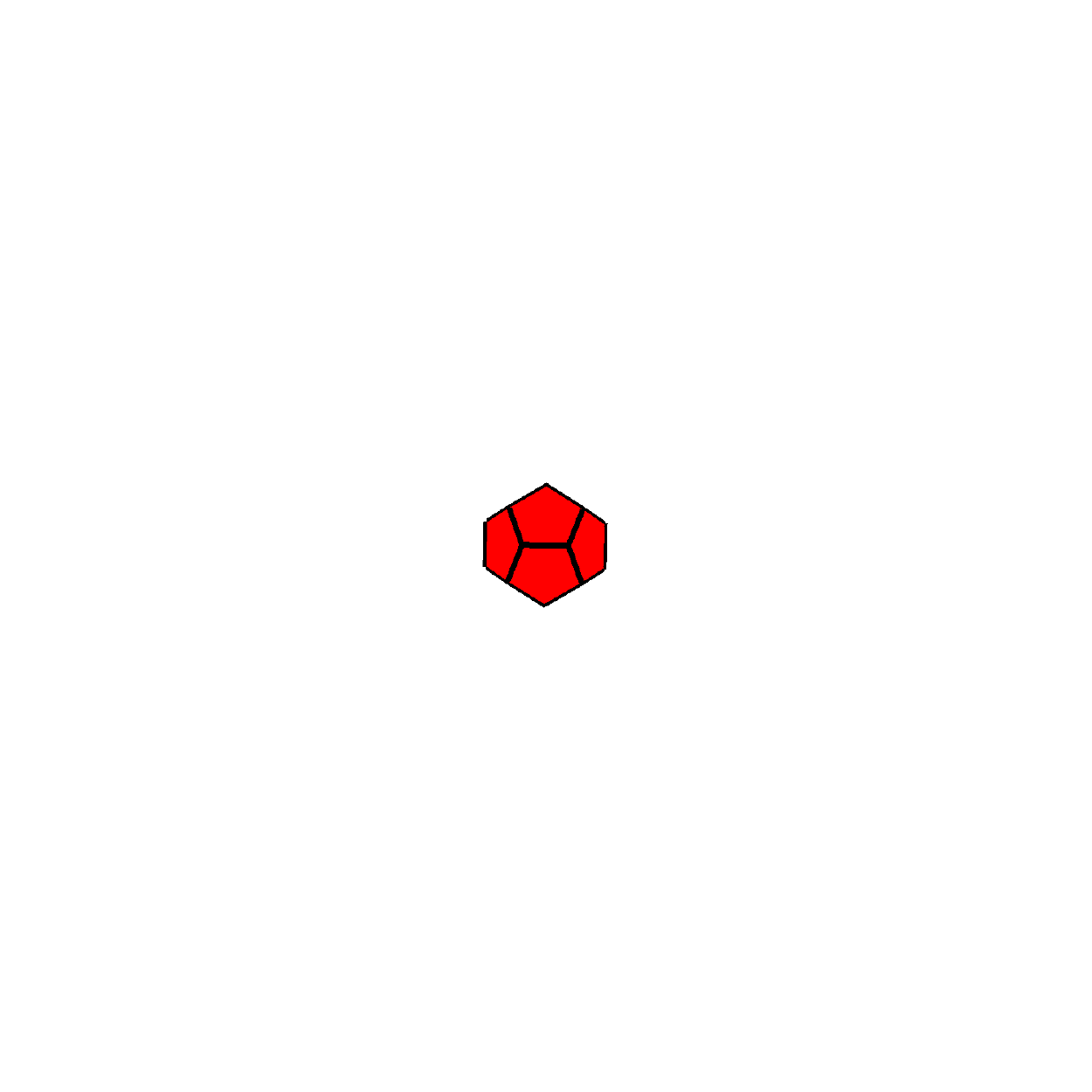}
\hspace{-3.3cm}
\includegraphics[width = 4.4cm, trim=0 0 0 0]{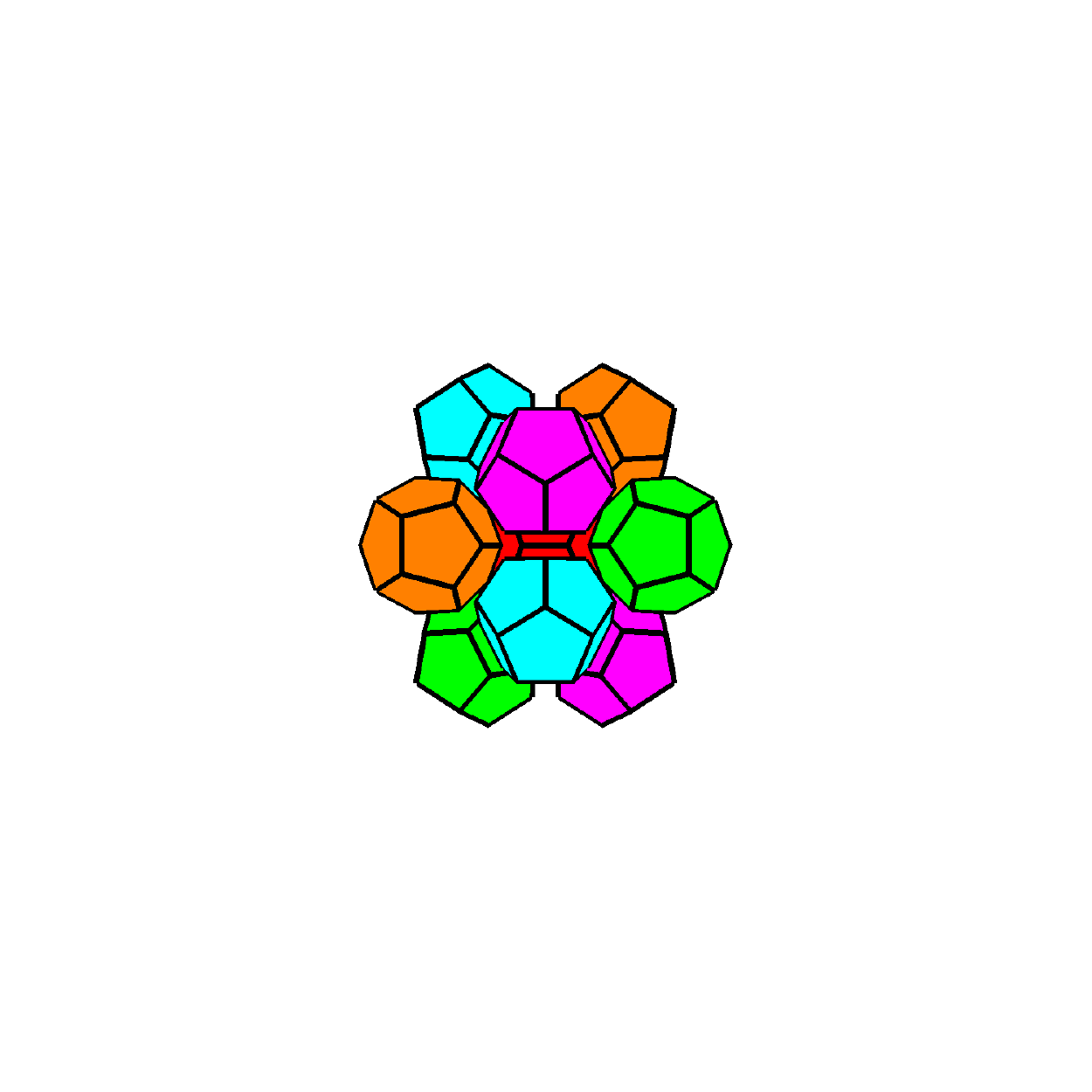}
\hspace{-2.5cm}
\includegraphics[width = 4.4cm, trim=0 0 0 0]{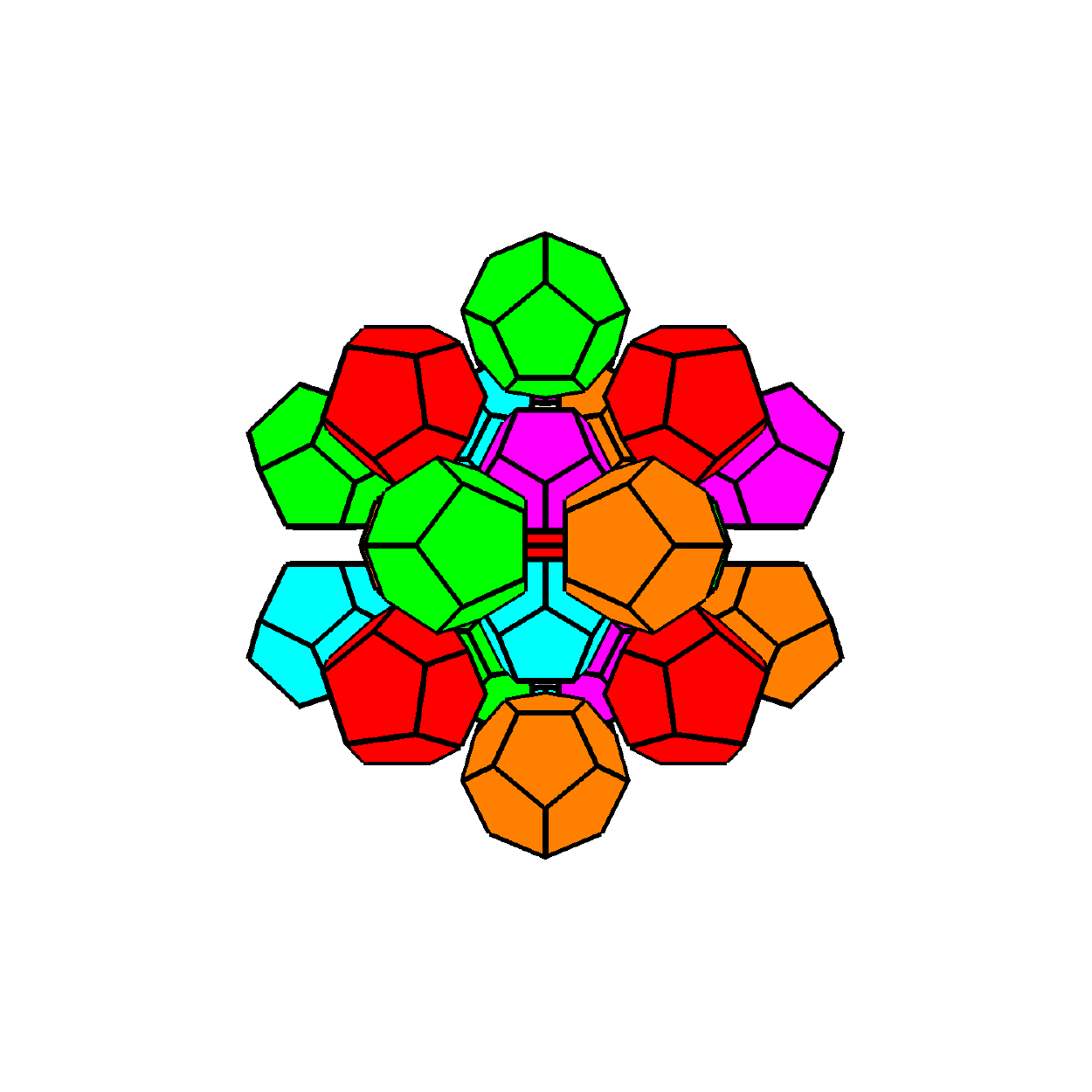}
\hspace{-1.8cm}
\includegraphics[width = 4.4cm, trim=0 0 0 0]{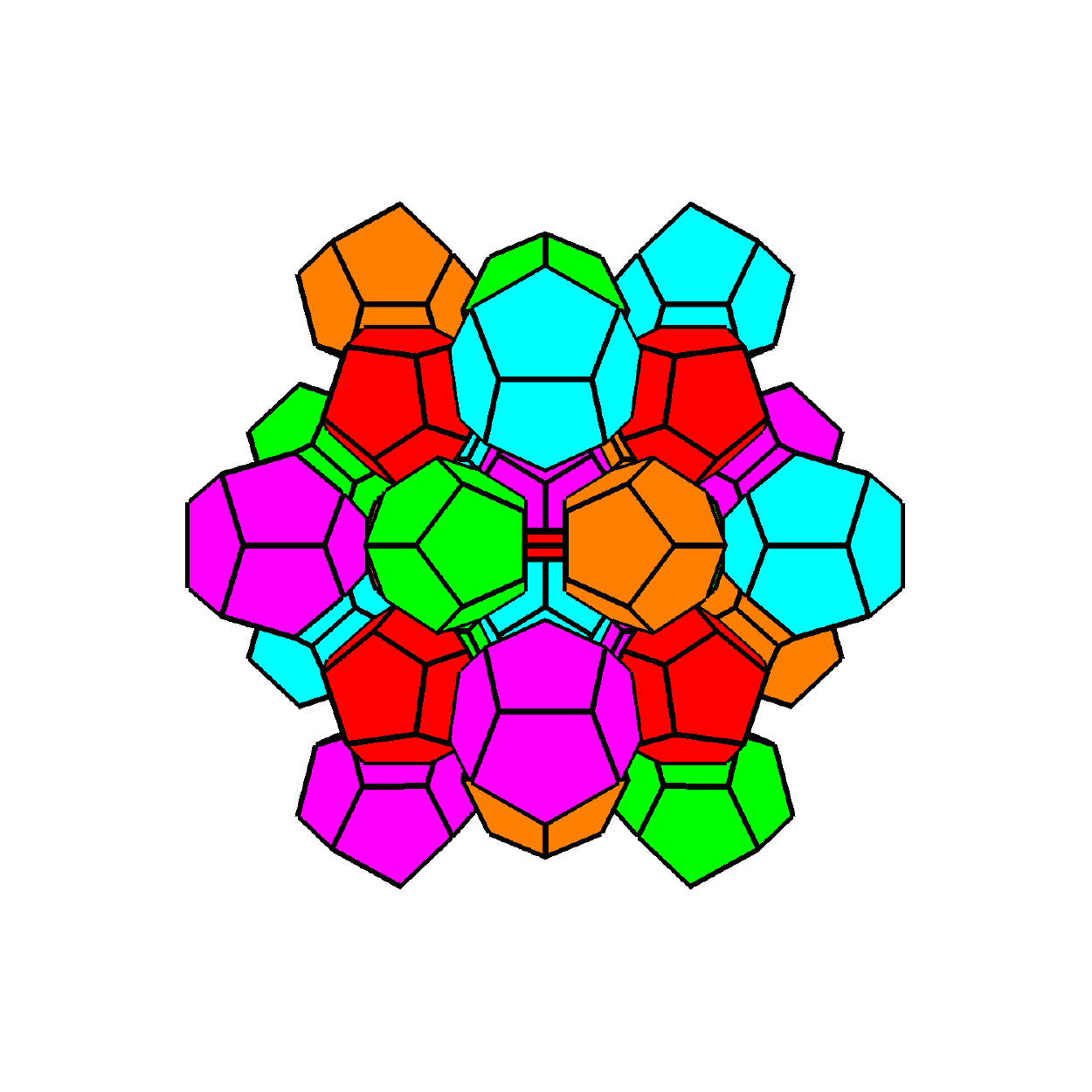}
\hspace{-0.8cm}
\includegraphics[width = 4.0cm, trim=0 -6mm 0 0 0]{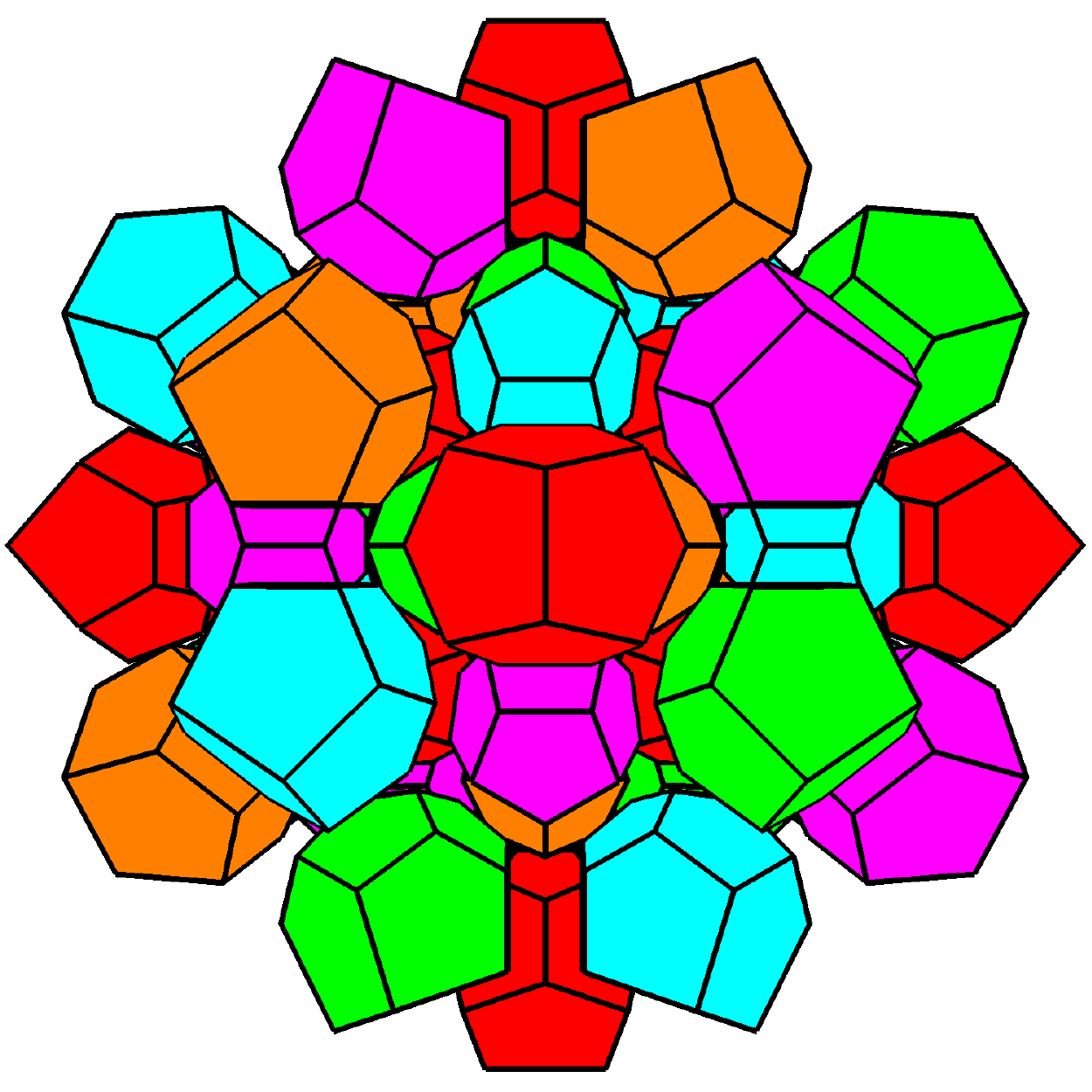}
\caption{The five-coloring of the faces of the 120-cell, stereographically projected to~$\RR^3$. Shown are five (out of nine) layers of dodecahedral cells surrounding a single red cell. The cells are pulled apart from each other a small amount to expose the internal layers.}
\label{fig:120}
\end{figure}

We now explain how to color the 120 vectors $w_i$ (\ie the corresponding cells) with $5$ colors so that no two neighbors have the same color.
Seen as unit quaternions, the $w_i$ form a group which maps surjectively, via the covering $\psi: \mathbb{S}^3 \rightarrow \SO(3)$, to the icosahedron group $\mathfrak{A}_5$ (even permutations on $5$ symbols $\{0,1,2,3,4\}$), with kernel $\{1,-1\}$.
We color each $w_i$ with the value $\sigma(i) \in \{0,1,2,3,4\}$ that the associated permutation takes at the symbol~$0$.
Any neighbors $w_i, w_j$ always have different colors: indeed, the corresponding permutations differ by a $5$-cycle, since $\mathrm{Re}(w_i^{-1} w_j) = \langle w_i,w_j \rangle_{4,0} = \varphi/2=\cos(\pi/5)$ shows that $\psi(w_i^{-1} w_j)$ has order $5$ in $\mathfrak{A}_5$.
Propositions~\ref{prop:ex-G}.(b) and~~\ref{prop:ex-affine}.(b) are proved.

\begin{remark}
As in Remark~\ref{rem:colors}.(2), if $\Gamma$ is one of the reflection groups in Propositions \ref{prop:ex-affine} and~\ref{prop:ex-G}, then $\Gamma$ admits a finite-index, torsion-free subgroup $\Gamma_1$, which is either a surface group (case~(a)) or a $4$-manifold group (case~(b)).

In case~(a), Proposition~\ref{prop:ex-affine} gives proper affine actions of the surface group~$\Gamma_1$ on $\RR^6 \cong \oo(3,1)$.
The linear part of each such affine action is the composition of a certain quasi-Fuchsian representation $\Gamma_1 \to \SO_0(3,1) \cong \PSL(2,\CC)$ with the adjoint representation $\SO_0(3,1) \to \SL(\oo(3,1))$.

By contrast, for any dimension~$n\geq 2$, an affine action on $\RR^n$ whose linear part is the composition of a Fuchsian (\ie faithful and discrete) representation $\rho_0: \Gamma_1 \to \SL(2,\RR)$ with an irreducible representation $\tau_n: \SL(2,\RR) \to \SL(n,\RR)$ is never properly discontinuous. 
This was proved by Mess~\cite{mes07} and Goldman--Margulis~\cite{gm00} in the case $n = 3$ (for which $\tau_3$ is the adjoint representation, with image $\SO_0(2,1)$) and in general by Labourie~\cite{lab01}. 
Recently, it was shown further~\cite{dz} that a continuous deformation of $\tau_n \circ \rho_0$ is never the linear part of a proper affine action, either.
Such representations make up what is known as a \emph{Hitchin component} of $\Hom(\Gamma_1, \SL(n,\RR))$. 
It would be interesting to determine which connected components of $\Hom(\Gamma_1, \SL(n,\RR))$ contain the linear part of a proper affine action by a surface group $\Gamma_1$.
\end{remark}

\subsection{A variant of Lemma~\ref{lem:compare}}

In order to prove Theorems~\ref{thm:proper-action-g} and \ref{thm:proper-action-G} in Section~\ref{subsec:final}, we will need the following variant of Lemma~\ref{lem:compare}, in which we replace the Euclidean metric on an affine chart of $\PP^n(\RR)$ with the standard spherical metric of $\PP^n(\RR)$ given, for all $v,w\in\RR^{n+1}\smallsetminus\{0\}$, by
\begin{equation} \label{eqn:spherical-metric}
d_{\PP^n(\RR)}\big([v],[w]\big) = \min\big( \measuredangle(v,w),\measuredangle(v,-w)\big) \in \Big[0,\frac{\pi}{2}\Big],
\end{equation}
where by convention angles between vectors take values in $[0,\pi]$.

\begin{lemma} \label{lem:compare-general}
Let $(H_{\tau})_{\tau\geq 0}$ be a smooth family of connected open subsets of $\PP^n(\RR)$ with smooth boundaries $\partial H_{\tau}$.
Let $\ell \subset \PP^n(\RR)$ be an open projective segment intersecting $\partial H_0$ twice, transversely.
For small $\tau\geq 0$, let $(a_{\tau}, b_{\tau})$ be the open segment $\ell\cap H_{\tau}$, endowed with its Hilbert metric $d_{(a_{\tau}, b_{\tau})}$.
If $\partial H_{\tau}$ expands outwards everywhere with normal velocity $\geq 1$ (for $d_{\PP^n(\RR)}$) at $\tau=0$, then for all $x,y\in (a_0,b_0)$,
$$\frac{\D }{\D t}\Big|_{\tau=0} \, d_{(a_{\tau}, b_{\tau})} (x,y) \leq -2 \, d_{(a_0, b_0)} (x,y).$$
\end{lemma}

\begin{proof}
Let $a_0,x,y,b_0 \in \ell\cap\overline{H}_0$ be lined up in this order.
Let $s\mapsto x_s$ be a unit-speed (for $d_{\PP^n(\RR)}$) parametrization of~$\ell$ such that $x=x_0$ and $y=x_{\delta}$ for some $\delta>0$.
For any small $\tau\geq 0$, we have $(a_{\tau},x,y,b_{\tau})=(x_{-\alpha_{\tau}},x_0,x_\delta, x_{\beta_{\tau}})$ for some $\alpha_{\tau}, \beta_{\tau} \in (0,\pi)$ with $\alpha_{\tau}+\beta_{\tau}<\pi$.
If $\partial H_{\tau}$ expands outwards with normal velocity $\geq 1$, then $\frac{\D}{\D\tau}\big|_{\tau=0}\,\alpha_{\tau}\geq 1$ and $\frac{\D}{\D\tau}\big|_{t=0}\,\beta_{\tau}\geq 1$.
Then
$$d_{(a_{\tau},b_{\tau})}(x,y)=\frac{1}{2}\log \left ( \frac{\tan \delta + \tan \alpha_{\tau}}{\tan \delta-\tan \beta_{\tau}} \middle / \frac{0 + \tan \alpha_{\tau}}{0 -\tan \beta_{\tau}} \right ) \underset{\delta\to 0}\sim \delta \, \frac{\cot \alpha_{\tau} + \cot \beta_{\tau} }{2}~.$$
The factor $\nu^{\tau}_{\ell, x} := (\cot \alpha_{\tau}+\cot \beta_{\tau})/2$ expresses the Hilbert metric $d_{(a_{\tau},b_{\tau})}$ near~$x$ in terms of the ambient spherical metric, in the $\ell$ direction. 
Since $\cot'=-\sin^{-2}$, its logarithmic derivative at $\tau=0$ satisfies
\begin{align*}
\frac{{\displaystyle\frac{\D}{\D\tau}\Big|_{\tau=0}\, \nu^{\tau}_{\ell, x}}}{\nu^{0}_{\ell, x}} & = \frac{- \big(\frac{\D}{\D\tau}\big|_{\tau=0}\, \alpha_{\tau}\big) \sin^{-2}\alpha_0 - \big(\frac{\D}{\D\tau}\big|_{\tau=0}\,\beta_{\tau}\big) \sin^{-2} \beta_0}{\cot \alpha_0 + \cot \beta_0} \\ 
 & \leq - \; \frac{\sin^{-2}\alpha_0+ \sin^{-2} \beta_0}{\cot \alpha_0 + \cot \beta_0} = - \; \frac{\frac{\sin \alpha_0}{\sin \beta_0} + \frac{\sin \beta_0}{\sin \alpha_0}}{\sin (\alpha_0+\beta_0)} \leq -2.
\end{align*}
Thus $(\D \nu^{\tau}_{\ell, x}/\D\tau)|_{\tau=0} \leq -2 {\nu^{0}_{\ell, x}}$ for all $x\in \ell\cap H_0$. Integrating $\nu^{\tau}_{\ell, x}$ for $x$ in a subsegment $\sigma \subset \ell\cap H_{\tau}$ returns the Hilbert length of $\sigma$ in $\ell\cap H_{\tau}$; the result follows by exchanging the integration and differentiation.
\end{proof}

\section{Pseudo-Riemannian contraction and properness} \label{sec:pseudo-Riem-contraction}

In Propositions~\ref{prop:coarse-proj-G} and~\ref{prop:coarse-proj-g} we established sufficient properness conditions for actions of a discrete group $\Gamma$ on a topological group $G$ by right-and-left multiplication, and on a finite-dimensional Lie algebra~$\g$ by affine transformations through the adjoint action.
These conditions involved a notion of coarse uniform contraction in a metric $G$-space $(\mathbb{X},d)$.

Fixing integers $p,q\in\NN$ with $p+q\geq 1$, we shall now state and prove sufficient properness conditions (Theorem~\ref{thm:contract-proper}) for similar actions on $G=\OO(p,q+1)$ and on $\g=\mathfrak{o}(p,q+1)\simeq\RR^{(p+q+1)(p+q)/2}$. This will be used in Section~\ref{sec:space-contr-Coxeter} to prove Theorems \ref{thm:proper-action-g} and~\ref{thm:proper-action-G}.

\subsection{Uniform spacelike contraction in~$\HH^{p,q}$ and proper actions} \label{subsec:spacelike-contract-def}

In order to state Theorem~\ref{thm:contract-proper}, we first introduce a notion of \emph{spacelike} coarse uniform contraction in the pseudo-Riemannian space $\HH^{p,q}$, endowed with the \emph{pseudo-metric} $d_{\HH^{p,q}}$ of Notation~\ref{not:d-Hpq}.

Let $G=\OO(p,q+1)$.
Given $(\rho,u):\Gamma\rightarrow G\ltimes \g$, we say that a vector field $Z$ defined on a $\rho(\Gamma)$-invariant subset $\mathcal{O}$ of $\HH^{p,q}$ is $(\rho, u)$-\emph{equivariant} if $Z$ satisfies~\eqref{eqn:equivar} for all $\gamma\in\Gamma$ and $x\in \mathcal{O}$. From now on, we will drop the map $\Psi$ of~\eqref{eqn:equivar} from the notation, as it is a canonical isomorphism between the Lie algebra $\g=\mathfrak{o}(p,q+1)$ and the space of Killing fields on~$\HH^{p,q}$.
We introduce the following terminology extending Definitions \ref{def:contract-equivar-deform-G} and~\ref{def:contract-equivar-deform-g}.

\begin{definition} \label{def:contract-spacelike}
Let $\Gamma$ be a discrete group and $\rho : \Gamma\to G=\OO(p,q+1)$ a representation with finite kernel and discrete image, preserving a nonempty properly convex open subset $\Omega$ of $\HH^{p,q}$.
\begin{enumerate}
  \item\label{def:contract-Hpq-G} A representation $\rho' : \Gamma\to G$ is \emph{coarsely uniformly contracting in spacelike directions} with respect to $(\rho, \Omega)$ if there exist a closed nonempty $\rho(\Gamma)$-invariant subset $\mathcal{O}$ of~$\Omega$ (\eg $\Omega$ itself, or a single $\rho(\Gamma)$-orbit) and a continuous $(\rho,\rho')$-equivariant map $f : \mathcal{O}\to\HH^{p,q}$ which is \emph{coarsely $\Kap$-Lipschitz in spacelike directions} for some $\Kap<1$, \ie there exists $\Kap'\in \RR$ such that for all $x,y\in\mathcal{O}$ on a spacelike line,
  $$d_{\HH^{p,q}}(f(x),f(y)) \leq \Kap\,d_{\HH^{p,q}}(x,y) + \Kap' \, .$$
  If we can take $\Kap'=0$, then we say that $f$ is \emph{$\Kap$-Lipschitz in spacelike directions} and $\rho'$ is \emph{uniformly contracting in spacelike directions} with respect to $(\rho,\Omega)$.
  \item\label{def:contract-Hpq-g} A $\rho$-cocycle $u : \Gamma\to\g$ is \emph{coarsely uniformly contracting in spacelike directions} with respect to~$\Omega$ if there are a closed nonempty $\rho(\Gamma)$-invariant subset $\mathcal{O}$ of~$\Omega$, and a continuous $(\rho,u)$-equivariant vector field $Z : \mathcal{O}\to T\HH^{p,q}$  on~$\mathcal{O}$ which is \emph{coarsely $\kap$-lipschitz in spacelike directions} for some $\kap<0$, \ie there exists $\kap'\in\RR$ such that for all $x,y\in\mathcal{O}$ on a spacelike line,
  $$\frac{\D}{\D t}\Big|_{t=0} \: d_{\HH^{p,q}}\left ( \exp_x(tZ(x)), \exp_y(tZ(y)) \right ) \leq \kap \, d_{\HH^{p,q}}(x,y) + \kap' \, .$$
If we can take $\kap'=0$, then we say that $Z$ is \emph{$\kap$-lipschitz in spacelike directions} and $u$ is \emph{uniformly contracting in spacelike directions} with respect to~$\Omega$.
\end{enumerate}
\end{definition}

Of course there are always \emph{two} projective segments between any given pair $(x,y)$ of points of $\PP(\RR^{p+q+1})$.
For $x,y\in\HH^{p,q}$, if one of these segments is a spacelike geodesic segment of $\HH^{p,q}$, then the other projective segment exits~$\HH^{p,q}$.
In Definition~\ref{def:contract-spacelike}, for $x,y\in\mathcal{O}\subset\Omega\subset\HH^{p,q}$ on a spacelike line, the segment between $x$ and~$y$ that remains in~$\HH^{p,q}$ also remains in the properly convex set~$\Omega$.

The following statement is similar to Lemma~\ref{lem:derivate} and its proof is identical, restricted to pairs of points in spacelike position.

\begin{lemma} \label{lem:derivate-Hpq}
Consider an open interval $I\ni 0$, a smooth path $(\rho_\tau)_{\tau\in I}$ in $\mathrm{Hom}(\Gamma, \OO(p,q+1))$, and the $\rho_0$-cocycle $u := \frac{\D}{\D \tau}\big |_{\tau=0}\,\rho_{\tau} \rho_0^{-1}$.
For any smooth family $(f_\tau : \Omega \to \HH^{p,q})_{\tau\in I}$ of maps such that $f_0=\mathrm{Id}_{\Omega}$ and $f_\tau$ is $(\rho_0, \rho_\tau)$-equivariant for all $\tau\in I$, the derivative $Z(x):=\frac{\D}{\D \tau} {\big |}_{\tau=0}\, f_\tau(x)$ is $(\rho_0,u)$-equivariant.
If moreover there exists $\kap\in\RR$ such that $f_\tau$ is $(1+\kap\tau)$-Lipschitz in spacelike directions for all $\tau\geq 0$, then $Z$ is $\kap$-lipschitz in spacelike directions.
\end{lemma}

Here is the main result of this section, generalizing Propositions~\ref{prop:coarse-proj-G}~and~\ref{prop:coarse-proj-g}.

\begin{theorem} \label{thm:contract-proper}
Let $G=\OO(p,q+1)$ for $p,q\in\NN$ with $p+q>1$.
Let $\Gamma$ be a discrete group and $\rho : \Gamma\to G$ a representation with finite kernel and discrete image, preserving a nonempty properly convex open subset $\Omega$ of $\HH^{p,q}$.
\begin{enumerate}
  \item\label{item:Hpq-contract-G} Let $\rho' : \Gamma\to G$ be a strongly irreducible representation such that $\rho'(\Gamma)$ contains a proximal element. 
  If $\rho'$ is coarsely uniformly contracting in spacelike directions with respect to $(\rho,\Omega)$, then the action of $\Gamma$ on $G$ by right-and-left multiplication via $(\rho,\rho')$ is properly discontinuous.
  \item\label{item:Hpq-contract-g} Let $u : \Gamma\to\g$ be a $\rho$-cocycle.
  If $u$ is coarsely uniformly contracting in spacelike directions with respect to~$\Omega$, then the affine action of $\Gamma$ on $\g \simeq \RR^{(p+q+1)(p+q)/2}$ via $(\rho,u)$ is properly discontinuous.
\end{enumerate}
\end{theorem}

See Section~\ref{subsec:remind-lambda1} for the notions of proximality and strong irreducibility.
We shall prove the infinitesimal statement~\eqref{item:Hpq-contract-g} in Section~\ref{subsec:proof-pseudo-Riem-contract-proper-g}, and the statement~\eqref{item:Hpq-contract-G} first in Section~\ref{subsec:proof-pseudo-Riem-contract-proper-G}.

\subsection{A preliminary lemma: actions on convex subsets of $\HH^{p,q}$}

For $x\in\HH^{p,q}$, we denote by $T_x^{+1}\HH^{p,q}$ the set of unit spacelike tangent vectors at~$x$; it is isometric to the quadric $\{v\in\RR^{p,q}~|~\langle v, v \rangle_{p,q}=+1 \}$.

\begin{lemma} \label{lem:spacelike-Gamma-orbit}
Let $\Gamma$ be a discrete group and $\rho : \Gamma\to\OO(p,q+1)$ a representation with finite kernel and discrete image, preserving a nonempty properly convex open subset $\Omega$ of $\HH^{p,q}\subset\PP(\RR^{p,q+1})$.
For any compact subset $\mathcal{D}$ of~$\Omega$,
\begin{enumerate}
  \item\label{item:accum} all accumulation points of the $\rho(\Gamma)$-orbit of~$\mathcal{D}$ are contained in $\partial\HH^{p,q}$;
  \item\label{item:Kx} there exists a bounded family of compact sets $\mathcal{K}_x \subset T_x^{+1}\HH^{p,q}$, for $x$ ranging over~$\mathcal{D}$, such that for all but finitely many $\gamma\in \Gamma$,
  $$\rho(\gamma) \cdot \mathcal{D} \subset \bigcap_{x\in\mathcal{D}} \exp_x(\RR^+ \mathcal{K}_x);$$
  \item\label{item:flight} in particular, if $(\gamma_n)_{n\in \NN}$ goes to infinity in $\Gamma$ (\ie leaves every finite subset of~$\Gamma$), then for any sequences $(x_n)_{n\in \NN}$, $ (x'_n)_{n\in \NN}$ of $\mathcal{D}$ we have $d_{\HH^{p,q}}(x_n, \rho(\gamma_n)\cdot x'_n) \rightarrow +\infty$.
\end{enumerate}
\end{lemma}

\begin{proof}
\eqref{item:accum} Suppose by contradiction that there are sequences $(x_n)_{n\in \NN}$ of $\mathcal{D}$ and $(\gamma_n)_{n\in \NN}$ of $\Gamma$ such that the $\gamma_n$ are pairwise distinct and for $y_n := \rho(\gamma_n)\cdot x_n$ the sequence $(y_n)_{n\in \NN}$ converges to some $y\in\HH^{p,q}$.
We can lift the $x_n\in\HH^{p,q}$ to vectors $\widehat{x}_n\in\widehat{\HH}^{p,q}\subset\RR^{p,q+1}$, \ie  $\langle \widehat{x}_n,\widehat{x}_n \rangle_{p,q+1}=-1$.
Both the $\widehat{x}_n$ and the $\rho(\gamma_n)\cdot \widehat{x}_n$ stay in a compact subset of~$\RR^{p,q+1}$ and $(\rho(\gamma_n)\cdot \widehat{x}_n)_{n\in \NN}$ converges to a unit timelike vector~$\widehat{x}$.
On the other hand, since $\rho$ has finite kernel and discrete image, there exists a vector $v \in \RR^{p,q+1}$ such that $(\rho(\gamma_n)\cdot v)_{n\in \NN}$ leaves every compact subset of~$\RR^{p,q+1}$.
(Indeed, at least one vector of any given basis of $\RR^{p,q+1}$ must satisfy this property.)
Up to passing to a subsequence, we may assume that the direction of $\rho(\gamma_n)\cdot v$ converges to some null direction~$\ell$.
There exists $\varepsilon > 0$ such that all segments $[\widehat{x}_n-\varepsilon v,\widehat{x}_n+\varepsilon v] \subset \RR^{p,q+1}\smallsetminus\{0\}$ project to segments $\sigma_n$ contained in~$\Omega$.
The images $\rho(\gamma_n)\cdot\sigma_n$, which are again contained in~$\Omega$, converge to the full projective line spanned by $\widehat{x}$ and~$\ell$.
This contradicts the proper convexity of~$\Omega$.
Thus the $\rho(\Gamma)$-orbit of~$\mathcal{D}$ does not have any accumulation point in~$\HH^{p,q}$.

\eqref{item:Kx} Let $y\in\partial\HH^{p,q}$ be an accumulation point of the orbit $\rho(\Gamma)\cdot\mathcal{D}$, and consider $x\in\mathcal{D}$.
Then $y$ cannot be seen from~$x$ in a timelike direction by~\eqref{item:accum}, since timelike geodesics do not meet $\partial\HH^{p,q}$.
It cannot be seen in a lightlike direction either: otherwise, the tangent plane to $\partial \HH^{p,q}$ at $y$ contains the interval $[x,y)\subset \Omega$, and any small perturbation $[x',y)$ still lies in~$\Omega$ --- but $x'$ can be chosen so that near $y$ this perturbation crosses $\partial \HH^{p,q}$, which would contradict $\Omega\subset \HH^{p,q}$.
Therefore $y\in\partial\HH^{p,q}$ is seen from~$x$ in a spacelike direction.
We conclude using the compactness of the accumulation set.

\eqref{item:flight} The third statement is a direct consequence of \eqref{item:accum} and~\eqref{item:Kx}.
\end{proof}

\subsection{Properness for affine actions on $\g=\oo(p,q+1)$} \label{subsec:proof-pseudo-Riem-contract-proper-g}

We now prove Theorem~\ref{thm:contract-proper}.\eqref{item:Hpq-contract-g}.
As in Sections \ref{subsec:o(p,q+1)} and~\ref{subsec:spacelike-contract-def}, we view $\g=\mathfrak{o}(p,q+1)$ as the space of Killing fields on~$\HH^{p,q}$.
As in Section~\ref{sec:metric-contraction}, we denote by $\mathcal{F}(\mathcal{O})$ the set of compact subsets of $\mathcal{O}\subset\Omega$, endowed with the Hausdorff topology for the restriction of the Hilbert metric $d_{\Omega}$ of Section~\ref{subsec:prop-conv}.
Note that in Definition~\ref{def:contract-spacelike}.\eqref{def:contract-Hpq-g}, up to restricting the set $\mathcal{O}\subset \Omega$ we may always choose it to be compact modulo $\rho(\Gamma)$.
Theorem~\ref{thm:contract-proper}.\eqref{item:Hpq-contract-g} therefore reduces to the following.

\begin{proposition} \label{prop:coarse-proj-g-pq}
Let $\Gamma$ be a discrete group, $\rho : \Gamma\to G=\OO(p,q+1)$ a representation with finite kernel and discrete image, preserving a properly convex open subset $\Omega$ of~$\HH^{p,q}$, and $u : \Gamma\to\g$ a $\rho$-cocycle. Suppose that 
$u$ is coarsely uniformly contracting in spacelike directions with respect to $(\rho, \Omega)$, with $\mathcal{O},Z,c<0$ and $c'$ as in Definition~\ref{def:contract-spacelike}.\eqref{def:contract-Hpq-g} and $\rho(\Gamma)\backslash \mathcal{O}$ compact.
Choose a continuous family of norms $(\Vert \cdot \Vert_x)_{x\in\mathcal{O}}$ on $T_x\HH^{p,q}$ which is $\rho$-invariant in the sense that $\Vert \rho(\gamma)_*v \Vert_{\rho(\gamma)\cdot x}= \Vert v \Vert_x$ for all $\gamma\in \Gamma$, all $x\in\mathcal{O}$, and all $v\in T_x\HH^{p,q}$. 
Then the map
\begin{eqnarray*}
\pi :\quad \g & \longrightarrow & \hspace{.45cm} \mathcal{F}(\mathcal{O})\\
Y & \longmapsto & \big\{ x\in \mathcal{O} ~|~ \Vert (Z-Y)(x)\Vert_{x}\ \mathrm{is}\ \mathrm{minimal}\big\}
\end{eqnarray*}
is well defined and takes any compact subset of~$\g$ to a compact subset of $\mathcal{F}(\mathcal{O})$.
Moreover, $\pi$ is equivariant with respect to the affine action of $\Gamma$ on~$\g$ via $(\rho,u)$ and the action of $\Gamma$ on $\mathcal{F}(\mathcal{O})$ via~$\rho$.
In particular, the affine action of $\Gamma$ on~$\g$ via $(\rho,u)$ is properly discontinuous.
\end{proposition}

The proof follows closely that of Proposition~\ref{prop:coarse-proj-g}.

\begin{proof}
Let $\mathcal{D} \subset \HH^{p,q}$ be a compact fundamental domain for the action of $\Gamma$ on $\mathcal{O}$ via~$\rho$.
By Lemma~\ref{lem:spacelike-Gamma-orbit}.\eqref{item:Kx}, there is a compact set $\mathcal{D}'\subset \mathcal{O}$ such that any $x\in \mathcal{D}$ sees any point of $\mathcal{O}\smallsetminus \mathcal{D}'$ in a spacelike direction belonging to $\mathcal{K}_x \subset T_{x}^{+1}\HH^{p,q}$, for some compact set $\mathcal{K}_x$ staying away from the null directions.
By compactness, there exists $R>0$ such that
$$|\mathsf{g}^{p,q}_x(w,v)| \leq R \,\Vert w \Vert_x$$
for all $w\in T_x\HH^{p,q}$ and $v\in\mathcal{K}_x$ with $x\in\mathcal{D}$, where $\mathsf{g}^{p,q}$ is the pseudo-Riemannian metric of $\HH^{p,q}$ from Section~\ref{subsec:prelim-Hpq}.
Consider $y\in\mathcal{O} \smallsetminus \mathcal{D}'$: it belongs to $\rho(\gamma)\cdot \mathcal{D}$ for some $\gamma\in\Gamma$.
By equivariance, for any $x\in\mathcal{D}$, the point $y$ sees~$x$ in a spacelike direction in $\mathcal{K}_{y} := \rho(\gamma)_*\mathcal{K}_{\rho(\gamma)^{-1}\cdot y}$, and $|\mathsf{g}^{p,q}_{y}(w,v)| \leq R \,\Vert w \Vert_{y}$ for all $w\in T_{y}\HH^{p,q}$ and $v\in\mathcal{K}_{y}$.
Applying Proposition~\ref{prop:deriv-dHpq} to the pair $(x, y)$, we obtain that for any 
vector field $V$ defined at both $x$ and~$y$,
$$ \frac{\D}{\D t}\Big|_{t=0} \, d_{\HH^{p,q}} \big( \exp_x(tV(x)),\;\exp_{y}(tV(y))\big) \geq - R\,\Vert V(x)\Vert_x - R\,\Vert V(y)\Vert_y \, ;
$$
in particular, if $V$ is coarsely $\kap$-lipschitz in spacelike directions (Definition~\ref{def:contract-spacelike}.\eqref{def:contract-Hpq-g}), then
\begin{equation} \label{eqn:ineq-c-lip-spacelike}
\kap\, d_{\HH^{p,q}}(x,y) + \kap' \geq - R\,\Vert V(x)\Vert_x - R\,\Vert V(y)\Vert_y .
\end{equation}
A vector field is coarsely $\kap$-lipschitz in spacelike directions if and only if its sum with any Killing field is.
Therefore, for any $Y\in \g$, by applying \eqref{eqn:ineq-c-lip-spacelike} to $V=Z-Y$, we find
$$R\,\Vert (Z-Y)(y)\Vert_y \geq |\kap| \, d_{\HH^{p,q}}(x,y) - \big(\kap'+R\,\Vert (Z-Y)(x)\Vert_x\big).$$
The term $c'+R\,\Vert (Z-Y)(x)\Vert_x$ is independent of~$y$ and remains bounded as $Y$ varies in a compact set, while the term $|\kap|\,d_{\HH^{p,q}}(x,y)$ is independent of~$Y$ and goes to $+\infty$ as $y$ goes to infinity in~$\mathcal{O}$, by Lemma~\ref{lem:spacelike-Gamma-orbit}.\eqref{item:flight}.
Since $Z$ is continuous, this shows that $\pi$ is well defined and takes compact sets to compact sets.

The equivariance of~$\pi$ follows from that of~$Z$: for any $\gamma\in\Gamma$ and $x\in\mathcal{O}$,
\begin{align*}
 & \Vert \big(Z - (\mathrm{Ad}(\rho(\gamma)) Y + u(\gamma))\big)(\rho(\gamma)\cdot x) \Vert_{\rho(\gamma)\cdot x} \\
 = & \left \Vert \rho(\gamma)_* (Z(x)) + u(\gamma)(\rho(\gamma)\cdot x) - (\mathrm{Ad}(\rho(\gamma)) Y +u(\gamma))(\rho(\gamma)\cdot x) \right \Vert_{\rho(\gamma)\cdot x} \\
 = & \Vert \rho(\gamma)_*((Z-Y)(x)) \Vert_{\rho(\gamma)\cdot x} = \Vert (Z-Y)(x) \Vert_{x}.
\end{align*}
By equivariance of~$\pi$, since the action of $\Gamma$ on $\mathcal{F}(\mathcal{O})$ via $\rho$ is properly discontinuous, so is the affine action of $\Gamma$ on~$\g$ via $(\rho,u)$.
\end{proof}

\subsection{Properness for actions on $G=\OO(p,q+1)$ by right-and-left multiplication} \label{subsec:proof-pseudo-Riem-contract-proper-G}

Theorem~\ref{thm:contract-proper}.\eqref{item:Hpq-contract-G} is an immediate consequence of Proposition~\ref{prop:coarse-proj-G} and of the following.

\begin{proposition} \label{prop:HtoX-contraction}
Let $\Gamma$ be a discrete group and $\rho : \Gamma\to G=\OO(p,q+1)$ a representation with finite kernel and discrete image, preserving a properly convex open subset $\Omega \neq \varnothing$ of $\HH^{p,q}$. 
Let $\rho':\Gamma\to G$ be a strongly irreducible representation
such that $\rho'(\Gamma)$ contains a proximal element.
If $\rho'$ is coarsely uniformly contracting in spacelike directions with respect to~$(\rho, \Omega)$ (Definition~\ref{def:contract-spacelike}.\eqref{def:contract-Hpq-G}), then $\rho'$ is coarsely uniformly contracting with respect to~$(\rho, \mathbb{X})$ (Definition~\ref{def:contract-equivar-deform-G}), where $\mathbb{X}:=G/(\OO(p)\times \OO(q+1))$ is the Riemannian symmetric space of~$G$ endowed with the $G$-invariant metric $d_{\mathbb{X}}$ of~\eqref{eqn:finsler}.
\end{proposition}

See Section~\ref{subsec:remind-lambda1} for the notions of proximality and strong irreducibility.

\begin{proof}
By Definition~\ref{def:contract-spacelike}.\eqref{def:contract-Hpq-G} of spacelike coarse uniform contraction, there exist a $\rho(\Gamma)$-invariant subset $\mathcal{O}\neq\varnothing$ of~$\Omega$ and a $(\rho,\rho')$-equivariant map $f : \mathcal{O}\to\HH^{p,q}$ that is coarsely $\Kap$-Lipschitz in spacelike directions, for some $\Kap<1$.

Consider the orbit $\mathcal{O}_{\mathbb{X}}:=\rho(\Gamma)\cdot x_0 \subset \mathbb{X}$ of the basepoint $x_0=[e]\in\mathbb{X}$.
We only need to find $\Kap'\in \mathbb{R}$ such that the map $f_{\mathbb{X}} : \mathcal{O}_{\mathbb{X}} \rightarrow \mathbb{X}$ taking every $\rho(\gamma)\cdot x_0$ to $\rho'(\gamma)\cdot x_0$ satisfies
$$d_{\mathbb{X}}\big(f_{\mathbb{X}}(\rho(\gamma_1)\cdot x_0),f_{\mathbb{X}}(\rho(\gamma_2)\cdot x_0)\big) \leq \Kap \, d_{\mathbb{X}}(\rho(\gamma_1)\cdot x_0,\rho(\gamma_2)\cdot x_0) + \Kap'$$
for all $\gamma_1, \gamma_2\in \Gamma$.
Since $f_{\mathbb{X}}$ is $(\rho, \rho')$-equivariant, we can restrict our attention to $\gamma_2=e$.
By the definitions \eqref{eqn:finsler} of $d_{\mathbb{X}}$ and~\eqref{eqn:muis} of $\mu_1$, it is therefore enough to show that for any $\gamma\in \Gamma$,
\begin{equation} \label{eqn:contract-macro-mu}
\mu_1(\rho'(\gamma)) \leq \Kap\,\mu_1(\rho(\gamma)) + \Kap' .
\end{equation}

We first check that for any $\gamma\in\Gamma$ with $\rho'(\gamma)$ proximal,
\begin{equation} \label{eqn:contract-macro-lambda}
\lambda_1(\rho'(\gamma)) \leq \Kap\,\lambda_1(\rho(\gamma)).
\end{equation}
For $\gamma\in\Gamma$ with $\rho'(\gamma)$ proximal, we have $\lambda_1(\rho'(\gamma))>0$.
Let $H^{\pm}\subset\RR^{p,q+1}$ be the sum of the generalized eigenspaces of $\rho'(\gamma)$ for eigenvalues of modulus $\neq \mathrm{e}^{\mp\lambda_1(\rho'(\gamma))}$.
Suppose by contradiction that $f(\mathcal{O})\subset H^+\cup H^-$.
Since $f(\mathcal{O})$ is $\rho'(\Gamma)$-invariant, so is the Zariski closure $\mathcal{Z}$ of $f(\mathcal{O})$.
Any irreducible component $\mathcal{Z}_i$ of $\mathcal{Z}$ is contained either in $H^+$ or in $H^-$, hence spans a proper subspace of $\RR^{p,q+1}$.
The union of these subspaces is preserved by $\rho'(\Gamma)$, contradicting strong irreducibility.
Therefore there exists $x\in \mathcal{O}$ such that $f(x) \notin H^+\cup H^-$, and Lemma~\ref{lem:orbit-growth}.\eqref{item:growth-prox} gives 
$$\lim_{n\to +\infty} \frac{1}{n} \, d_{\HH^{p,q}}(f(x), \rho'(\gamma)^n\cdot f(x)) = \lambda_1(\rho'(\gamma)).$$ 
On the other hand, by Lemma~\ref{lem:spacelike-Gamma-orbit}, for any large enough $n\in\NN$ the points $x$ and $\rho(\gamma^n)\cdot x$ are on a spacelike line, hence
$$d_{\HH^{p,q}}(f(x), \rho'(\gamma)^n\cdot f(x)) \leq \Kap \, d_{\HH^{p,q}}(x,\rho(\gamma)^n\cdot x) + C''$$
by assumption on~$f$, for some $C''\in\RR$ independent of~$n$.
Using Lemma~\ref{lem:orbit-growth}.\eqref{item:growth-general}, we obtain $\lambda_1(\rho'(\gamma)) \leq \Kap\,\lambda_1(\rho(\gamma))$, \ie \eqref{eqn:contract-macro-lambda} holds.

Let us now find $\Kap'\in \RR$ such that \eqref{eqn:contract-macro-mu} holds for all $\gamma\in\Gamma$.
Let $F \subset \Gamma$ and~$C_{\rho'} \geq 0$ be given by Fact~\ref{fact:AMS}, and let
$$\Kap' := C_{\rho'} + \Kap\max_{f\in F} \mu_1(\rho(f)) \in \RR.$$
For any $\gamma\in\Gamma$, we can find $f\in F$ such that $\rho'(\gamma f)$ is proximal and $\mu_1(\rho'(\gamma))  \leq \lambda_1(\rho'(\gamma f))+C_{\rho'}$.
By \eqref{eqn:contract-macro-lambda}, we have $\lambda_1(\rho'(\gamma f)) \leq \Kap\,\lambda_1(\rho(\gamma f))$.
For any $g\in G$ we have $\mu_1(g) = \log\,\Vert g\Vert$, hence $\mu_1(g) \geq \lambda_1(g)$ and $\mu_1(gg') \leq \mu_1(g) + \mu_1(g')$ for all $g,g'\in G$.
We deduce
$$\mu_1(\rho'(\gamma)) \leq \Kap\,\lambda_1(\rho(\gamma f)) + C_{\rho'} \leq \Kap\,\mu_1(\rho(\gamma f)) + C_{\rho'} \leq \Kap\,\mu_1(\rho(\gamma)) + \Kap'.\qedhere$$
\end{proof}

\begin{remark}
At the level of proofs, the parallel between $\g$ and $G$ broke down to some extent between Sections~\ref{subsec:proof-pseudo-Riem-contract-proper-g} and~\ref{subsec:proof-pseudo-Riem-contract-proper-G}.
In Section~\ref{subsec:proof-pseudo-Riem-contract-proper-g}, we were not able to use the spacelike-contracting vector fields on $\mathcal{O} \subset \Omega \subset \HH^{p,q}$ to produce contracting vector fields on a Hadamard (or even Finsler) manifold $\mathbb{X}$, to which we might have applied Proposition~\ref{prop:coarse-proj-g}; but we could mimic Proposition~\ref{prop:coarse-proj-g} by building an equivariant projection to $\mathcal{F}(\mathcal{O})$ using the pseudo-distance $d_{\HH^{p,q}}$ on~$\HH^{p,q}$. 
In Section~\ref{subsec:proof-pseudo-Riem-contract-proper-G}, starting from spacelike contracting maps from $\mathcal{O} \subset \Omega \subset \HH^{p,q}$ to $\HH^{p,q}$, we were not able to mimic Proposition~\ref{prop:coarse-proj-G} and build a well-behaved projection to $\mathcal{F}(\mathcal{O})$; but we could produce contracting maps in the symmetric space $\mathbb{X}=G/(\OO(p)\times \OO(q+1))$, endowed with an appropriate $G$-invariant Finsler metric, and apply Proposition~\ref{prop:coarse-proj-G} directly.
It is unclear to us whether or how the two arguments could be unified.
\end{remark}

\section{Uniform spacelike contraction for right-angled Coxeter groups} \label{sec:space-contr-Coxeter}

In this section we prove Theorems \ref{thm:proper-action-g} and~\ref{thm:proper-action-G} using the sufficient conditions for properness provided by Theorem~\ref{thm:contract-proper}.

Here is an outline of the argument: for a right-angled Coxeter group $\Gamma$ on $k$ generators,
we consider a certain natural one-parameter family $(\rho_t)_{t\in (-\infty,-1]}$ of deformations of the Tits canonical representation of $\Gamma$ into $\GL(k,\RR)$.
Vinberg's theory \cite{vin71} gives a natural properly convex domain $\mathcal{U}_t$ of $\PP(\RR^k)$ on which $\Gamma$ acts properly discontinuously via~$\rho_t$.
We truncate $\mathcal{U}_t$ to get a smaller properly convex $\rho_t(\Gamma)$-invariant domain $\Omega_t$ of $\PP(\RR^k)$ that lives in a copy $\Hpqv_t$ of $\HH^{p,q}$ for some $p,q\in\NN$ with $p+q+1=k$.
We may assume that the signature $(p,q)$ stays constant for $t$ in a certain open interval in $(-\infty,-1)$.
Up to conjugating everything to the standard copy of $\HH^{p,q}$, we may therefore meaningfully ask if certain equivariant maps between these domains $\Omega_t$ are uniformly contracting in spacelike directions: we prove that this is indeed the case (Proposition~\ref{prop:goal-final-section}) for some explicit piecewise projective maps, and we also show a vector-field counterpart.
This allows us to apply Theorem~\ref{thm:contract-proper} to prove Theorems \ref{thm:proper-action-g} (hence~\ref{thm:main}) and~\ref{thm:proper-action-G}.

\subsection{Basic setting}

We fix a right-angled Coxeter group
\begin{equation}\label{eqn:coxeter}
 \Gamma = \Gamma_S = \langle \gamma_1,\dots,\gamma_k ~|~ (\gamma_i \gamma_j)^{m_{i,j}} = 1\quad \forall i,j\rangle,
\end{equation}
where $m_{i,i}=1$ and $m_{i,j} = m_{j,i} \in\{ 2,\infty\}$ for all $i\neq j$.
Any subset $S'$ of the generating set $S = \{\gamma_1,\dots,\gamma_k\}$ defines a subgroup $\Gamma_{S'}$ of~$\Gamma$, with a presentation obtained from~\eqref{eqn:coxeter} by restricting to $i,j$ such that $\gamma_i, \gamma_j \in S'$.
We assume $\Gamma$ to be \emph{irreducible}, which means that $S$ cannot be written as a nontrivial disjoint union $S = S'\sqcup S''$ such that $\Gamma_{S'}$ and $\Gamma_{S''}$ commute.

If the number $k$ of generators is~$1$, then $\Gamma\simeq\ZZ/2\ZZ$ and Theorems \ref{thm:proper-action-g} and~\ref{thm:proper-action-G} are trivial.
If $k=2$, then $\Gamma$ is an infinite dihedral group; it admits properly discontinuous actions on the line~$\HH^1$, to which we can apply Proposition~\ref{prop:color} with $(m,p)=(0,1)$ and conclude using Propositions \ref{prop:coarse-proj-G} and~\ref{prop:coarse-proj-g} just as in Section~\ref{sec:riem-ex}.

From now on, we will assume $k\geq 3$.
In particular, $\Gamma$ is infinite.

\subsection{The canonical representation and its deformations} \label{subsec:remind-Coxeter}

The matrix $M_{-1} := (-\cos(\pi/m_{i,j}))_{1\leq i,j\leq k}$, with the convention $\pi/\infty=0$, is called the \emph{Gram matrix} of~$\Gamma$.
It defines a (possibly degenerate) symmetric bilinear form $\langle\cdot,\cdot\rangle_{-1}$ on~$\RR^k$.
Let $(e_1,\dots,e_k)$ be the standard basis of~$\RR^k$.
The \emph{canonical} (or \emph{geometric}) representation $\rho_{-1} : \Gamma\to\mathrm{Aut}(\RR^k,\langle\cdot,\cdot\rangle_t)\subset\GL(k,\RR)$, studied by Tits and others, is given by
$$\rho_{-1}(\gamma_i) ~:~ v \longmapsto v- 2\langle v, e_i \rangle_{-1} \: e_i, \quad\quad 1\leq i\leq k.$$

Note that $M_{-1} := \mathrm{Id}_k - N$, where $N = (N_{i,j})_{1\leq i , j \leq k}$ satisfies $N_{i,j}=1$ if $m_{i,j}=\infty$, and $0$ otherwise.
This matrix $N$ is irreducible with nonnegative entries. 
By the Perron--Frobenius theorem, there is a unique (up to scale) eigenvector $v_\PF$ of~$N$ with positive coordinates, corresponding to the highest eigenvalue $\lambda_{\PF}>0$.
In fact $\lambda_\PF \geq\sqrt{2}$ since, by irreducibility, $N$ contains a principal submatrix $\Big( \begin{smallmatrix} 0 & 1 & 0 \\ 1 & 0 & 1 \\ 0 & 1 & 0 \end{smallmatrix} \Big)$ or $\Big( \begin{smallmatrix} 0 & 1 & 1 \\ 1 & 0 & 1 \\ 1 & 1 & 0 \end{smallmatrix} \Big)$. 

One way to deform the canonical representation is to consider, for any $t\in (-\infty, -1]$, the matrix $M_t := \mathrm{Id}_{\RR^k} + tN$, \ie $M_t = ((M_t)_{i,j})_{1\leq i,j\leq k}$ with
$$(M_t)_{i,j} = \left\{ \begin{array}{ll}
1 & \text{if }m_{i,j}=1, \text{ \ie $i=j$},\\
0 & \text{if }m_{i,j}=2,\\
t\leq -1 & \text{if }m_{i,j}=\infty.
\end{array}\right.$$
This matrix $M_t$ still defines a symmetric bilinear form $\langle\cdot,\cdot\rangle_t$ on~$\RR^k$, and one can define a representation $\rho_t : \Gamma\to\mathrm{Aut}(\RR^k,\langle\cdot,\cdot\rangle_t)$ by
$$\rho_t(\gamma_i) : v \longmapsto v - 2\langle v, e_i \rangle_t \, e_i , \quad\quad 1\leq i\leq k.$$
Similar deformations were studied \eg in \cite{kra94}.

Note that $\mathrm{det}(M_t)$ is a polynomial in~$t$ which is not identically zero (consider $t=0$), hence it is nonzero outside some finite set $E$ of exceptional values of~$t$.
For any $t \in (-\infty, -1] \smallsetminus E$, the form $\langle\cdot,\cdot\rangle_t$ is nondegenerate.

The general theory of reflection groups developed by Vinberg applies to the representations~$\rho_t$.
For any $t \in (-\infty,-1] \smallsetminus E$, the convex cone
$$\widetilde{\Delta}_t = \{ v\in\RR^k ~|~ \langle v, e_i\rangle_t \leq 0 \ \,\forall i\}$$
has nonempty interior $\mathrm{Int}(\widetilde{\Delta}_t)$.
Indeed, $v_\PF\in\mathrm{Int}(\widetilde{\Delta}_t)$: for any~$i$ we have $\langle v_\PF,e_i\rangle_t<0$ since it is the $i$-th coordinate of $M_t (v_\PF)=(1+t\lambda_\PF)v_\PF$ and $t\lambda_\PF\leq -\sqrt{2}$.
Each generator $\gamma_i$ of~$\Gamma$ acts via~$\rho_t$ by reflection in the hyperplane $\mathrm{Ker}(\langle\cdot,e_i\rangle_t)$, called the $i$-th \emph{wall} of~$\widetilde{\Delta}_t$.
By \cite[Th.\,2 \&~5]{vin71}, the representation $\rho_t$ is faithful and discrete, and the open cone
$$\widetilde{\mathcal{U}}_t := \mathrm{Int} \big( \rho_t(\Gamma)\cdot\widetilde{\Delta}_t \big)$$
is convex. The action of $\Gamma$ on $\widetilde{\mathcal{U}}_t$ via~$\rho_t$ is properly discontinuous, with fundamental domain $\widetilde{\Delta}_t \cap \widetilde{\mathcal{U}}_t$.
The image $\mathcal U_t$ of $\widetilde{\mathcal U}_t$ in the projective space $\PP(\RR^k)$ is an open convex subset of $\PP(\RR^k)$, and the action of $\rho_t$ on $\mathcal U_t$ is properly discontinuous with fundamental domain $\Delta_t \cap \mathcal{U}_t$, where $\Delta_t$ is the image of $\widetilde \Delta_t$ in $\PP(\RR^k)$.
We shall call $\mathcal{U}_t$ the \emph{Tits--Vinberg domain} of $\rho_t(\Gamma)$.
By \cite[Prop.\,19]{vin71}, the representation $\rho_t$ is irreducible, hence $\mathcal{U}_t$ is properly convex, \ie its closure contains no projective line. 
In fact the following holds (see also \cite[Th.\,2.18]{mar17} or \cite{dgk-racg-cc}).

\begin{proposition} \label{prop:strong-irred}
For any $t\in (-\infty, -1]\smallsetminus E$, the representation $\rho_t : \Gamma \to \GL(k,\RR)$ is strongly irreducible.
\end{proposition}

\begin{proof}
Since $\Gamma$ is an irreducible right-angled Coxeter group on $k\geq 3$ generators, its Gram matrix contains a principal submatrix 
$\Big( \begin{smallmatrix} \;1 & \text{-}1 & \text{-}1 \\ \text{-}1 & \; 1 & \text{-}1 \\ \text{-}1 & \text{-}1 & \;1 \end{smallmatrix} \Big)$ 
or 
$\Big( \begin{smallmatrix} \;1 & \text{-}1 & \,0 \\ \text{-}1 & \; 1 & \text{-}1 \\ \,0 & \text{-}1 & \;1 \end{smallmatrix} \Big)$. 
The corresponding $3$-generator subgroup of $\Gamma$ is isomorphic to the group generated by the reflections in the sides of a triangle of $\HH^2$ with three ideal vertices (\resp two ideal vertices and one right angle).
In particular, $\Gamma$ contains a nonabelian free group on two generators.

For $t\in (-\infty, -1]\smallsetminus E$, suppose by contradiction that $\rho_t$ preserves a finite collection $\mathcal{V}=\{V_1,\dots, V_m\}$ of subspaces $0\subsetneq V_i \subsetneq \RR^k$. We may assume that each intersection $V_i\cap V_j$ is either $\{0\}$ or another~$V_\ell$. 
The action of $\Gamma$~on~$\RR^k$ via $\rho_t$ permutes the $V_i$; let $\Gamma_1<\Gamma$ be the finite-index subgroup fixing every~$V_i$.

We claim that $\dim(V_i)\geq 2$ for all~$i$.
Indeed, if $V_i$ were a line $\RR v$, then $\rho_t(\Gamma)\cdot v$ would span $\RR^k$ (since $\rho_t$ is irreducible), hence would contain a basis of~$\RR^k$, which would be a simultaneous eigenbasis for all elements of $\rho_t(\Gamma_1)$, making $\Gamma_1$ abelian.
But $\Gamma$ contains a nonabelian free group, hence cannot be virtually abelian: this shows that $\dim(V_i)\geq 2$ for all~$i$.

Up to reordering we may assume $r:=\dim V_1=\min_{V_i\in \mathcal{V}} \dim V_i \geq 2$.
For any $1\leq j\leq k$ we have $\rho_t(\gamma_j)\cdot V_1 \in \mathcal{V}$.
If $\rho_t(\gamma_j) \cdot V_1 \neq V_1$, since $\rho_t (\gamma_j)$ is a reflection in a hyperplane, we get that $V_1\cap \rho_t (\gamma_j)\cdot V_1 \in \mathcal{V}\cup \{0\}$  has dimension $r-1 > 0$, contradicting the minimality of~$r$.
Thus $\rho_t (\gamma_i )\cdot V_1 = V_1$ for all~$j$, contradicting the irreducibility of~$\rho_t$.
\end{proof}

\begin{remark} \label{rem:Delta}
For $t \in (-\infty,-1] \smallsetminus E$, the convex cone $\widetilde{\Delta}_t$ is the nonnegative span of the vectors $e'_1(t), \ldots, e'_k(t)$ given by the columns of the matrix $-M_t^{-1}$, \ie $\langle e'_i(t), e_j\rangle_t = - \delta_{ij}$ for all $1\leq i,j\leq k$.
Its projectivization $\Delta_t$ is a simplex with vertices $[e'_1(t)], \ldots, [e'_k(t)]$.
\end{remark}

\subsection{Construction of convex sets $\Omega_t$ in pseudo-Riemannian hyperbolic spaces} \label{subsec:construct-Cox}

We now fix an open interval $I \subset (-\infty, -1) \smallsetminus E$. 
For $t\in I$ the symmetric bilinear form $\langle\cdot,\cdot\rangle_t$ is nondegenerate of constant signature; since $\Gamma$ is infinite this signature has the form $(p,q+1)$ for some $p\geq 1$ and $q\geq 0$.
The group $\mathrm{Aut}(\RR^k,\langle\cdot,\cdot\rangle_t)$ identifies with $\OO(p,q+1)$ and we can consider the pseudo-Riemannian hyperbolic space
$$\Hpqv_t := \{ [v]\in\PP(\RR^k) ~|~ \langle v, v\rangle_t < 0 \},$$
defined like $\HH^{p,q}$ in Section~\ref{subsec:prelim-Hpq}.
The Tits--Vinberg domain $\mathcal{U}_t \subset \PP(\RR^k)$ is properly convex, but not contained in~$\Hpqv_t$ in general.
With the eventual goal of applying Theorem~\ref{thm:contract-proper}, we now look for a $\rho_t(\Gamma)$-invariant properly convex open subset $\Omega_t \subset \mathcal{U}_t$ contained in $\Hpqv_t$.

As in~\eqref{eqn:dual-in-P(V)}, using the nondegenerate symmetric bilinear form $\langle\cdot,\cdot\rangle_t$ we view the dual convex cone $\widetilde{\mathcal U}_t^*$ of $\widetilde{\mathcal U}_t$ as a subset of~$\RR^k$ (rather than of the dual vector space of~$\RR^k$):
$$\widetilde{\mathcal U}_t^* = \big\{ x \in \RR^k ~|~ \langle x, v \rangle_t <0\quad \forall v\in\overline{\widetilde{\mathcal{U}}_t}\big\}.$$
We also set
$$\widetilde{\Omega}_t := \widetilde{\mathcal{U}}_t \cap \widetilde{\mathcal{U}}_t^* , $$
and denote by $\mathcal U_t^*$ and $\Omega_t$ the respective images of $\widetilde{\mathcal U}^*_t$ and $\widetilde{\Omega}_t$ in $\PP(\RR^k)$.
For $e'_1(t), \ldots, e'_k(t)\in \RR^k$ as in Remark~\ref{rem:Delta}, let us consider the polyhedral cone
\begin{align}
 \widetilde \Sigma_t :=\;& \widetilde \Delta_t \cap \{ v \in \RR^k \,|\, \langle v, e_i'(t)\rangle_t \leq 0 \;\; \forall i \} \,=\, \widetilde \Delta_t \cap \textstyle \sum_{i=1}^k \RR^+ e_i  \notag \\
 =\;& \big\{ v \in \textstyle \sum_{i=1}^k \RR^+ e_i ~\big|~ \langle v, e_i \rangle_t \leq 0 \;\; \forall i \big\} . \label{eqn:defP}
 \end{align}
The image $\Sigma_t$ of $\widetilde \Sigma_t$ in $\PP(\RR^k)$ is obtained from the simplex $\Delta_t$ by truncating each vertex $[e_i'(t)]$ by the hyperplane dual to $[e_i'(t)]$.
We observe that $\Sigma_t$ is nonempty: for instance, $[v_\PF] \in \mathrm{Int}(\Sigma_t)$ since $v_\PF$ has positive entries.
 
\begin{lemma} \label{lem:Russian-dolls}
For any $t\in I$, the set $\Omega_t$ is nonempty and properly convex.
It is the intersection of all nonempty, $\rho_t(\Gamma)$-invariant properly convex open subsets of~$\mathcal{U}_t$, and satisfies $\Omega_t=\mathrm{Int}(\rho_t(\Gamma)\cdot\Sigma_t) \subset \Hpqv_t$. Moreover, $\Sigma_t\subset \Hpqv_t$.
\end{lemma}
 
\begin{remark}
Similar convex domains for reflection groups in pseudo-Rie\-mannian hyperbolic spaces $\HH^{p,q}$ were previously investigated, in somewhat different language, by Dyer~\cite{dye12} and Dyer--Hohlweg--Ripoll~\cite{dhr}, motivated by the study of Kac--Moody algebras.
\end{remark}
 
\begin{proof}[Proof of Lemma~\ref{lem:Russian-dolls}]
Let us first show that $\widetilde{\mathcal{U}}_t \cap -\widetilde{\mathcal{U}}_t^*=\varnothing$. By $\rho_t(\Gamma)$-invariance, it is enough to check $\widetilde{\Delta}_t \cap -\widetilde{\mathcal{U}}_t^*=\varnothing$. 
Points of $-\widetilde{\mathcal{U}}_t^*$ pair positively with the $e'_i \in \partial \widetilde{\mathcal{U}}_t$, \ie can be written $\sum_{i=1}^k s_i e_i$ with $s_i<0$.
If $x$ is such a point and $|s_j|=\min_{1\leq i \leq k} |s_i|$, then $\langle x,e_j\rangle_t=s_j+t\sum_{m_{r,j} = \infty} s_r >0$, showing $x \notin \widetilde{\Delta}_t$.
Thus $\widetilde{\mathcal{U}}_t \cap -\widetilde{\mathcal{U}}_t^*=\varnothing$.

It follows that $\overline{\widetilde{\mathcal{U}}_t} \cap -\overline{\widetilde{\mathcal{U}}_t^*}$ spans a $\rho_t(\Gamma)$-invariant linear subspace of~$\RR^k$ of dimension $<k$, hence reduces to $\{0\}$ by irreducibility of $\rho_t$. It also follows that~$\Omega_t$ (which might be properly convex \emph{or} empty) coincides with $\mathcal{U}_t \cap \mathcal{U}_t^*$.

Now, the group $\rho_t(\Gamma)$ contains elements which are proximal in $\PP(\RR^k)$, for instance $\rho_t(\gamma_i\gamma_j)$ for any $i\neq j$ with $m_{i,j}=\infty$: this is seen by a direct computation, as in the basis $\{e_1, \dots, e_k\}$ the matrix of $\rho_t(\gamma_i)$ is the identity minus twice the $i$-th row of $M_t$.
Let $\Lambda_t \subset \PP(\RR^k)$ be the closure of the set of attracting fixed points of all proximal elements of $\rho_t(\Gamma)$; necessarily $\Lambda_t \subset \overline{\mathcal{U}_t}$ and we can lift $\Lambda_t$ to a $\rho_t(\Gamma)$-invariant union $\widetilde{\Lambda}_t$ of rays in $\partial \widetilde{\mathcal{U}}_t$. Since $\rho_t(\Gamma)$ preserves both $\widetilde{\mathcal{U}}_t$ and $\widetilde{\mathcal{U}}^*_t$ while 
$\overline{\widetilde{\mathcal{U}}_t} \cap -\overline{\widetilde{\mathcal{U}}_t^*}=\{0\}$, it follows that $\widetilde{\Lambda}_t\subset \partial \widetilde{\mathcal{U}}^*_t$ also. Therefore $\widetilde{\mathcal{U}}_t$ and $\widetilde{\mathcal{U}}^*_t$ (not just their boundaries) intersect, otherwise the intersection of their closures would span a nonzero, $\rho_t(\Gamma)$-invariant proper subspace of~$\RR^k$, contradicting the irreducibility of~$\rho_t$.
The nonnegative span of $\widetilde{\Lambda}_t$ projects down to a properly convex subset $\mathcal{C}_t$ of $\overline{\mathcal{U}_t \cap \mathcal{U}_t^*}=\overline{\Omega_t}$, and $\mathrm{Int}(\mathcal{C}_t)\neq \varnothing$ by irreducibility of $\rho_t(\Gamma)$. In particular, 
$\Omega_t\neq \varnothing$.

To prove Lemma~\ref{lem:Russian-dolls}, it is enough to show that $\Sigma_t \subset \Hpqv_t$ and that
\begin{equation} \label{eqn:Russian-dolls}
\begin{array}{ccccccc}
\mathrm{Int}(\mathcal{C}_t) &
\underset{\mathrm{(i)}}{\subset} & \Omega_t &
\underset{\mathrm{(ii)}}{\subset} & \mathrm{Int}(\rho_t(\Gamma)\cdot\Sigma_t) & 
\underset{\mathrm{(iii)}}{\subset} & \mathrm{Int}(\mathcal{C}_t).
\end{array}
\end{equation}
 
Any nonempty $\rho_t(\Gamma)$-invariant convex open subset of~$\mathcal{U}_t$ contains $\mathrm{Int}(\mathcal{C}_t)$, hence (i) holds.
Next, as observed above, $\widetilde{\mathcal U}_t^* \subset \sum_{i=1}^k \RR_{>0}e_i$, hence $\widetilde \Delta_t \cap  \widetilde{\mathcal U}_t^* \subset \widetilde \Sigma_t$, and so $\Delta_t \cap  {\mathcal U}_t^*\subset \Sigma_t$. Since $\Omega_t \subset \rho_t(\Gamma) \Delta_t \cap  {\mathcal U}_t^*$ is open, (ii) follows.

In order to establish (iii), let us check that 
$\mathrm{Int}(\Sigma_t) \subset \mathcal{C}_t$.
Suppose by contradiction that $v = \sum_i s_i e_i$ satisfies $s_i>0>\langle v,e_i \rangle_t$ for all $1\leq i \leq k$ but $[v]\notin \mathcal{C}_t$. Let $\mathcal{C}'_t \subset \mathcal{U}_t$ be the smallest $\rho(\Gamma)$-invariant convex subset of $\mathcal{U}_t$ containing~$[v]$. Necessarily $[v]\in \partial \mathcal{C}'_t$, otherwise a closed uniform neighborhood of $\mathcal{C}_t$ in the Hilbert metric of $\mathrm{Int}(\mathcal{C}'_t)$ would already be convex~\cite[(18.12)]{busemann} and contain $[v]$.
In order to reach a contradiction, it is therefore enough to show that the positive span of the $\rho_t(\gamma_i) \cdot v = v - 2 \langle v, e_i \rangle_t \, e_i =:w_i$ contains $v$ in its interior. By substitution,
 \begin{equation}
 \sum_{i=1}^k \frac {s_i}{-\langle v,e_i \rangle_t} \, w_i = \Big ( 2+\sum_{i=1}^k \frac{s_i}{-\langle v,e_i\rangle_t}\Big ) \,v,
 \end{equation}
 hence\ $v$ does belong to the positive span of the $w_i$, which do span $\RR^k$ because $e_i=\frac{1}{2\langle v,e_i \rangle_t}(v-w_i)$. Thus, $\mathrm{Int}(\Sigma_t) \subset \mathcal{C}_t$.
 It follows that $\Sigma_t \subset \overline{\mathcal{C}}_t$, hence (iii) holds because $\mathrm{Int}(\overline{\mathcal{C}}_t)=\mathrm{Int}(\mathcal{C}_t)$.
 Thus all inclusions of~\eqref{eqn:Russian-dolls} are equalities.
 
Finally, let us prove $\Sigma_t \subset \Hpqv_t$. 
Any $v = \sum_{i=1}^k s_i e_i\in \widetilde{\Sigma}_t=(\RR^+)^k \cap \widetilde{\Delta}_t$ satisfies $\langle v, v\rangle_t = \sum_{i=1}^k s_i \, \langle v, e_i\rangle_t \leq 0$ since $s_i\geq 0 \geq \langle v, e_i\rangle_t$ by definition of~$\widetilde{\Delta}_t$.
There must exist positive coordinates $s_j, s_\ell>0$ such that $m_{j,\ell}=\infty$, otherwise $\langle v,v\rangle_t$>0.
One of the summands $s_j \, \langle v,e_j\rangle_t$ or $s_\ell \, \langle v,e_\ell\rangle_t$ must be negative, since $\langle v,e_j\rangle_t=s_j+t\sum_{m_{i,j}=\infty} s_i \leq s_j+t s_\ell$ and $\langle v,e_\ell\rangle_t=s_\ell+t\sum_{m_{i,\ell}=\infty} s_i \leq s_\ell+t s_j$ add up to a number $\leq (s_j+s_\ell) +t (s_\ell+s_j)<0$. Thus in fact $\langle v, v\rangle_t<0$, which proves $\Sigma_t \subset \Hpqv_t$.
\end{proof}

\begin{remark} \label{rem:moussong}
The region $\rho_t(\Gamma)\cdot\Sigma_t$, a union of compact subsets of $\Hpqv_t$, is closed in $\Hpqv_t$ if $\Gamma$ is word hyperbolic.
Indeed, the condition that no point of $\Delta_t$ with infinite stabilizer survives in $\Sigma_t$ can be shown to be equivalent to Moussong's criterion \cite{mou87} for hyperbolicity of~$\Gamma$. 
The action of $\Gamma$ via $\rho_t$ on this region is proper and cocompact, and indeed the subgroup $\rho_t(\Gamma)$ satisfies a notion of convex cocompactness in $\Hpqv_t$ recently introduced in~\cite{dgk-cc-Hpq} (see also \cite{dgk-proj-cc, dgk-racg-cc}).
\end{remark}

\subsection{Constructing equivariant contracting maps} \label{subsec:final}

We now choose a smooth family $(\iota_t : \Hpqv_t\rightarrow \HH^{p,q})_{t\in I}$ of isometries to the standard copy of~$\HH^{p,q}$.
This can be done for instance by writing $M_t=P_t-Q_t$ with $P_t$, $Q_t$ symmetric positive semidefinite, of respective ranks $p$ and $q+1$, commuting with~$M_t$; if $U_t\in\mathrm{SO}(k)$ takes the decomposition $\mathrm{Im}(P_t)\oplus \mathrm{Im}(Q_t)$ of $\RR^k$ to $\RR^p\oplus \RR^{q+1}$, then $\iota_t:=U_t(P_t^{1/2}-Q_t^{1/2})$ takes $\langle \cdot , \cdot \rangle_t$ to the standard symmetric bilinear form of $\RR^{p,q+1}$.
In our case, $U_t$ can be chosen independent of~$t$, as all matrices $M_t$ share the same eigendirections.

By conjugating by~$\iota_t$, for $t\in I$ we obtain representations
$$\rho^{\bullet}_t := \iota_t \circ\rho_t (\cdot) \circ \iota_t^{-1} : \Gamma \longrightarrow \OO(p,q+1)$$
which now all have the same target group.
Define the sets $(\Omega_t^{\bullet},\Sigma_t^{\bullet}, \mathcal{U}_t^\bullet, \Delta_t^\bullet):=\linebreak (\iota_t(\Omega_t), \iota_t(\Sigma_t), \iota_t(\mathcal{U}_t),\iota_t(\Delta_t))$ and the $\rho_t^\bullet$-cocycle
$$u_t := \frac{\D}{\D\tau}\Big |_{\tau=t}\, \rho^{\bullet}_{\tau} {\rho^{\bullet}_t}^{-1} : \Gamma \longrightarrow \oo(p,q+1).$$
Since $\rho_t^\bullet$ is strongly irreducible (Proposition~\ref{prop:strong-irred}), Theorems \ref{thm:proper-action-g} and~\ref{thm:proper-action-G} will be a direct consequence of Theorem~\ref{thm:contract-proper} and of the following.

\begin{proposition} \label{prop:goal-final-section}
For any $t<s$ in $I$, the representation $\rho^{\bullet}_{s}$ is coarsely uniformly contracting in spacelike directions with respect to $(\rho^{\bullet}_t,\Omega_t^{\bullet})$, and the $\rho_t^\bullet$-cocycle $u_t$ is uniformly contracting in spacelike directions with respect to~$\Omega_t^{\bullet}$ (Definition~\ref{def:contract-spacelike}).
\end{proposition}

In order to prove Proposition~\ref{prop:goal-final-section}, we now construct $(\rho^{\bullet}_t,\rho^{\bullet}_{\tau})$-equivariant maps $f_{t,\tau} : \Omega_t^{\bullet} \to\HH^{p,q}$ and $(\rho_\tau^\bullet,u_\tau)$-equivariant vector fields $Z_\tau$ on $\HH^{p,q}$ with appropriate contraction properties in spacelike directions, for $\tau\in[t,s]\subset I$.

Observe that $\langle M_{\varsigma}^{-1}M_{\tau} v, w\rangle_{\varsigma} = \langle v,w\rangle_\tau$ for all $v,w\in\RR^k$ and $\tau, \varsigma\in [t,s]$, by definition of the symmetric bilinear forms $\langle\cdot,\cdot\rangle_\varsigma$ and $\langle\cdot,\cdot\rangle_{\tau}$.
In particular, the matrix $M_{\varsigma}^{-1} M_{\tau}\in\GL(k,\RR)$ takes $\Delta_{\tau}$ to $\Delta_{\varsigma}$, and the reflection wall $\PP(\mathrm{Ker}(\langle\cdot,e_i\rangle_t))$ of $\rho_t(\gamma_i)$ to the reflection wall $\PP(\mathrm{Ker}(\langle\cdot,e_i\rangle_{\tau}))$ of $\rho_{\tau}(\gamma_i)$ for any $1\leq i\leq k$.
We can therefore extend the $(M_{\varsigma}^{-1}M_{\tau})|_{\Delta_\tau}$ to a family of $(\rho_\tau, \rho_{\varsigma})$-equivariant maps
$$\Phi_{\tau,\varsigma} : \rho_\tau(\Gamma)\cdot\Delta_\tau \longrightarrow \rho_{\varsigma}(\Gamma)\cdot\Delta_{\varsigma}$$
for $\tau, \varsigma \in [t,s]$; these maps are continuous along the walls $\PP(\mathrm{Ker}(\langle\cdot,e_i\rangle_\tau))$, hence induce homeomorphisms $\mathcal{U}_\tau\rightarrow \mathcal{U}_{\varsigma}$; they depend smoothly on the pair $(\tau, \varsigma)$ and satisfy the compatibility relation $\Phi_{\tau', \tau''}\circ \Phi_{\tau,\tau'}=\Phi_{\tau,\tau''}$ for all $\tau, \tau', \tau''\in [t,s]$.
The maps
$$f_{\tau,\varsigma} := \iota_{\varsigma} \circ \Phi_{\tau,\varsigma}|_{\Omega_\tau} \circ \iota_\tau^{-1} : \; \Omega_\tau^{\bullet} \longrightarrow \PP(\RR^k)$$
are then $(\rho^{\bullet}_\tau,\rho^{\bullet}_{\varsigma})$-equivariant, continuous, and by construction $f_{\tau,\tau}=\mathrm{Id}_{\Omega_\tau^{\bullet}}$ for every $\tau$.
The family $(f_{t,\tau})_{\tau\in [t,s]}$ is smooth, and for every $\tau\in [t,s]$ the vector field $Z_\tau$ defined on $\Omega_\tau^{\bullet}$ by
$$Z_\tau(x) := \frac{\D}{\D\varsigma}\Big |_{\varsigma=\tau}\, f_{\tau,\varsigma}(x)$$
is continuous and $(\rho^{\bullet}_\tau,u_\tau)$-equivariant by Lemma~\ref{lem:derivate-Hpq}.
Moreover,
\begin{equation}\label{eqn:drift}
Z_\tau(f_{t,\tau}(x))=\frac{\D}{\D\varsigma}\Big |_{\varsigma=\tau}\, f_{t,\varsigma}(x)\quad \text{if } f_{t,\tau}(x)\in\Omega_\tau^\bullet. 
\end{equation}

By Definition~\ref{def:contract-spacelike} of spacelike uniform contraction, in order to prove Proposition~\ref{prop:goal-final-section}, it is enough to establish the following.

\begin{proposition} \label{prop:RACG-maps-contract}
For any $t<s$  in $I$, there exists $\kap<0$ such that
\begin{enumerate}[label=(\alph*)]
  \item $Z_\tau$ is $\kap$-lipschitz in spacelike directions on $\Omega_\tau^{\bullet}$ for any $\tau\in [t,s]$,
  \item $f_{t,s}|_{\Omega_t^\bullet}$ takes values in $\HH^{p,q}$, and is coarsely $\mathrm{e}^{\kap (s-t)}$-Lipschitz in spacelike directions on $\mathcal{O}_t^{\bullet}:=\rho_t^{\bullet}(\Gamma)\cdot [\iota_t(v_\PF)] \subset \Omega_t^{\bullet}$.
\end{enumerate}
\end{proposition}

\subsection{Proof of Proposition~\ref{prop:RACG-maps-contract}} \label{subsec:proof-spacelike-contract-Cox}

For $\tau \in [t,s]$, let $\langle \cdot,\cdot \rangle_\tau$, $\langle \cdot,\cdot \rangle_{0}$, and $\llangle \cdot,\cdot \rrangle_\tau$ be the symmetric bilinear forms on~$\RR^k$ defined by the matrices $M_\tau$, $\mathrm{Id}$, and $M_\tau^{-1}$ respectively.
We have $\langle v,w\rangle_\tau=\llangle M_\tau v,M_\tau w\rrangle_\tau$ for all $v,w\in\RR^k$, hence the following diagram commutes:
\begin{equation} \label{eqn:diagram}
\begin{tikzcd}
\left (\RR^{p,q+1},\langle \cdot,\cdot \rangle_{p,q+1} \right )
\arrow[swap]{d}{\text{\small $f_{t,\tau}$}}      &   
\left (\RR^k,\langle \cdot,\cdot \rangle_{t} \right )
\arrow[swap]{d}{\text{\small $\Phi_{t,\tau}$}} 
\arrow[swap]{r}{M_{t}}  
\arrow{l}{\iota_{t}}      & 
\left (\RR^k,\llangle \cdot,\cdot \rrangle_{t}) \right )
\arrow{d}{\text{\small 
$J_{t,\tau}$
}}    \\ 
\left (\RR^{p,q+1},\langle \cdot,\cdot \rangle_{p,q+1}\right )  &
\left (\RR^k,\langle \cdot,\cdot \rangle_{\tau} \right )
\arrow{r}{M_{\tau}}   
\arrow[swap]{l}{\iota_{\tau}}&
\left (\RR^k,\llangle \cdot,\cdot \rrangle_{\tau} \right )
\end{tikzcd}
\end{equation}
where $J_{t,\tau}:=M_{\tau}\Phi_{t,\tau}M_t^{-1}$ satisfies by construction
\begin{equation} \label{eqn:fixed-chamber}
 J_{t,\tau}|_{M_t(\widetilde{\Delta}_t)}=\mathrm{Id}_{M_t(\widetilde{\Delta}_t)}.
\end{equation}
The horizontal arrows of~\eqref{eqn:diagram} are isometries, but not the vertical ones.
The symmetric bilinear form $\llangle\cdot,\cdot\rrangle_\tau$ still has signature $(p,q+1)$, and we can consider the corresponding pseudo-Riemannian hyperbolic space
$$\Hpqvv_\tau := \{ [v]\in\PP(\RR^k) ~|~ \llangle v, v\rrangle_\tau < 0 \} = M_\tau \, \Hpqv_\tau, $$
with boundary $\partial \Hpqvv_\tau = M_\tau \, \partial\Hpqv_\tau$ (we see the matrix $M_{\tau}\in\GL(k,\RR)$ as acting both on $\RR^k$ and on $\PP(\RR^k)$).
The key point is the following observation.

\begin{lemma} \label{lem:quadric-expansion}
For any $t<s$ in~$I$, there exists $\kap<0$ such that, as $\tau\in [t,s]$ increases, the boundary of \emph{$\Hpqvv_{\tau}$} expands outwards everywhere with normal velocity $\geq -c/2>0$, for the spherical metric \eqref{eqn:spherical-metric} on $\PP(\RR^k)$.
\end{lemma}

\begin{proof}
Let $\mathrm{Null}(M_{\tau}^{-1}) := \{ v\in\RR^k \,|\, \llangle v, v\rrangle_\tau = 0\}$ be the preimage of $\partial\Hpqvv_\tau$ and $\mathrm{Null}(M_{\tau}) := \{ v\in\RR^k \,|\, \langle v, v\rangle_\tau = 0\}$ the preimage of $\partial\Hpqv_\tau$ in~$\RR^k$.
The intersection of $\mathrm{Null}(M_{\tau}^{-1})$ with the $\langle \cdot,\cdot\rangle_0$-unit (Euclidean) sphere $\mathbb{S}$ is the $0$-level set, in $\mathbb{S}$, of the function $v\mapsto \langle v, M_{\tau}^{-1} v \rangle_{0}$.
Since $\mathbb{S}$ is compact, the desired uniform expansion property for $\partial\Hpqvv_{\tau}$ can therefore be written simply:
\begin{equation} \label{eqn:shrink}
\frac{\D}{\D \varsigma}\Big|_{\varsigma=\tau}\, \langle v, M_{\varsigma}^{-1} v \rangle_{0}<0 \text{ for all } v \in \mathrm{Null}(M_{\tau}^{-1}) \cap \mathbb{S}.
\end{equation}

Note that \eqref{eqn:shrink} is equivalent to
\begin{equation} \label{eqn:expand}
\frac{\D}{\D \varsigma}\Big|_{\varsigma=\tau}\,\big \langle w, M_{\varsigma} w \big \rangle_{0}>0 \text{ for all } w \in \mathrm{Null}(M_{\tau}) \cap \mathbb{S}.
\end{equation}
Indeed, $\frac{\D}{\D \varsigma}\big|_{\varsigma=\tau}\, M_{\varsigma}^{-1} = -M_\tau^{-1} \big ( \frac{\D}{\D \varsigma}\big|_{\varsigma=\tau}\, 
M_{\varsigma} \big ) M_\tau^{-1}$, hence under the change of variable $w=M_{\tau}^{-1}v$, condition \eqref{eqn:shrink} becomes \eqref{eqn:expand}. 
(In other words, expansion of $\partial\Hpqvv_{\tau}$ is equivalent to shrinking of $\partial\Hpqv_{\tau}$.)

But $M_\tau=\mathrm{Id}+\tau N$ and $\tau <-1$, and so condition \eqref{eqn:expand} is clearly satisfied: $w\in \mathrm{Null}(M_{\tau})$ means $\langle w, N w \rangle_{0} = \frac{-1}{\tau} \langle w, w \rangle_{0}$, and therefore implies
$\frac{\D}{\D \varsigma}\big|_{\varsigma=\tau}\, \langle w, M_{\varsigma}w \rangle_{0} =  \langle w, N w \rangle_{0}>0$.
\end{proof}

\begin{proof}[Proof of Proposition~\ref{prop:RACG-maps-contract}.(a)]
Consider $t<s$ in~$I$.
By Lemma~\ref{lem:quadric-expansion}, there exists $\kap<\nolinebreak 0$ such that, as $\tau\in [t,s]$ increases, the boundary $\partial \Hpqvv_{\tau}$ expands outwards everywhere with normal velocity $\geq -c/2>0$.
Since horizontal arrows of~\eqref{eqn:diagram} are isometries, using \eqref{eqn:fixed-chamber} and Lemma~\ref{lem:compare-general}, this shows that the vector field $Z_\tau=\frac{\D}{\D\varsigma}\big |_{\varsigma=\tau} f_{\tau, \varsigma}$ is $\kap$-lipschitz in spacelike directions in restriction to (any convex subset of) $\Delta_\tau^{\bullet}\cap \HH^{p,q}$, for any $\tau\in [t,s]$.

Since $Z_\tau$ is $(\rho_\tau^{\bullet},u_\tau)$-equivariant (Definition~\ref{def:contract-equivar-deform-g}) and since the sum of a $\kap$-lipschitz vector field and a Killing field is still $\kap$-lipschitz (Proposition~\ref{prop:deriv-dHpq}), the vector field $Z_\tau$ is also $\kap$-lipschitz in spacelike directions in restriction to $\rho_\tau^{\bullet}(\gamma)\cdot\Delta_\tau^{\bullet}\cap \HH^{p,q}$ for any $\gamma\in\Gamma$.

From this we see that $Z_\tau$ is $\kap$-lipschitz in spacelike directions on $\Omega_\tau^{\bullet}$.
Indeed, $\Sigma_\tau^{\bullet}$ is a fundamental domain for the $\rho_t^{\bullet}$-action of $\Gamma$ on the closure of the properly convex set $\Omega_\tau^{\bullet}$ in $\HH^{p,q}$, by Lemma~\ref{lem:Russian-dolls}.
If $x,y\in\Omega_\tau^{\bullet}$ are on a spacelike line, then we can find points $x=x_0,x_1,\dots,x_m=y$ in $\Omega_\tau^\bullet$, lined up in this order, such that for any $1\leq i\leq m$ there exists $\eta_i\in\Gamma$ with $[x_{i-1},x_i]\subset\rho_\tau^{\bullet}(\eta_i)\cdot\Sigma_\tau^{\bullet}$.
Since $Z_\tau$ is continuous, and $\kap$-lipschitz in spacelike directions on each $\rho_\tau^{\bullet}(\eta_i)\cdot\Sigma_\tau^{\bullet}$, Proposition~\ref{prop:deriv-dHpq} applied to each $[x_{i-1},x_i]$ yields
\begin{align*}
 & \frac{\D}{\D r}\Big|_{r=0} \, d_{\HH^{p,q}}\left ( \exp_x(rZ_\tau(x)), \exp_y(rZ_\tau(y)) \right ) \\
 = &   \sum_{i=1}^m \, \frac{\D}{\D r}\Big|_{r=0} \, d_{\HH^{p,q}}\big( \exp_{x_{i-1}}(rZ_\tau(x_{i-1})), \exp_{x_{i}}(rZ_\tau(x_{i}))\big)  \\
 \leq & \,  \kap \sum_{i=1}^m \, d_{\HH^{p,q}}(x_{i-1},x_i) = \kap \, d_{\HH^{p,q}}(x,y). \qedhere  \end{align*}
 \end{proof}

\begin{proof}[Proof of Proposition~\ref{prop:RACG-maps-contract}.(b)]
Due to~\eqref{eqn:fixed-chamber}, Lemma~\ref{lem:quadric-expansion} also shows that $f_{t,s}$ takes values in $\HH^{p,q}$ on the whole set $\Delta_t^\bullet \cap \HH^{p,q}$, hence also on $\mathcal{U}_t^\bullet \cap\nolinebreak\HH^{p,q}$ by equivariance, and a fortiori on its subset $\Omega_t^\bullet$.
In order to prove Proposition~\ref{prop:RACG-maps-contract}.(b), we observe that $\Phi_{t,\tau}|_{\Delta_t}=M_\tau^{-1}M_t$ always fixes the point $[v_\PF]$, for any $\tau\in [t,s]$.
Therefore, $x_\tau:=\iota_\tau([v_\PF])\in\mathrm{Int}(\Sigma_\tau^{\bullet})\subset \Omega_\tau^\bullet$ satisfies 
\begin{equation} \label{eqn:track}
f_{t,\tau}(x_t)=x_{\tau} ~\text{ for all }~ \tau \in [t,s]. 
\end{equation}

Since $x_{\tau}\in\Omega^\bullet_\tau$ for all $\tau$, by Lemma~\ref{lem:spacelike-Gamma-orbit}.\eqref{item:Kx}, there exists a finite subset $F$ of~$\Gamma$ such that for any $\tau \in [t,s]$ and any $\gamma\in\Gamma\smallsetminus F$, the point $x_\tau$ sees $\rho_{\tau}^{\bullet}(\gamma)\cdot x_\tau$ in a spacelike direction.
Then
\begin{align*}
& \frac{\D}{\D \varsigma}\Big|_{\varsigma=\tau} \: d_{\HH^{p,q}}\big ( f_{t,\varsigma}(x_t),f_{t,\varsigma}(\rho_t^{\bullet}(\gamma)\cdot x_t)\big )\\
=\ & \frac{\D}{\D r}\Big|_{r=0} \: d_{\HH^{p,q}}\big ( \exp_{x_\tau}(rZ_{\tau}(x_{\tau})), \exp_{\rho_\tau^\bullet(\gamma)\cdot x_\tau}(rZ_{\tau}(\rho_{\tau}^{\bullet}(\gamma)\cdot x_{\tau})) \big )\\
\leq\ & \kap \, d_{\HH^{p,q}}(x_{\tau},\rho_{\tau}^{\bullet}(\gamma)\cdot x_{\tau})) = \kap \, d_{\HH^{p,q}}\big ( f_{t,\tau}(x_t),f_{t,\tau}(\rho_t^{\bullet}(\gamma)\cdot x_t)\big ),
\end{align*}
where we use \eqref{eqn:drift}, Proposition~\ref{prop:RACG-maps-contract}.(a), and \eqref{eqn:track} in this order.
Integrating over $\tau \in [t,s]$, we obtain
$$d_{\HH^{p,q}}(f_{t,s}(x_t),f_{t,s}(\rho_t^{\bullet}(\gamma)\cdot x_t)) \leq \mathrm{e}^{\kap (s-t)} \, d_{\HH^{p,q}}(x_t,\rho_t^{\bullet}(\gamma)\cdot x_t) $$
for all $\gamma\in\Gamma\smallsetminus F$.
Up to an additive constant, this is still true of all $\gamma\in \Gamma$.
In other words, $f_{t,s}$ is coarsely $\mathrm{e}^{\kap (s-t)}$-Lipschitz in spacelike directions on $\mathcal{O}^\bullet_t=\rho_t^{\bullet}(\Gamma)\cdot x_t$.
\end{proof}



\begin{thebibliography}{DGKLM}

\bibitem[Ag]{ago13}
\textsc{I. Agol}, with an appendix by \textsc{I. Agol, D. Groves, J. Manning}, \textit{The virtual Haken conjecture}, Doc. Math. 18 (2013), p.~1045--1087.

\bibitem[AMS1]{ams95}
\textsc{H. Abels, G. A. Margulis, G. A. Soifer}, \textit{Semigroups containing proximal linear maps}, Israel J. Math.~91 (1995), p.~1--30.

\bibitem[AMS2]{ams97}
\textsc{H. Abels, G. A. Margulis, G. A. Soifer}, \textit{Properly discontinuous groups of affine transformations with orthogonal linear part}, C. R. Acad. Sci. Paris~324 (1997), p.~253--258.

\bibitem[AMS3]{ams02}
\textsc{H. Abels, G. A. Margulis, G. A. Soifer}, \textit{On the Zariski closure of the linear part of a properly discontinuous group of affine transformations}, J.\ Differential Geom.~60 (2002), p.~315--344.

\bibitem[AMS4]{ams11}
\textsc{H. Abels, G. A. Margulis, G. A. Soifer}, \textit{The linear part of an affine group acting properly discontinuously and leaving a quadratic form invariant}, Geom.\ Dedicata~153 (2011), p.~1--46.

\bibitem[AMS5]{ams12}
\textsc{H. Abels, G. A. Margulis, G. A. Soifer}, \textit{The Auslander conjecture for dimension less than~$7$}, preprint, arXiv:1211.2525.

\bibitem[Au]{aus64}
\textsc{L. Auslander}, \textit{The structure of compact locally affine manifolds}, Topology~3 (1964), p.~131--139.

\bibitem[Be]{ben97}
\textsc{Y. Benoist}, \textit{Propri\'et\'es asymptotiques des groupes lin\'eaires}, Geom. Funct. Anal.~7 (1997), p.~1--47.

\bibitem[BW]{bw12}
\textsc{N. Bergeron, D. T. Wise}, \textit{A boundary criterion for cubulation}, Amer. J. Math.~134 (2012), p.~843--859.

\bibitem[BB]{bb97}
\textsc{M. Bestvina, N. Brady}, \textit{Morse theory and finiteness properties of groups}, Invent. Math.~129 (1997), p.~123--139.

\bibitem[Bi]{bie11-12}
\textsc{L. Bieberbach}, \textit{\"Uber die Bewegungsgruppen der Euklidischen R\"aume}, Math. Ann.~70 (1911), p.~297--336, and Math. Ann.~72 (1912), p.~400--412.

\bibitem[Bu]{busemann}
\textsc{H. Busemann}, \textit{The geometry of geodesics}, Academic Press Inc., New York,1955.

\bibitem[CDG]{cdg16}
\textsc{V. Charette, T. Drumm, W. M. Goldman}, \textit{Proper affine deformations of two-generator Fuchsian groups}, Transform. Groups~21 (2016), p.~953--1002.

\bibitem[DGK1]{dgk16}
\textsc{J. Danciger, F. Gu\'eritaud, F. Kassel}, \textit{Geometry and topology of complete Lorentz spacetimes of constant curvature}, Ann. Sci. \'Ec. Norm. Sup.~49 (2016), p.~1--56.

\bibitem[DGK2]{dgk-strips}
\textsc{J. Danciger, F. Gu\'eritaud, F. Kassel}, \textit{Margulis spacetimes via the arc complex}, Invent. Math.~204 (2016), p.~133--193.

\bibitem[DGK3]{dgk-cc-Hpq}
\textsc{J. Danciger, F. Gu\'eritaud, F. Kassel}, \textit{Convex cocompactness in pseudo-Riemannian hyperbolic spaces}, Geom. Dedicata~192 (2018), p.~87--126, special issue \textit{Geometries: A Celebration of Bill Goldman's 60th Birthday}.

\bibitem[DGK4]{dgk-proj-cc}
\textsc{J. Danciger, F. Gu\'eritaud, F. Kassel}, \textit{Convex cocompact actions in real projective geometry}, preprint, arXiv:1704.08711.

\bibitem[DGK5]{dgk-parab}
\textsc{J. Danciger, F. Gu\'eritaud, F. Kassel}, \textit{Margulis spacetimes with parabolic elements}, in preparation.

\bibitem[DGKLM]{dgk-racg-cc}
\textsc{J. Danciger, F. Gu\'eritaud, F. Kassel, G.-S. Lee, L. Marquis}, \textit{Convex cocompactness for Coxeter groups}, in preparation.

\bibitem[DZ]{dz}
\textsc{J. Danciger, T. Zhang}, \textit{Affine actions with Hitchin linear part}, preprint, arXiv:1812.03930.

\bibitem[DJ]{dj00}
\textsc{M. W. Davis, T. Januszkiewicz}, \textit{Right-angled Artin groups are commensurable with right-angled Coxeter groups}, J. Pure Applied Alg.~153 (2000), p.~229--235.

\bibitem[DT]{dt16}
\textsc{B. Deroin, N. Tholozan}, \textit{Dominating surface group representations by Fuchsian ones}, Int. Math. Res. Not.~2016 (2016), p.~4145--4166.

\bibitem[De]{der00}
\textsc{A. Derdzinski}, \textit{Einstein metrics in dimension four}, in \textit{Handbook of differential geometry}, vol.~I, p.~419--70, North-Holland, Amsterdam, 2000.

\bibitem[Dr]{dru92}
\textsc{T. Drumm}, \textit{Fundamental polyhedra for Margulis space-times}, Topology~21 (1992), p.~677--683.

\bibitem[Dy]{dye12}
\textsc{M. J. Dyer}, \textit{Imaginary cone and reflection subgroups of Coxeter groups}, arXiv:1210.5206.

\bibitem[DHR]{dhr}
\textsc{M. Dyer, C. Hohlweg, V. Ripoll}, \textit{Imaginary cones and limit roots of infinite Coxeter groups}, Math. Z.~284 (2016), p.~715--780.

\bibitem[FG]{fg83}
\textsc{D. Fried, W. M. Goldman}, \textit{Three-dimensional affine crystallographic groups}, Adv. Math.~47 (1983), p.~1--49.

\bibitem[Gh]{ghy95}
\textsc{\'E. Ghys}, \textit{D\'eformations des structures complexes sur les espaces homog\`enes de~$\SL_2(\CC)$}, J. Reine Angew. Math.~468 (1995), p.~113--138.

\bibitem[GlMo]{gm16}
\textsc{O. Glorieux, D. Monclair}, \textit{Critical exponent and Hausdorff dimension for quasi-Fuchsian AdS manifolds}, preprint, arXiv:1606.05512.

\bibitem[GoMa]{gm00}
\textsc{W. M. Goldman, G. A. Margulis}, \textit{Flat Lorentz $3$-manifolds and cocompact Fuchsian groups}, in \textit{Crystallographic groups and their generalizations (Kortrijk, 1999)}, p.~135--145, Contemp. Math.~262, American Mathematical Society, Providence, RI, 2000.


\bibitem[GoK]{gk83}
\textsc{W. M. Goldman, Y. Kamishima}, \textit{The fundamental group of a compact flat Lorentz space form is virtually polycyclic}, J.\ Differential Geom.~19 (1983), p.~233--240.

\bibitem[GLM]{glm09}
\textsc{W. M. Goldman, F. Labourie, G. A. Margulis}, \textit{Proper affine actions and geodesic flows of hyperbolic surfaces}, Ann.\ of Math.~170 (2009), p.~1051--1083.

\bibitem[GuK]{gk17}
\textsc{F. Gu\'eritaud, F. Kassel}, \textit{Maximally stretched laminations on geometrically finite hyperbolic manifolds}, Geom. Topol.~21 (2017), p.~693--840.

\bibitem[GKW]{gkw15}
\textsc{F. Gu\'eritaud, F. Kassel, M. Wolff}, \textit{Compact anti-de Sitter 3-manifolds and folded hyperbolic structures on surfaces}, Pacific J. Math.~275 (2015), p.~325--359.

\bibitem[HW1]{hw08}
\textsc{F. Haglund, D. T. Wise}, \textit{Special cube complexes}, Geom. Funct. Anal.~17 (2008), p.~1551--1620.

\bibitem[HW2]{hw10}
\textsc{F. Haglund, D. T. Wise}, \textit{Coxeter groups are special}, Adv. Math.~224 (2010), p.~1890--1903.

\bibitem[JS]{js03}
\textsc{T. Januszkiewicz, J. \'Swi\c{a}tkowski}, \textit{Hyperbolic Coxeter groups of large dimension}, Comment.\ Math.\ Helv.~78 (2003), p.~555--583.

\bibitem[KM]{km12}
\textsc{J. Kahn, V. Markovic}, \textit{Immersing almost geodesic surfaces in a closed hyperbolic three manifold}, Ann. of Math.~175 (2012), p.~1127--1190.

\bibitem[K1]{kas08}
\textsc{F. Kassel}, \textit{Proper actions on corank-one reductive homogeneous spaces}, J. Lie Theory~18 (2008), p.~961--978.

\bibitem[K2]{kasPhD}
\textsc{F. Kassel}, \textit{Quotients compacts d'espaces homog\`enes r\'eels ou $p$-adiques}, PhD thesis, Universit\'e Paris-Sud, 2009, available at \url{www.ihes.fr/}$\sim$\url{kassel/}.

\bibitem[Ko]{kob98}
\textsc{T. Kobayashi}, \textit{Deformation of compact Clifford--Klein forms of indefinite-Riemannian homogeneous manifolds}, Math. Ann.~310 (1998), p.~395--409.

\bibitem[KN]{kono}
\textsc{S. Kobayashi, K. Nomizu}, \textit{Foundations of differential geometry II}, Interscience tracts in mathematics 15 (1969), Chapter VIII.

\bibitem[Kr]{kra94}
\textsc{D. Krammer}, \textit{The conjugacy problem for Coxeter groups}, PhD thesis, Universiteit Utrecht, 1994, published in Groups Geom. Dyn.~3 (2009), p.~71--171.

\bibitem[KR]{kr85}
\textsc{R. S. Kulkarni, F. Raymond}, \textit{3-dimensional Lorentz space-forms and Seifert fiber spaces}, J. Differential Geom.~21 (1985), p.~231--268.

\bibitem[La]{lab01}
\textsc{F. Labourie}, \textit{Fuchsian affine actions of surface groups}, J. Diff. Geom.~59 (2001), p.~15--31.

\bibitem[LL]{ll17}
\textsc{G. Lakeland, C. J. Leininger}, \textit{Strict contractions and exotic $\SO_0(d,1)$ quotients}, J. Lond. Math. Soc.~96 (2017), p.~642--662.

\bibitem[Li]{liu13}
\textsc{Y. Liu}, \textit{Virtual cubulation of nonpositively curved graph manifolds}, J. Topol.~6 (2013), p.~793--822.

\bibitem[LM]{lm85}
\textsc{A. Lubotzky, A. Magid}, \textit{Varieties of representations of finitely generated groups}, Memoirs of the AMS, vol.~336, American Mathematical Society, 1985.

\bibitem[M1]{mar83}
\textsc{G. A. Margulis}, \textit{Free completely discontinuous groups of affine transformations} (in Russian), Dokl. Akad. Nauk SSSR~272 (1983), p.~785--788.

\bibitem[M2]{mar84}
\textsc{G. A. Margulis}, \textit{Complete affine locally flat manifolds with a free fundamental group}, J. Soviet Math.~1934 (1987), p.~129--139, translated from Zap.~Naucha. Sem. Leningrad. Otdel. Mat. Inst. Steklov (LOMI)~134 (1984), p.~190--205.

\bibitem[Ma]{mar17}
\textsc{L. Marquis}, \textit{Coxeter group in Hilbert geometry}, Groups Geom. Dyn.~11 (2017), p.~819--877.

\bibitem[Me]{mes07}
\textsc{G. Mess}, \textit{Lorentz spacetimes of constant curvature}, Geom. Dedicata~126 (2007), p.~3--45.

\bibitem[Mi]{mil77}
\textsc{J. Milnor}, \textit{On fundamental groups of complete affinely flat manifolds}, Adv. Math.~25 (1977), p.~178--187.

\bibitem[Mo]{mou87}
\textsc{G. Moussong}, \textit{Hyperbolic Coxeter groups}, PhD thesis, Ohio State University, 1987.

\bibitem[O]{osa13}
\textsc{D. Osajda}, \textit{A construction of hyperbolic Coxeter groups}, Comment. Math. Helv.~88 (2013), p.~353--367.

\bibitem[PW]{pw18}
\textsc{P. Przytycki, D. T. Wise}, \textit{Mixed $3$-manifolds are virtually special}, J. Amer. Math. Soc.~31 (2018), p.~319--347.

\bibitem[Sag]{sag95}
\textsc{M. Sageev}, \textit{Ends of group pairs and non-positively curved cube complexes}, Proc. London Math. Soc.~71 (1995), p.~585--617.

\bibitem[Sal]{sal00}
\textsc{F. Salein}, \textit{Vari\'et\'es anti-de Sitter de dimension~3 exotiques}, Ann. Inst. Fourier~50 (2000), p.~257--284.

\bibitem[Se]{sel60}
\textsc{A. Selberg}, \textit{On discontinuous groups in higher-dimensional symmetric spaces} (1960), in ``Collected papers'', vol.~1, p.~475--492, Springer-Verlag, Berlin, 1989.

\bibitem[S1]{smi16}
\textsc{I. Smilga}, \textit{Proper affine actions on semisimple Lie algebras}, Ann. Inst. Fourier~66 (2016), p.~785--831.

\bibitem[S2]{smi16bis}
\textsc{I. Smilga}, \textit{Proper affine actions: a sufficient criterion}, arXiv:1612.08942.

\bibitem[T1]{tom90}
\textsc{G. Tomanov}, \textit{The virtual solvability of the fundamental group of a generalized Lorentz space form}, J. Differential Geom.~32 (1990), p.~539--547.

\bibitem[T2]{tom16}
\textsc{G. Tomanov}, \textit{Properly discontinuous group actions on affine homogeneous spaces}, Proc. Steklov Inst. Math.~292 (2016), p.~260--271.

\bibitem[V]{vin71}
\textsc{E. B. Vinberg}, \textit{Discrete linear groups generated by reflections}, Math.\ USSR Izv.~5 (1971), p.~1083--1119.

\bibitem[W1]{wis11}
\textsc{D. T. Wise}, \textit{The structure of groups with a quasiconvex hierarchy}, preprint (2011), see \url{http://www.math.mcgill.ca/wise/papers}.

\bibitem[W2]{wis14}
\textsc{D. T. Wise}, \textit{The cubical route to understanding groups}, in \textit{Proceedings of the International Congress of Mathematicians, Seoul 2014}, volume~II, p.~1075--1099, Kyung Moon, Seoul, Korea, 2014.

\end{thebibliography}
\end{document}